\documentclass[11pt]{article}
\usepackage{mathrsfs}
\usepackage{amsfonts}
\usepackage{cite}
\usepackage[titletoc]{appendix}
\usepackage[leqno]{amsmath}
\usepackage{graphicx}
\usepackage{latexsym}
\usepackage{amsmath, bbm, amsfonts, amssymb, amsthm, mathrsfs, euscript, makeidx, color}
\usepackage{indentfirst}  \usepackage{bm} \usepackage{dsfont}
\usepackage{enumerate}

\def\sqr#1#2{{\vcenter{\vbox{\hrule height.#2pt
    \hbox{\vrule width.#2pt height#1pt \kern#1pt \vrule width.#2pt}
    \hrule height.#2pt}}}}

\def\3n{\negthinspace \negthinspace \negthinspace }
\def\2n{\negthinspace \negthinspace }
\def\1n{\negthinspace }
\def\bel{\begin{equation}\label}
\def\eel{\end{equation}}
\def\dbA{\mathbb{A}}

\def\dbC{\mathbb{C}}

\def\dbE{\mathbb{E}}
\def\dbF{\mathbb{F}}

\def\dbH{\mathbb{H}}

\def\dbK{\mathbb{K}}

\def\dbM{\mathbb{M}}

\def\dbP{\mathbb{P}}

\def\dbR{\mathbb{R}}
\def\dbS{\mathbb{S}}

\def\dbX{\mathbb{X}}
\def\dbY{\mathbb{Y}}
\def\dbZ{\mathbb{Z}}

\def\sA{\mathscr{A}}

\def\sC{\mathscr{C}}
\def\sD{\mathscr{D}}

\def\sJ{\mathscr{J}}

\def\sL{\mathscr{L}}
\def\sM{\mathscr{M}}

\def\sP{\mathscr{P}}
\def\sQ{\mathscr{Q}}
\def\sR{\mathscr{R}}

\def\={\buildrel \triangle \over =}

\def\ds{\displaystyle}

\def\ns{\noalign{\ss}}
\def\a{\alpha}
\def\b{\beta }
\def\g{\gamma}
\def\d{\delta}
\def\e{\varepsilon}

\def\k{\kappa}
\def\l{\lambda}

\def\si{\sigma}

\def\f{\varphi}
\def\th{\theta}
\def\o{\omega}

\def\i{\infty}
\def\G{\Gamma}

\def\Th{\Theta}
\def\L{\Lambda}

\def\O{\Omega}

\def\cA{{\cal A}}

\def\cF{{\cal F}}

\def\cK{{\cal K}}
\def\cL{{\cal L}}

\def\cS{{\cal S}}
\def\cT{{\cal T}}
\def\cU{{\cal U}}

\def\cX{{\cal X}}
\def\cY{{\cal Y}}
\def\cZ{{\cal Z}}

\def\BJ{{\bf J}}

\def\BX{{\bf X}}

\def\Bb{{\bf b}}

\def\Bk{{\bf k}}

\def\Bq{{\bf q}}
\def\Br{{\bf r}}

\def\Bu{{\bf u}}

\def\no{\noindent}

\def\ss{\smallskip}
\def\ms{\medskip}

\def\q{\quad}
\def\qq{\qquad}

\def\lan{\mathop{\langle}}
\def\ran{\mathop{\rangle}}

\def\esssup{\mathop{\rm esssup}}
\def\essinf{\mathop{\rm essinf}}

\def\h{\widehat}
\def\wt{\widetilde}

\def\cd{\cdot}

\def\tr{\hbox{\rm tr$\, $}}

\def\les{\leqslant}
\def\ges{\geqslant}

\def\({\Big (}
\def\){\Big )}
\def\[{\Big[}
\def\]{\Big]}
\def\bde{\begin{definition}\label}
\def\ede{\end{definition}}
\def\be{\begin{equation}}
\def\bel{\begin{equation}\label}
\def\ee{\end{equation}}
\def\bt{\begin{theorem}\label}
\def\et{\end{theorem}}
\def\bc{\begin{corollary}\label}
\def\ec{\end{corollary}}
\def\bl{\begin{lemma}\label}
\def\el{\end{lemma}}
\def\bp{\begin{proposition}\label}
\def\ep{\end{proposition}}
\def\bas{\begin{assumption}}
\def\eas{\end{assumption}}
\def\br{\begin{remark}\label}
\def\er{\end{remark}}
\def\ba{\begin{array}}
\def\ea{\end{array}}
\def\ed{\end{document}}

\def\rf{\eqref}

\def\square#1{\vbox{\hrule\hbox{\vrule height#1%
  \kern#1\vrule}\hrule}}
\def\rectangle#1#2{\vbox{\hrule\hbox{\vrule height#1%
  \kern#2\vrule}\hrule}}

\font\tenbb=msbm10 \font\sevenbb=msbm7 \font\fivebb=msbm5

\newfam\bbfam
\scriptscriptfont\bbfam=\fivebb \textfont\bbfam=\tenbb
\scriptfont\bbfam=\sevenbb

\newtheorem{theorem}{Theorem}[section]
\newtheorem{corollary}[theorem]{Corollary}

\newtheorem{lemma}[theorem]{Lemma}
\newtheorem{proposition}[theorem]{Proposition}

\theoremstyle{definition}
\newtheorem{definition}[theorem]{Definition}
\newtheorem{remark}[theorem]{Remark}

\makeatletter
 
 \@addtoreset{equation}{section}
\makeatother

\begin{document}

\title{Infinite Horizon Linear Quadratic Mean Field {Control and Game} Problems with Common Noise and Regime Switching via Conditional McKean-Vlasov FBSDEs
\footnote{This work is supported by the National Key R\&D Program of China (No. 2023YFA1009002),  the
National Natural Science Foundation of China (No. 12371443).}}

\author{Qingmeng Wei \footnote{School of Mathematics and Statistics, Northeast Normal University, Changchun 130024, P. R. China; {E-mail: weiqm100@nenu.edu.cn}}\qq
Yaqi Xu \footnote{School of Mathematics and Statistics, Northeast Normal University, Changchun 130024, P. R. China; {E-mail: xuyq222@nenu.edu.cn}}}

\maketitle

\begin{abstract}
This paper studies infinite horizon  linear quadratic (LQ)  mean field problems with common noise and regime switching, covering both control and game formulations.  To establish a theoretical foundation for the  LQ   framework, we first analyze fully coupled forward-backward stochastic differential equations (FBSDEs) of conditional McKean-Vlasov  type with Markovian switching and  establish its well-posedness  under a generalized domination-monotonicity condition.   Building upon this solvability result,   we then  derive  necessary and sufficient conditions for both the open-loop optimal control  in the control problem and the  mean-field Nash equilibria in the game problem.
\end{abstract}

\ms

\no\bf Keywords: \rm
infinite horizon  FBSDEs, markovian switching, conditional McKean-Vlasov,  common noise, open-loop optimal control, mean-field Nash equilibria


\ms

\no\bf AMS Mathematics Subject Classification. \rm
93E20, 60H10, 49N10

\section{Introduction}\label{SEC_Int}

  This work investigates infinite horizon linear quadratic mean field control and game problems  that incorporate common noise and Markovian switching. Research in this  area  dates back to the pioneering contributions of Huang et al. \cite{HMC-2006}  and Lasry, Lions \cite{Lasry-Lions-2007}, which established   mean field game theory as a rigorous mathematical framework for analyzing large-scale systems of interacting agents.
     For comprehensive overviews of the theory and applications of mean field methods, we refer to \cite{Gueant-2011, Bensoussan-Frehse-Yam-2013, CD-2018-1, CD-2018-2,  Carmona-Delarue-Lachapelle-2013, Gomes-Saude-2013} and references therein.
It is widely recognized that the theoretical framework of mean field problems
has been developed mainly along two complementary approaches: the analytic
approach and the probabilistic approach.
The analytic approach is based on systems of coupled
Hamilton{-}Jacobi{-}Bellman and Fokker{-}Planck equations, {which characterize,
respectively, the value function at the individual level and the evolution of
the population distribution; see, for
instance,
\cite{Achdou-Cardaliaguet-Delarue-Porretta-2020,
Cardaliaguet-Seeger-Souganidis-2025,  Cardaliaguet-Graber-2015, Cardaliaguet-2015,
Cardaliaguet-Graber-Porretta-Tonon-2015,
Gomes-Saude-2013, Porretta-2015, Gomes-Pimentel-Voskanyan-2016}.
At a more global level, such systems are encoded by the master equation, an
infinite-dimensional PDE on the space of probability measures, which
characterizes the dependence of the value function on the population
distribution; see, among others,
\cite{Bensoussan-Frehse-Yam-2015, Bensoussan-Frehse-Yam-2017,
Cardaliaguet-Delarue-Lasry-Lions-2019,
Chassagneux-Crisan-Delarue-2022,
Gangbo-Meszaros-Mou-Zhang-2022,
Bensoussan-Kim-Yam-2025}.}
Complementary to the analytic approach, probabilistic approaches provide
powerful tools for studying mean field problems. One line of research uses
coupled McKean{-}Vlasov FBSDEs to characterize optimal controls and equilibria
through the interaction between state dynamics and adjoint equations; see, for
instance,
\cite{CD-2013, CD-2015, CD-2018-1, Ahuja-Ren-Yang-2019,
Bayraktar-Zhang-2023}. {Another important line is based on probabilistic weak formulations and
mean-field limit arguments, which are used to study mean field equilibria,
weak limits of $N$-player Nash equilibria, and approximate Nash equilibria;
see, e.g.,
\cite{Carmona-Lacker-2015, Carmona-Delarue-Lacker-2016, Lacker-2016}.
These probabilistic viewpoints are particularly useful for handling complex
stochastic environments, including common noise.}

Early literature on mean-field games  generally assumes that agents are influenced by mutually independent idiosyncratic noises, thereby excluding common noise.
 While this simplification facilitates mathematical formulation, it limits the model's capacity to capture real-world phenomena such as systemic risk and collective dependence.
  Therefore,  recent  research has increasingly focused on mean field problems that incorporate common noise,  aiming to advance the theory and precisely model complex interconnected systems,  see \cite{Graber-2016, CD-2018-2, Ahuja-Ren-Yang-2019, HT-2022}, etc.
{The presence of common noise makes several analytical techniques developed for mean field models without common noise no longer directly applicable, thereby posing new challenges to traditional PDE methods.}
 Owing to  inherent advantages  in handling stochastic dynamics,   the probabilistic method has a distinct  advantage  for  studying     large-scale interactive systems.  Consequently, the related theoretical frameworks are relatively well-developed in the finite-time horizon setting, we refer to  \cite{CD-2013, CD-2015, Graber-2016,  Ahuja-Ren-Yang-2019, HT-2022, 2024-GNY} and the references therein.
Due to their theoretical and practical importance, infinite horizon mean field
problems have attracted increasing attention; see, for instance,
\cite{HLY-2012, Bayraktar-Zhang-2023, Wei-Xu-Yu-2023, Hua-Luo}.
{Markovian switching provides a natural framework for modeling
random structural changes of the underlying environment; see, e.g.,
\cite{Mao-Yuan-2006}. In financial applications, for instance, markets may
switch between bull and bear regimes, or between low- and high-volatility
regimes. In such situations, the drift, volatility, control effectiveness and
cost coefficients are naturally regime-dependent. Hence the switching process
changes the active dynamics and the associated optimization problem, rather
than merely serving as an additional source of noise. Such switching structures
have been widely used in financial and control models involving economic
fluctuations, market regime transitions, option pricing, portfolio management
and risk minimization; see, e.g.,
\cite{Zhou-Yin-2003, Zhang-Yin-2004, Elliott-Chan-Siu-2005,
Elliott-Siu-2010, 2012-DH}. Regime-switching mean field control and game
problems have also been studied via stochastic maximum principles and dynamic
programming; see, for instance,
\cite{Bensoussan-Djehiche-Tembine-Yam-2020,
Nguyen-Nguyen-Yin-2020, Nguyen-Yin-Nguyen-2021, Liu-Ma-2022, 2024-GNY, Mei-Wang-Wei-Yong-2026}.}

Motivated by these developments and by the wide use of LQ models in
finance, economics and engineering, this paper adopts a probabilistic
approach to study infinite horizon linear-quadratic mean field control and
game problems with both common noise and Markovian switching. More precisely,
we aim to characterize optimal controls and equilibria through infinite horizon
coupled FBSDEs of conditional McKean{-}Vlasov type with Markovian switching,
which arise as the corresponding Hamiltonian systems.
{
As will be seen from the LQ analysis, Markovian switching does not merely add
an extra source of noise. Instead, it gives rise to a family of
regime-dependent Hamiltonian systems, in which the coefficients, cost
functionals and adjoint equations depend on the current state of the Markov
chain. Moreover, the transition mechanism of the Markov chain couples the
equations associated with different regimes. Consequently, the optimal controls
and equilibria have to be characterized simultaneously with respect to the
state, the conditional law and the finite regime space.}
This observation naturally leads to the well-posedness problem for such
infinite horizon conditional mean-field FBSDEs with Markovian switching.
Although the main applications considered in this paper are linear systems
with quadratic costs, we first establish the well-posedness theory in a general
nonlinear framework, so that the analysis can be applied to a broader class of
conditional mean-field FBSDEs.
Coupled FBSDEs and Markovian switching have been extensively studied
separately; see, for instance,
\cite{Hu-Peng-1995, Peng-Wu-1999, Peng-Shi-2000, Yu-2022, WY-2021}
for coupled and infinite horizon FBSDEs. However, infinite horizon
conditional mean-field FBSDEs incorporating both common noise and Markovian
switching are still less developed.

To establish the well-posedness, we employ the method of continuity, as
developed in \cite{Hu-Peng-1995, Peng-Wu-1999, Peng-Shi-2000}, and generalize
the domination-monotonicity condition used in
\cite{WY-2021, Yu-2022, Tian-Yu, Hua-Luo} to the present framework.
In the mean-field literature, such conditions are usually formulated in a
probabilistic form; see, for instance,
\cite{Bensoussan-Yam-Zhang-2015, Ahuja-Ren-Yang-2019, Tian-Yu}.
Here we introduce a formulation in terms of state variables and probability
measures, which makes the underlying structural relations more explicit. We
also show that this formulation implies the corresponding probabilistic
monotonicity conditions required in the well-posedness analysis.

  For conditional LQ mean field control problems, we  characterize the open-loop optimal control via    infinite horizon coupled FBSDEs of  conditional McKean-Vlasov  type with Markovian switching. This characterization is established via the variational method under a semi-positive definite condition.   In doing so,
  we   perform a linear transformation, motivated by our previous findings   \cite{WY-2021, Wei-Xu-Yu-2023} concerning the role of cross-term coefficients in the optimal control space,  to eliminate these terms from the cost functional,  thereby enabling a unified analysis.
Within this setting, the domination-monotonicity condition is explicitly specified for the resulting linear  Hamiltonian   system, which in turn ensures its well-posedness.

We proceed to extend our analysis to  game problems.  The characterization of  Nash equilibrium is equivalent to a fixed-point problem, whose solvability is  closely tied to the well-posedness of an associated Hamiltonian system.
Existing well-posedness results, such as those in \cite{Ahuja-Ren-Yang-2019, Graber-2016},
 typically require the removal of all conditional mean-field terms involving both the state and control from the state equation, as well as all conditional mean-field terms of   control  from the cost functional.
Our framework accommodates conditional mean-field   controls in both the state dynamics and cost functional, thereby extending the existing theory.

The rest of the paper is organized as follows.  We begin by introducing some necessary notations in Section 2. Section 3 is devoted to a preliminary analysis of infinite horizon conditional McKean-Vlasov SDEs and BSDEs with Markovian switching. We then proceed, in Section 4, to establish the well-posedness of  infinite horizon coupled  FBSDEs of conditional McKean-Vlasov type with Markovian switching. The theoretical results are subsequently applied in Section 5 to characterize optimal control  and equilibria in conditional mean field control and game problems over an infinite horizon.

\section{Preliminaries}\label{SEC_2}

We  work on the underlying probability space $(\O,\cF, \dbP)$, which is the complete  product space  of $(\Omega^0, \mathcal{F}^0 ,  \mathbb{P}^0)$ and $(\Omega^1, \mathcal{F}^1, \mathbb{P}^1)$ specified as follows.
\begin{itemize}
  \item  $(\Omega^0, \mathcal{F}^0,   \mathbb{P}^0)$ is a complete filtered probability space   endowed with a $d_0$-dimensional standard Brownian motion $ W^0  $ {and a continuous-time Markov chain $\{\alpha_s\}_{s\geqslant 0}$}. {The chain $\{\alpha_s\}_{s\geqslant 0}$ takes values in a finite state space $\textbf{M}=\{1,\ldots,m_0\}$ and is assumed to be independent of $W^0$.}
{Denoting by $\dbF^{0,W}=\{\cF_s^{0,W}\}_{s\ges 0}$ and $\dbF^{0,\alpha}=\{\cF_s^{0,\alpha}\}_{s\ges 0}$ the filtrations generated by $W^0$ and $\alpha$, resp., we introduce $\dbF^0=\{\cF_s^0\}_{s\ges 0}$ with $\cF_s^0:=\cF_s^{0,W}\vee\cF_s^{0,\alpha}$, augmented by all the $\dbP^0$-null sets.}
  \item   $(\Omega^1, \mathcal{F}^1,  \mathbb{P}^1)$ is also a complete filtered probability space, on which a $d$-dimensional standard Brownian motion $W$ is defined. {Moreover, $\mathbb{F}^1=\{\mathcal{F}^1_s\}_{s\ges 0}$ is the natural filtration generated by $W$ and augmented by all the $\dbP^1$-null sets.}
\end{itemize}
Then,  $\Omega:=\Omega^0\times \Omega^1,$ $ \cF:=\cF^0\otimes \cF^1,$ $ \dbP:=\dbP^0\otimes\dbP^1$   with $\cF$ completed with respect
to $\dbP$.
The filtration  $\dbF=\{\cF_s\}_{s\ges 0}$   on   $(\Omega,\cF,\dbP)$ is given with  $\cF_s:=\cF_s^0\otimes\cF_s^1$ and augmented by all $\dbP$-null sets.
In what follows, $ W^0  $ and  $W  $ are referred  to as common noise and individual noise, resp., {while $\alpha$ is referred to as the common regime process}.

\ss

We now specify  the Markov chain $\{\alpha_{s}\}_{s\geqslant 0}$. Under ${\mathbb{P}^0}$,
its generator is given by  $\cA  =\left( a_{ij} \right)_{m_0\times m_0}$, where $a_{ij} $ denotes the  transition intensity from state $i$ to state $j$,  if  $i\neq j$, while $a_{ii}=-\sum\limits_{j\neq i}a_{ij}$ for all $i,j\in\textbf{M} $.
{ For each $i,j\in\mathbf M$ with $i\ne j$, define the counting process
of jumps from state $i$ to state $j$ by
$\ds
N_{ij}(t):=\sum_{0<s\les t}
\mathbf 1_{\{\alpha(s-)=i\}}\mathbf 1_{\{\alpha(s)=j\}},
$ $ t\ges 0 .
$
Since $\alpha$ has generator $\cA$, the {$\mathbb F^{0,\alpha}$}-predictable
intensity of $N_{ij} $ is
$
a_{ij}\mathbf 1_{\{\alpha(s-)=i\}} .
$
Hence its compensator is
$$
A_{ij} (t):=\int_0^t a_{ij}\mathbf 1_{\{\alpha(s-)=i\}}\,{\mathrm d}  s .
$$
Then, by the compensator theorem for counting processes,
$
M_{ij}(t):=N_{ij}(t)-A_{ij}(t),$ $ t\ges0,
$
is a purely discontinuous  {$\mathbb F^{0,\alpha}$-martingale, and hence also an $\mathbb F^0$- and $\mathbb F$-martingale}, see also \cite[Lemma IV.21.12]{2006-RW} and  \cite{2024-GNY}.
Moreover, since
$M_{ij}$ has jumps of size one exactly when $\alpha$ jumps from $i$ to $j$, its
optional and predictable quadratic variations are given by
$
[M_{ij}]_t=N_{ij}(t),
$ $
\langle M_{ij}\rangle_t=A_{ij}(t).
$
In particular, for every $T>0$,
$
\mathbb E\langle M_{ij}\rangle_T
=
\mathbb E A_{ij}(T)
\les a_{ij}T<\infty.
$
Thus
$M_{ij}$ is   locally
square-integrable.
For an $\mathbb F$-predictable process
$\varphi(\cdot)=(\varphi_{ij}(\cdot))_{i,j\in\mathbf M}$, we use the notation
$$
{\int_0^t\f_s\bullet \mathrm dM_s}=\sum_{\substack{i,j\in\mathbf M\\ i\ne j}}\int_0^t\f_{ij}(s)\mathrm d M_{ij}(s). $$
Similarly, for a nonnegative predictable process
$\psi(\cdot)=(\psi_{ij}(\cdot))_{i,j\in\mathbf M}$, we write
$$
\int_0^t \psi_s\bullet {\mathrm d}[M]_s
:=
\sum_{\substack{i,j\in\mathbf M\\ i\ne j}}
\int_0^t \psi_{ij}(s)\,{\mathrm d}[M_{ij}]_s .
$$}

Let $\mathbb{H}$ be a Euclidean space and  $\sP(\dbH)$ be the space of probability measures on $\dbH$.
  For any  $\dbH$-valued  random variable $\xi$ on $(\Omega, \mathcal{F}, \mathbb{P})$,   for each fixed $\omega^0 \in \Omega^0$, we denote  the law of $\xi(\omega^0, \cdot)$ by  $\mathcal{L}^1_{\xi(\omega^0)}=\mathcal{L}_{\xi(\omega^0, \cdot)}$.
Then,  by   \cite[Lemma 2.4]{CD-2018-2},
 $\cL^1_\xi$ defines a random variable from $(\Omega^0, \mathcal{F}^0, \mathbb{P}^0)$ into $(\mathscr{P}(\mathbb{H}), \mathcal{B}(\mathscr{P}(\mathbb{H})))$, and provides  a conditional law of $\xi$ given $\cF^0$.
{Under the above construction, this means that the conditional law is taken given both the common Brownian noise $W^0$ and the common regime process $\alpha$, while only the idiosyncratic randomness on $\Omega^1$ is averaged out.}

Next, we recall  the space of probability measures with finite second moments
$$\sP_p(\dbH)=\Big\{\mu \in \sP(\dbH) \bigm|
\int_{\dbH}|x|^p\mu(\mathrm dx)<\i\Big\},$$
which is complete  under  the following {\it Wasserstein-$p$ metric} (or simply {\it $\mathbf{w}_p$-metric}), for $\mu_1, \mu_2\in\sP_p(\dbH)$,
$$\ba{ll}
\ns\ds \mathbf{w}_p(\mu_1, \mu_2)=\inf\bigg\{\(\int_{\dbH\times\dbH}|x_1-x_2|^p \mu (\mathrm dx_1, \mathrm dx_2)\)^{1\over p}
\biggm|\2n\ba{ll} \ns\mu\in\sP_p(\dbH\times\dbH), \ \mu(\cd, \dbH)=\mu_1(\cd),\\
\ns
 \mu(\dbH,\cd)=\mu_2(\cd)\ea\bigg\}. \ea$$
For $i=0,1$,  $\mathbb{E}^i[\cd]$  denotes   the expectation taken over
   $\omega^i \in\Omega^i$, and  $\delta_0$ represents the Dirac measure at the origin.
For $\xi_1, \xi_2 \in L^2_{\mathcal{F}}(\Omega, \mathbb{H})$    with conditional distributions $\mu_1$ and $\mu_2$ given $\cF^0$, resp.,    the following inequality  holds,
\bel{E1-inequ}
 \left| \mathbb{E}^1[\xi_1] - \mathbb{E}^1[\xi_2] \right| \les \mathbf{w}_2(\mu_1, \mu_2)
\les \left( \mathbb{E}^1[|\xi_1 - \xi_2|^2] \right)^{1/2}.
\ee
%

%

For   some constant $\kappa \in \mathbb{R}$, $p\geqslant 1$, for any $0 \leqslant t <T< \infty$, we  introduce  several spaces of random variables and stochastic processes as follows.
\begin{itemize}
%
%

\item $\!\!L_{\mathcal{F}_{t}}^{p}(\Omega ; \mathbb{H}):=\Big\{\xi: \Omega\to\mathbb{H} \mid  \xi$ is $\mathcal{F}_{t}$-measurable, and $\mathbb{E}[ |\xi|^{p}]<\infty \Big\}$;
\item $\!\!\sD^p:=\{  (t,x_t,\iota)\mid t\in  [0, \infty),\ x_t\in  L_{\mathcal{F}_t}^{p}(\Omega;\mathbb{R}^n),\ \iota\in \textbf{M} \}$;
\item $\!\!\cS^p_\dbF(t,T;\mathbb{H}):=\Big\{ \f: \Omega\times [t,T]\to\mathbb{H}\mid \f$ is $\dbF$-adapted with continuous paths, $\dbP$-a.s.,\\
    \hspace*{10.7cm}
     and $\mathbb{E}\[\sup\limits_{s\in[t,T]}|\f_s|^p\]<\infty \Big\}$;
\item  {
$\!\!\mathcal D_{\mathbb F}^{p,\kappa}(t,T;\mathbb{H}):=\Big\{ \f: \Omega\times [t,T]\to\mathbb{H}\mid \f$ is  $\mathbb F$-adapted with c\`adl\`ag paths,\\
    \hspace*{10.2cm}   and $\ds
\mathbb E\[
\sup_{s\in[t,T]} {\mathrm e^{\kappa s}}|\varphi_s|^p
\]<\infty
 \Big\};$ }

%
%
%
%
%
%
%
\item $\!\!L_{\mathbb{F}^{0}}^{\infty}(t, \infty ; \mathbb{H}):=
\Big\{\varphi: \Omega^{0}\times[t, \infty)\times\textbf{M}\to\mathbb{H}\mid \varphi   $ is $\mathbb{F}^{0}$-progressively measurable,\\
\hspace*{8.1cm}  and
$\esssup\limits_{(s, \omega^{0}, \imath) \in[t, \infty) \times \Omega^{0}\times \textbf{M}}\|\varphi(s, \omega^{0}, \imath)\|<\infty \Big\}$;
\item $\!\!L_{\mathbb{F}^0}^{p, \kappa}(t, \infty; \mathbb{H})\!:=\!
\Big\{ \varphi: \Omega^0\times[t, \infty)\to\mathbb{H}\mid \varphi$ is $\mathbb{F}^0$-progressively measurable,\\
\hspace*{9.9cm}and
$\mathbb{E}^0\[ \displaystyle\int_{t}^{\infty} \mathrm e^{\kappa s}|\varphi_s|^{p}  \mathrm ds\]  < \infty \Big\}$;
\item $\!\!L_{\mathbb{F}}^{p, \kappa}(t, \infty; \mathbb{H})\!:=\!
\Big\{ \varphi: \Omega\times[t, \infty)\to\mathbb{H}\mid \varphi$ is $\mathbb{F}$-progressively measurable,\\
\hspace*{10cm}and
$\mathbb{E}\[ \displaystyle\int_{t}^{\infty} \mathrm e^{\kappa s}|\varphi_s|^{p}  \mathrm ds\] < \infty \Big\}$;
%
%
%
%
%
%
%
%
\item $\!\!\mathcal{M}_{\mathbb{F}}^{p, \kappa}(t, \infty; \mathbb{H}):=
\Big\{\phi=(\phi_{kl})^{m_0}_{i, j=1}\mid\phi_{i j}: \Omega\times[t, \infty)\to\mathbb{H}$ is $\mathbb{F}$-progressively measurable, \\
\hspace*{6.5cm} and $\ds\sum_{\substack{i,j\in\mathbb M\\ i\ne j}}
\mathbb E\[
\int_t^\infty {\mathrm e^{\kappa s}}|\phi_{ij}(s)|^p
a_{ij}\mathbf 1_{\{\alpha_{s-}=i\}}\,\mathrm ds
\]<\infty  \Big\}$.
%
%
%
%
%
%
%
%
\end{itemize}
%
%

%
%
%
%
%
%

To simplify the notations, we  introduce the following abbreviations,
$$
\begin{aligned}
&\mathscr{R}:= \mathbb{R}^{n} \times \mathbb{R}^{n  d} \times \mathbb{R}^{n  d_0} , \\
&\cL_{\mathbb{F}}^{2, \kappa}(t, \infty):=  L_{\mathbb{F}}^{2, \kappa}(t, \infty;\mathbb{R}^n)
\times L_{\mathbb{F}}^{2, \kappa}(t, \infty;\mathbb{R}^{n  d})
\times L_{\mathbb{F}}^{2, \kappa}(t, \infty;\mathbb{R}^{n  d_0})
\times\mathcal{M}_{\mathbb{F}}^{2, \kappa}(t, \infty; \mathbb{R}^n) .
\end{aligned}
$$
The inner product $\langle \cdot,\cdot\rangle$ and the norm $\left|\cdot\right|$ of $\mathscr{R}$ are introduced as follows,
for any $\theta =(y^{\top}, z^{\top},$ $  (z^0)^\top)$,
$\bar{\theta} =(\bar{y}^{\top}, \bar{z}^{\top}, (\bar{ z }^0)^{\top})^{\top}\in\mathscr{R}$,
$$\ba{ll}
\ns\ds \langle \theta, \bar{\theta}\rangle := \langle y, \bar{y}\rangle
+\langle z, \bar{z}\rangle+\langle  z^0  , \bar{ z}^0  \rangle,
\q |\theta|:=\sqrt{\langle \theta, \theta\rangle}.
\ea $$

\section{Infinite horizon conditional McKean-Vlasov   FBSDEs with markovian switching---decoupled case}

In this section, we study infinite horizon conditional McKean-Vlasov FBSDEs  with Markovian switching in the decoupled setting.    Let us be given the following mappings
$$
\begin{aligned}
&b:\Omega\times[0, \infty) \times\mathbb{R}^n\times \mathscr{P}(\mathbb{R}^n)\times \textbf{M}\to\mathbb{R}^n, \\
& \sigma:
\Omega\times[0, \infty) \times\mathbb{R}^n\times \mathscr{P} (\mathbb{R}^n)\times \textbf{M}\to\mathbb{R}^{n d}, \\
&\widetilde{\sigma}   :
\Omega\times[0, \infty) \times\mathbb{R}^n\times \mathscr{P} (\mathbb{R}^n)\times \textbf{M}\to\mathbb{R}^{n   d_0},\\
&  f:\Omega \times [0, \infty) \times \mathbb{R}^{n}\times\mathscr{R}\times \mathscr{P} (\mathbb{R}^{n}\times\mathscr{R})\times \textbf{M} \to\mathbb{R}^n.
\end{aligned}
$$
For any initial condition $(t,x_t,\iota)\in \sD^2 $,  we first consider the following forward SDE \begin{equation}\label{equ-SDE}
\left\{
\begin{aligned}
&\!   \mathrm dX_s =b(s, X_s, \cL^1_{X_s},\alpha_s) \, \mathrm ds
+\sigma(s, X_s, \mathcal{L}^1_{X_s},\alpha_s) \, \mathrm dW_s
+\widetilde{\sigma}  (s, X_s, \mathcal{L}^1_{X_s},\alpha_s) \, \mathrm d    W^0 _s,\ s\ges t,\\
&\! X_t = x_t,\quad
\alpha_t=\iota,
\end{aligned}
\right.
\end{equation}
and then    the following BSDE
\begin{equation}\label{equ-BSDE}
\begin{aligned}
&Y_s= Y_T+\int_{s}^{T}f(r, X_r, Y_r, Z_r, Z^0_r,\mathcal{L}^1_{(X_r,Y_r, Z_r, Z^0_r)},\alpha_r) \, \mathrm dr
- \int_{s}^{T} Z_{r} \, \mathrm dW_r\\
&\qquad -\int_{s}^{T}Z^0_{r} \, \mathrm d  W^0 _r
- {\int_{s}^{T}K_r \bullet\mathrm d M_r},  \q \text { for all }  t\les s \les T<\infty.
\end{aligned}
\end{equation}

 {
\begin{remark}\sl
We clarify that the regime switching in this paper is used to describe a
random environment that modulates the coefficients. Hence the forward state
process $X$ remains continuous, while its drift and diffusion coefficients
depend on the current regime $\alpha_s$. This is different from a
jump-diffusion model, where the state process itself contains jump amplitudes.

The Markov chain enters the backward equation in a different way. Since the
underlying filtration contains the Markov chain, the martingale representation
of the backward component involves the compensated martingales associated with
$\alpha$, which gives rise to the term
$
\int_s^T K_r\bullet{\mathrm d}M_r.
$
Thus $K$ describes the sensitivity of the backward component to regime changes.
In particular, although the forward state process has no explicit jump term,
the backward component is in general c\`adl\`ag rather than continuous. At a
jump time of the Markov chain, one   has
$
\Delta Y_s=K_s\bullet\Delta M_s.
$
This switching structure also affects the estimates of BSDEs. In contrast to
the Brownian part, which leads to integrals with respect to $\mathrm ds$, the
Markov-chain martingale part naturally involves the random optional quadratic
variations $[M_{ij}]_s=N_{ij}(s)$. Consequently, the estimates contain terms
such as
$
\mathbb E\int_t^T {\mathrm e^{\kappa s}}|K_s|^2\bullet{\mathrm d}[M]_s,
$
and, on the infinite horizon,
$
\mathbb E\int_t^\infty {\mathrm e^{\kappa s}}|K_s|^2\bullet{\mathrm d}[M]_s.
$
This is one of the main technical differences from the standard Brownian FBSDEs
framework.
\end{remark}}

To simplify the notations, we set
 $\Gamma:=(  b^{\top}, \sigma^{\top}, \widetilde{\sigma}  ^{\top})^{\top}$,
$\Theta :=( Y^{\top} , Z^{\top} ,(Z^0)^{\top} )^{\top}$ and define the following  operators
\bel{cL}\ba{ll}
\ns\ds[\sL V](s,x,\mu,\imath) :=\frac12\tr\big[(\si\si^\top+\widetilde{\si}\widetilde{\si}^\top)(s,x,\mu,\imath)\nabla^2 V (x,\imath)\big]  +b(s,x,\mu,\imath)\nabla  V (x,\imath), \\
\ns\ds
[\sA  V ](x,\imath):=\sum\limits_{j\in\textbf{M}}a_{\imath j}(x) V (x,j)=\sum\limits_{j\in\textbf{M},\ j\neq \imath}a_{\imath j}(x)\big( V (x,j)- V (x,\imath)\big),\  x\in\dbR^n, \ \imath \in \textbf{M},
\ea\ee
 for   $(s,x,\mu,\imath) \in [0,\i)\times \dbR^n\times \sP(\dbR^n)\times\bf{M}$ and    $ V (\cd,\imath)\in C^2(\dbR^n)$.

\subsection{Infinite horizon conditional McKean-Vlasov SDEs with markovian switching}

Let $p\ges2$,   we now  make the following  assumptions on    the  coefficient $\G$ of \eqref{equ-SDE}.

  \begin{description}
    \item[(A1)]   (i) For any $(x,  \imath) \in \mathbb{R}^{n} \times\textbf{M} $,
$ \mu \in \mathscr{P}  (\dbR^n)  $, $\G(\cdot,x, \mu,\imath)$
is  $\mathbb{F}$-progressively measurable;

  \   (ii) There exist  nonnegative constants $ l_{\mathbf{h}x},  l_{\mathbf{h}\mu}$ such that, for all $(\o,s, \imath) \in\Omega\times[0, \infty) \times \textbf{M}$, and  $(x,\mu)$, $(\bar{x},\bar{\mu}) \in \mathbb{R}^{n} \times \mathscr{P} _p (\mathbb{R}^n)$,
$$
|\mathbf{h}(s, x, \mu, \imath)-\mathbf{h}(s, \bar{x}, \bar{\mu}, \imath)| \les   l_{\mathbf{h}x}|x-\bar{x}|+ l_{\mathbf{h}\mu}\mathbf{w}_2(\mu, \bar{\mu}),
$$
   where $\mathbf{h}=b,\si,\widetilde{\si}$, resp.

    \item[(A2)$_p$]
   For all $\imath \in\textbf{M}$,
$ \G(\cdot, 0, \delta_0, \imath) \in L_{\mathbb{F}}^{p,  \kappa^{\ast}}(0, \infty ; \mathscr{R})$ for some $ \kappa^{\ast} \in\mathbb{R}$.

    \item[(A3)]
There exist     a  function  $ V :\dbR^n\times\textbf{M}\to\dbR^+$ being twice continuously differentiable in  $x\in\dbR^n$, a function $\f:\dbR^n\to\dbR^+$ satisfying  $\f(x)\les  V (x,\imath)$, a function $\l_3(\cd)\in L_{\mathbb{F}}^{1,\kappa^\ast}(0,\i;\dbR^+)$  and constants $\l_1\in\dbR$, $\l_2\in\dbR^+$   such that,
$$
[\sL  V ]( s,x, \mu,\imath)+[\sA  V ](x, \imath)\les \l_1 V (x,\imath)+\l_2\int_{\dbR^n}\f(x)\mu(\mathrm dx)+\l_3(s)
$$
holds  for  all $(s,x,\imath)\in [0,\i)\times\dbR^n\times \textbf{M} $ and $\mu\in \sP_{ V }(\dbR^n) :=\big\{\mu\in\sP(\dbR^n)\mid\int_{\dbR^n} V (x,\imath)\mu(\mathrm dx)<\i, \ \forall \imath \in \textbf{M}\big\}$.
  \end{description}

\no { We note that
  Assumption {\bf(A3)} is a Lyapunov-type condition for the forward equation;
see, for instance, \cite{Liu-Ma-2022}. The function $V$ is introduced to
control the growth of the state process on the infinite horizon and to derive
the required moment estimates. Assumption {\bf(A4)$_p$} below provides  a
concrete dissipativity condition under which {\bf(A3)} can be verified by
choosing
$
V(x,{\imath})=|x|^p,$ $ p\ges2.
$
This then yields the desired $L^p$-estimate for $X$. }

%
%
%
\ss

The following is the result about the existence and uniqueness of the solution of   \eqref{equ-SDE}.

\begin{proposition}\label{Pro-SDE-solu}\sl  Let $p\ges 2$  and   {\bf(A1)},  {\bf(A2)$_p$}   hold.

(i) For any
   $(t, x_t,\iota)\in \sD^p $  and   $T\ges t$,  the conditional McKean-Vlasov SDE \eqref{equ-SDE} admits a unique  solution $X \in \cS^p_\dbF(t,T;\dbR^n)$.

   (ii)  If in addition  {\bf(A3)} is assumed, then  for any $ s\ges t$ and $\k\les \k^\ast$,  the following estimates hold,
$$\left\{\2n\ba{ll}
\ns\ds  -( \kappa+\l_1+\l_2)\dbE\[ \int_t^s \mathrm e^{\kappa r }    V (X_r,\a_r) \mathrm dr\]   \les \dbE\[\mathrm e^{\kappa t }   V (  x_t,\iota )+ \int_t^s {\mathrm e^{\kappa  r } }  \l_3(r) \mathrm dr\],\\
 \ns\ds
 \dbE[\mathrm e^{\kappa s }  V (  X_s,\a_s)]\les     \mathrm e^ {(\kappa+\l_1+\l_2) (s-t)}  \dbE\[\mathrm e^{\kappa t }   V (  x_t,\iota) + \int_t^s {\mathrm e^{\kappa  r } }   \l_3(r) \mathrm dr \].
\ea\right.$$
\end{proposition}

\begin{proof}
(i)
 { For any $T\ges t$, let
$\xi\in\mathcal S_{\mathbb F}^p(t,T;\mathbb R^n)$ be given. Consider the following
equation
\bel{X-xi}
\begin{aligned}
X_s^\xi
&=
x_t+\int_t^s b(r,X_r^\xi,\mathcal L^1_{\xi_r},\alpha_r)\,{\mathrm d}r
+\int_t^s \sigma(r,X_r^\xi,\mathcal L^1_{\xi_r},\alpha_r)\,{\mathrm d}W_r  \\
&\quad
+\int_t^s \widetilde\sigma(r,X_r^\xi,\mathcal L^1_{\xi_r},\alpha_r)
\,{\mathrm d}W_r^0,\qquad s\in[t,T].
\end{aligned}
\ee
For fixed $\xi$, the conditional law flow $\mathcal L^1_{\xi_\cdot}$ is frozen.
Thus \eqref{X-xi} becomes an It\^o-type SDE with Markovian switching and
progressively measurable random coefficients, which can be regarded as a special case of the forward component in the Markovian-switching FBSDE framework
of \cite{2024-GNY}. Under {\bf(A1)} and {\bf(A2)$_p$}, it admits a unique
solution $X^\xi\in\mathcal S_{\mathbb F}^p(t,T;\mathbb R^n)$.
Hence we can define a map
$
\Psi:\cS^p_\dbF(t,T;\mathbb R^n)\longrightarrow
\cS^p_\dbF(t,T;\mathbb R^n)
$
by
$
\Psi(\xi):=X^\xi.
$
 Next, we show that $\Psi$ admits a unique fixed point.

Let $\xi,\eta\in\mathcal S_{\mathbb F}^p(t,T;\mathbb R^n)$.
 Since the regime space $\mathbf M$ is finite and
the assumptions are uniform in $\imath\in\mathbf M$,   the standard arguments for SDEs with Markovian switching apply; see, e.g., \cite{Mao-Yuan-2006, 2024-GNY}.
Applying It\^o's formula to $|X_s^\xi-X_s^\eta|^p$ and using the
Gronwall and BDG inequalities, we obtain, for $s\in[t,T]$,
$$
\begin{aligned}
&\mathbb E\Big[\sup_{t\les r\les s}|X_r^\xi-X_r^\eta|^p\Big]
 \les
C_T\mathbb E\int_t^s
\Big[ |b(r,X_r^\xi,\mathcal L^1_{\xi_r},\alpha_r)-b(r,X_r^\xi,\mathcal L^1_{\eta_r},\alpha_r)|^p\\
&\q   +|\si(r,X_r^\xi,\mathcal L^1_{\xi_r},\alpha_r)-\si(r,X_r^\xi,\mathcal L^1_{\eta_r},\alpha_r)|^p +|\wt \si(r,X_r^\xi,\mathcal L^1_{\xi_r},\alpha_r)-\wt\si(r,X_r^\xi,\mathcal L^1_{\eta_r},\alpha_r)|^p
\Big]\,{\mathrm d}r \\
&\les
C_T \dbE\int_t^s
\mathbf{w}_2^p(\mathcal L^1_{\xi_r},\mathcal L^1_{\eta_r}) {\mathrm d}r  \les
C_T\int_t^s\mathbb E|\xi_r-\eta_r|^p {\mathrm d}r ,
\end{aligned}
 $$
where $C_T$ is independent of $\xi$ and $\eta$.
Consequently, for any sufficiently small $\delta>0$,
$$
\|\Psi(\xi)-\Psi(\eta)\|_{\mathcal S_{\mathbb F}^p(t,t+\delta)}^p
\les
C_T\delta
\|\xi-\eta\|_{\mathcal S_{\mathbb F}^p(t,t+\delta)}^p .
$$
Choosing $\delta>0$ such that $C_T\delta<1$, we see that $\Psi$ is a
contraction on
$\mathcal S_{\mathbb F}^p(t,t+\delta;\mathbb R^n)$. Therefore,   $\Psi$ has a unique fixed point on $[t,t+\delta]$.
Repeating the same argument on finitely many deterministic subintervals of
$[t,T]$, we obtain a unique fixed point
$X\in\mathcal S_{\mathbb F}^p(t,T;\mathbb R^n)$ of $\Psi$. By the definition
of $\Psi$, this fixed point is  the unique solution of  the conditional McKean{-}Vlasov SDE \eqref{equ-SDE} in $\mathcal S_{\mathbb F}^p(t,T;\mathbb R^n)$.}

  \ms

(ii)
Applying the generalized  It\^o's formula to $\mathrm e^{\kappa s}V (  X_s,\a_s)  $ and using {\bf(A3)}, we  get
$$\ba{ll}
\ns\ds  \dbE[\mathrm e^{\kappa s }  V (  X_s,\a_s)  ] \\
\ns\ds  =\dbE[\mathrm e^{\kappa t }   V (  X_t,\a_t)]+\dbE \int_t^s \mathrm e^{\kappa r }\big(\kappa  V (X_r,\a_r)+  [\sL  V  ](r,   X_r,\cL^1_{X_r},\a_r )+  [\sA  V  ](   X_r,\a_r )\big )\mathrm dr  \\
\ns\ds  \les \dbE[\mathrm e^{\kappa t }   V (  X_t,\a_t)]+\dbE \int_t^s\mathrm e^{\kappa r }  \(( \kappa+\l_1) V (X_r,\a_r)+ \l_2\int_{\dbR^n}\f(x)\cL_{X_r}^1(\mathrm dx) +\l_3(r)\) \mathrm dr  \\
\ns\ds  \les \dbE[\mathrm e^{\kappa t }   V (  X_t,\a_t)]+\dbE \int_t^s\mathrm e^{\kappa r }   \big(( \kappa+\l_1+\l_2) V (X_r,\a_r) +\l_3(r)\big) \mathrm dr ,\q s\ges t .
\ea$$
The desired estimates therefore follow.
\end{proof}

%
%
%
%
%
%
%
%
From the preceding proof, we observe that the role of assumption \textbf{(A3)}   is to ensure the  following inequality
\bel{key}\dbE\[  [\sL  V ] (r, X_r, \cL^1_{X_r},\a_r)+ [\sA  V ] ( X_r,\a_r)\]\les\dbE\[( \l_1+\l_2)  V (X_r,\a_r) +\l_3(r)\].\ee
With this in mind, we may specify a suitable Lyapunov function in    \textbf{(A3)}   to  derive an
 $L^p$-estimate ($p\ges2$) for  the solution of  \eqref{equ-SDE} via Proposition \ref{Pro-SDE-solu} (ii).
To this end,  we impose  the following   condition  on the coefficient  $\G$.

\begin{description}
 \item[(A4)$_p$]
There exist  two constants  $\k_x,\k_{x\mu} \in\dbR$ such that,    for all $\imath\in\textbf{M}$ and $ s\ges 0$, $(x,\mu), (\bar x,\bar\mu)\in \dbR^n\times  \sP_p (\dbR^n)$,
$$\ba{ll}
\ns\ds
2\lan b(s,x,\mu,\imath)-b(s,\bar x,\bar\mu,\imath),x-\bar x \ran+(p-1)|\si(s,x,\mu,\imath)-\si(s,\bar x,\bar\mu,\imath)|^2 \\
\ns\ds +(p-1)|\widetilde{\si}(s,x,\mu,\imath)-\widetilde{\si}(s,\bar x,\bar\mu,\imath)|^2
\les \k_x|x-\bar x|^2+\k_{x\mu} |\mathbf{w}_{2}(\mu,\bar \mu)|^{2}  . \ea$$

  \end{description}

%

 \begin{remark}\label{Re-con-SDE}\sl Using  {\bf{(A1)}}-(ii)  and   {\bf{(A4)}}$_p$, we  carry out the following preliminary computation. For any $\e>0$ and  $(s,x,\mu,\imath)\in [0,\i)\times\dbR^n\times \sP_2(\dbR^n)\times \textbf{M} $,
 $$\ba{ll}
 \ns\ds  2\lan b(s,x,\mu,\imath),x \ran+ (p-1) |\si(s,x,\mu,\imath)|^2  +(p-1)  |\widetilde{\sigma}(s,x,\mu,\imath)|^2\\
 \ns\ds \les (\k_x+3\e)|x |^2+(\k_{x\mu}  +2\e)|\mathbf{w}_{2}(\mu,\d_0)|^{2}   +   C_{p,\e}| \G(s,0,\d_0,\imath)|^2 ,
 \ea$$
 where $C_{p,\e}$ is a nonnegative constant depending on $p$ and $\e$.
 Then,    taking  $ V (x,\imath)=|x|^p$,  we get
$$\ba{ll}
\ns\ds [\sL  V ](s,x,\mu,\imath)
 =  p|x|^{p-2}\lan b(s,x,\mu,\imath),x\ran + \frac p2 |x|^{p-2}\big(|\si(s,x,\mu,\imath)|^2+ |\widetilde{\sigma}(s,x,\mu,\imath)|^2\big)\\
 \ns\ds\hskip 3cm  +p\(\frac p2-1\) |x|^{p-4}\big(|\si(s,x,\mu,\imath)^\top x|^2  + |\widetilde{\sigma}(s,x,\mu,\imath)^\top x|^2 \big) \\
\ns\ds \les \frac {p} 2 |x|^{p-2}\( 2\lan b(s,x,\mu,\imath),x \ran +(p-1) |\si(s,x,\mu,\imath)|^2  +(p-1)  |\widetilde{\sigma}(s,x,\mu,\imath)|^2\)  \\
\ns\ds\les \frac {p} 2 |x|^{p-2}   \((\k_x+3\e)|x |^2+(\k_{x\mu} +2\e)|\mathbf{w}_{2}(\mu,\d_0)|^{2}  +    C_{p,\e}| \G(s,0,\d_0,\imath)|^2   \)   .
\ea$$
 \end{remark}

In what follows, we abbreviate
$$
{
\overline{\k}_p:=-\frac p2
\big(\k_x+\k_{x\mu}\mathbf{1}_{\{\k_{x\mu}>0\}}\big),
\qquad
\overline{\k}\equiv\overline{\k}_2
=-\k_x-\k_{x\mu}\mathbf{1}_{\{\k_{x\mu}>0\}}.}
$$

 \begin{corollary}\label{Cor-SDE-1}
 \sl Let $p\ges 2$ and  assume   {\bf{(A1)}, \bf{(A2)$_p$}} and  {\bf{(A4)}}$_p$ hold.  For
$(t, x_t,\iota )\in \sD^p$, the solution $X $ of \eqref{equ-SDE} belongs to $L^{p, \kappa}_\dbF(t,\i;\dbR^n) $ {whenever $\kappa< \overline{\k}_p$ and $\kappa\les\kappa^\ast$}. Moreover,  the following estimates hold,
  \bel{L-p-SDE-1} \ba{ll}
 \ns\ds \!\!\! {\rm(i)} \  \dbE[\mathrm e^{\kappa s }  |X_s|^p]\les    C_{p }\mathrm e^ {( \kappa-\overline \k_p +3\e ) (s-t)}  \dbE\[\mathrm e^{\kappa t } |x_t|^p+ \int_t^s{\mathrm e^{\kappa  r } } | \G(r,0,\d_0,\alpha_r)|^p \mathrm dr\],\q s\ges t,\\
\ns\ds  \!\!\! {\rm(ii)} \  \dbE \[\int_t^\i \mathrm e^{\kappa r }   |X_r|^p\mathrm dr\]  \les  C_p\dbE\[\mathrm e^{\kappa t }  |x_t|^p+ \int_t^\i {\mathrm e^{\kappa  r } }| \G(r,0,\d_0,\alpha_r)|^p \mathrm dr\].\\
\ea \ee
Furthermore, let  $\bar X $ denote  by the solution  of   \eqref{equ-SDE} associated  with another  initial state $\bar x_t\in L_{\cF_t}^p(\O;\dbR^n)$ and  coefficient  $\bar \G$ satisfying   {\bf{(A1)}, \bf{(A2)$_p$}} and  {\bf{(A4)}}$_p$. Then we have
 \begin{equation}\label{SDE-Cont}
\begin{aligned}
& \mathbb{E} \[\int_{t}^{\infty} \mathrm e^{ \kappa r}|X _r-\bar  X _ r|^{p} \mathrm dr\]
 \les   C_p \mathbb{E}\[ \mathrm e^{\kappa
 t} |x_t-\bar x_{t}  |^{p}+\int_t^\infty {\mathrm e^{\kappa  r } }|\G (r,\bar X_r,\cL^1_{\bar X_r},\a_r)-\bar\G (r,\bar X_r,\cL^1_{\bar X_r},\a_r)|^p \mathrm dr\] .
\end{aligned}
\end{equation}

 \end{corollary}

\begin{proof} Taking $ V (x,\imath)=|x|^p$,
 from Remark \ref{Re-con-SDE}, for any $\e>0$, we get
$$\ba{ll}
\ns\ds \dbE\big[[\sL  V ] (r, X_r, \cL^1_{X_r},\a_r)\big]\\
\ns\ds \les \dbE\[ \frac {p} 2 | X_r|^{p-2}   \big((\k_x+3\e)|X_r |^2+(\k_{x\mu} +2\e)|\mathbf{w}_{2}(\cL^1_{X_r},\d_0)|^{2}  +    C_{p,\e}| \G(r,0,\d_0,\a_r)|^2  \big)  \]\\
\ns\ds \les \frac {p(\k_x+3\e)} 2  \dbE  \big[| X_r|^{p} \big]+  \frac {p(\k_{x\mu} \textbf{1}_{\{\k_{x\mu} >0\}}+2\e)} 2
\dbE\[ | X_r|^{p-2}  \cd \dbE^1[|X_r |^{2} ]\] +C_{p,\e}\dbE \big[  | X_r|^{p-2}   | \G(r,0,\d_0,\a_r)|^2 \big] \\
\ns\ds \les \frac {p(\k_x+4\e)} 2  \dbE  \big[| X_r|^{p} \big]+  \frac {p(\k_{x\mu} \textbf{1}_{\{\k_{x\mu} >0\}}+2\e)} 2
\dbE^0\[ \dbE^1[| X_r|^{p-2} ] \cd \dbE^1[|X_r |^{2} ]\] +C_{p,\e} \dbE\big[| \G(r,0,\d_0,\a_r)|^p\big] \\
\ns\ds \les \frac {p(\k_x+4\e)} 2  \dbE  \big[| X_r|^{p} \big]+  \frac {p(\k_{x\mu} \textbf{1}_{\{\k_{x\mu} >0\}}+2\e)} 2
\(\dbE^0\[ (\dbE^1[| X_r|^{p-2} ])^{\frac {p}{p-2}}\]\)^\frac{p-2}{p} \cd \(\dbE^0\[ (\dbE^1[|X_r |^{2} ])^{\frac{p}{2}}\]\)^\frac 2p\\
\ns\q +C_{p,\e} \dbE\big[| \G(r,0,\d_0,\a_r)|^p\big]  \\
\ns\ds \les \frac {p(\k_x+4\e)} 2  \dbE  \big[| X_r|^{p} \big]+  \frac {p(\k_{x\mu} \textbf{1}_{\{\k_{x\mu}>0\}}+2\e)} 2
\(\dbE^0\[  \dbE^1[| X_r|^{p} ] \]\)^\frac{p-2}{p}\!\!\! \cd \! \(\dbE^0\[ \dbE^1[|X_r |^{p} ] \]\)^\frac 2p\\
\ns\q + C_{p,\e} \dbE\big[| \G(r,0,\d_0,\a_r)|^p\big]\\
%
\ns\ds \leq (-\overline \k_p+ 3\e  )   \dbE \big[| X_r|^{p} \big]
+C_{p,\e} \dbE\big[| \G(r,0,\d_0,\a_r)|^p\big],
\ea$$
  which    corresponds to \eqref{key}. When {$\kappa< \overline{\k}_p$ and $\kappa\les\kappa^\ast$}, \eqref{L-p-SDE-1} follows from  Proposition \ref{Pro-SDE-solu} (ii). The stability estimate \eqref{SDE-Cont} can be proved using analogous arguments.
\end{proof}

\subsection{Infinite horizon conditional McKean-Vlasov    BSDEs with markovian switching}
This section focuses on the backward equation \eqref{equ-BSDE}.
The driver $  f  $ is assumed to satisfy the following conditions.
 \begin{description}
   \item[(B1)]    (i)  For any $(x, \theta, \imath) \in \mathbb{R}^{n}\times\sR \times\textbf{M} $
and $ \nu \in \mathscr{P} (\mathbb{R}^{n}\times\sR)  $, $f(\cdot,x,\theta, \nu,\imath)$
is  $\mathbb{F}$-progressively measurable;

(ii)   For all $\imath \in\textbf{M}$,
 $f(\cdot, \textbf{0}, \delta_\textbf{0}, \imath)\in L_{\mathbb{F}}^{2,  \kappa^{\ast}}(0, \infty ; \dbR^n  )$ for some
 $\kappa^{\ast} \in\mathbb{R}$;

(iii) There exist   nonnegative constants $ l_{f x},  l_{f y}, l_{f z}, l_{f  z^0}, l_{f \nu}$ such that for all $(\o,s, \imath) \in\Omega\times[0, \infty) \times \textbf{M}$, and $(x,\theta,\nu)$, $(\bar{x},\bar{\theta},\bar{\nu}) \in \mathbb{R}^{n} \times\sR \times \mathscr{P}_{2}(\mathbb{R}^n\times\sR)$,
$$\ba{ll}
\ns\ds |f(s, x,\theta, \nu, \imath)-f(s, \bar{x},\bar{\theta}, \bar{\nu}, \imath)| \\
\ns\ds \les   l_{fx}|x-\bar{x}|+ l_{fy}|y-\bar{y}|+ l_{fz}|z-\bar{z}|
+ l_{f z^0}| z^0-\bar{z}^0|+ l_{f \nu}\mathbf{w}_2(\nu, \bar{\nu}).
\ea$$

\item[(B2)]
There exist constants $ \kappa_{y},$ $\kappa_{y\nu} \in \mathbb{R}$ such that,
$$
\langle f(s,x,\th, \nu, \imath)-f(s, x, \bar{\th}, \bar{\nu}, \imath), y-\bar{y}\rangle
\les \kappa_{y} |y-\bar{y}|^{2} +\kappa_{y\nu}|\mathbf{w}_{2}(\nu  ,\bar\nu  ) |^2
$$
holds for   all
$(s, \omega, \imath) \in[0, \infty) \times \Omega\times\textbf{M}$, $x\in\dbR^n$, $\th=(y,z, z^0),$ $\bar \th=(\bar y,z, z^0)\in \mathscr{R}$  and
$ \nu , \bar\nu \in   \mathscr{P}_{2}( \mathbb{R}^{n}\times\mathscr{R})$ that have identical marginal distributions  on
$\dbR^n\times\dbR^{nd}\times \dbR^{nd_0}$    under the projection $(x,y,z,z^0)\mapsto (x,z,z^0)$.

 \end{description}

 For convenience in subsequent analysis, we define $\underline \kappa :=2(\kappa_{y}+ \kappa_{y\nu}\mathbf{1}_{\{\kappa_{y\nu}>0\}}+  l^2_{f x}+ l^2_{f z}+ l^2_{f  z^0}+ l^2_{f\nu})$.
  {The exponent $\kappa^\ast$ controls the integrability of the data,
whereas $\overline\kappa_p$ and $\underline\kappa$ describe the forward and
backward dissipativity restrictions, respectively.}
Following the approach  in   \cite{WY-2021}, we first have the following result.

\bl {}\label{BSDE-Le}\sl
Suppose  the condition  {\bf{(B1)}} holds.  Let $(Y ,Z , Z^0, K )\in \cL_{\mathbb{F}}^{2, \kappa}(t, \infty)$ be  the  solution  to \eqref{equ-BSDE},
then   $\lim\limits_{T\to\i}\dbE\big[\mathrm e^{\k T} |Y_T|^2\big]=0$.
\el

We now present an a priori estimate for BSDE  \eqref{equ-BSDE}.

\begin{proposition}\label{Pr-BSDE-1}\sl
Assume  {\bf{(A1)}, \bf{(A2)}$_2$, \bf{(A4)}$_2$, \bf{(B1)}, \bf{(B2)}} hold.  Let $ \k\in ( \underline{\k}, \overline{\k})   $ and $(Y, Z, Z^0 , K)\in \cL_{\mathbb{F}}^{2, \kappa}(t, \infty)$ be the solution of the conditional McKean-Vlasov BSDE \eqref{equ-BSDE} with
 $(t, x_t,\iota)\in \sD^2$. Then  there exists a constant
$
C>0$ such that
\bel{L-BSDE-1}
\begin{aligned}
&\dbE\[ \mathrm e^{\k t}|Y_t|^2+ \int_{t}^{\i} \mathrm e^{\k s}  |\Th_s|^{2}  \mathrm ds +\int_t^\i \mathrm e^{\k s}{|K_s|^{2 }\bullet\mathrm d[M]_s}\]\\
&\les  C \dbE\[ {\mathrm e^{\kappa t}}|x_t|^2+ \int_{t}^{\i}{\mathrm e^{\kappa s } }\big(| \G(s,0,\d_0,\alpha_s)|^2 + |f(s,{\bf{0}},\d_{\bf{0}},\alpha_s)|^{2}\big)\mathrm ds  \].
\end{aligned}
\ee
Furthermore, under the conditions in Corollary  \ref{Cor-SDE-1},  if  $(\bar{Y}, \bar{Z}, \bar{Z^0}, \bar{K}) \in \cL_{\mathbb{F}}^{2, \kappa}(t, \infty)$ is the solution of \eqref{equ-BSDE}  associated with $\bar{X}$ and another coefficient  $\bar{f}$  satisfying {\bf{(B1)}} and {\bf{(B2)}}, then
\bel{L-BSDE-2}
\begin{aligned}
&\dbE\[ \mathrm e^{\k t} |Y_t-\bar{Y}_t|^2
+  \int_{t}^{\infty} \mathrm e^{\k  s}|\Th_s-\bar{\Th}_s|^{2}\mathrm ds
+\int_{t}^{\infty} \mathrm e^{\k  s}{|K_s-\bar{K}_s|^{2 }\bullet\mathrm d[M]_s}\]\\
&\les C\dbE\[\mathrm e^{\kappa t}|x_t-\bar{x}_t|^2
+ \int_t^\infty {\mathrm e^{\kappa  s } }\big(|\G(s,\bar{X}_s,\mathcal{L}^1_{\bar{X}_s},\alpha_s)
-\bar{\G}(s,\bar{X}_s,\mathcal{L}^1_{\bar{X}_s},\alpha_s)|^2\\
&\hskip 4cm+ |f(s,\bar{X}_s,\bar{\Theta}_s,\mathcal{L}^1_{(\bar{X}_s,\bar{\Th}_s )},\alpha_s)
-\bar{f}(s,\bar{X}_s,\bar{\Theta}_s,\mathcal{L}^1_{(\bar{X}_s,\bar{\Th}_s )},\alpha_s)|^{2}\big)\mathrm ds\].
\end{aligned}
\ee

\end{proposition}
\begin{proof}

For any $T \in(t, \infty)$, $\kappa \les \kappa^{\ast}$ and $\e>0$, applying the generalized It\^o's formula to $ \mathrm e^{\kappa s}|Y_{s}|^2$, we get
$$
\begin{aligned}
&\dbE\[ \mathrm e^{\k t}|Y_t|^2+\int_{t}^{T} \mathrm e^{\k s} (\kappa |Y_s|^{2}+ |Z_s|^{2} +|Z^0_s|^{2}) \mathrm ds + \int_{t}^{T} \mathrm e^{\k  s} {|K_s|^{2 }\bullet\mathrm d[M]_s}\]\\
&=  \dbE\[\mathrm e^{\k T} |Y_T|^{2} +\int_{t}^{T}2 \mathrm e^{\k s} \langle Y_s, f(s, X_s, \Theta_s,\mathcal{L}^1_{(X_s,\Theta_s)}, \alpha_s)\rangle
\mathrm ds \]\\
&\les   \mathbb{E}\[\mathrm e^{\k T} |Y_T|^{2} +\int_{t}^{T}2\mathrm e^{\k s} \(\kappa_{y}|Y_s|^{2}
+ \kappa_{y\nu}|\mathbf{w}_{2}(\mathcal{L}^1_{Y_s } ,\d_0)|^{2}
+|Y_s|\cdot( l_{f x}|X_s|+  l_{f z}|Z_s|+  l_{f  z^0}|Z^0_s|)\\
&\hskip4.5cm
+|Y_s|\cdot  l_{f \nu}|\mathbf{w}_{2}(\mathcal{L}^1_{(X_s, 0,Z_s,Z^0_s)}, \delta_0)|
+|Y_s|\cdot|f(s,\textbf{0} ,\delta_\textbf{0},\alpha_s)|\)\mathrm ds\]\\
&\les  \dbE\[\mathrm e^{\k T} |Y_T|^{2} +\int_{t}^{T}\mathrm e^{\k s}  \((2\kappa_{y}+2 \kappa_{y\nu}\mathbf{1}_{\{\kappa_{y\nu}>0\}}+2\mathbf{l}_{f}^2+7 \varepsilon) |Y_s|^{2}
+c^{\nu, \varepsilon}_{x}|X_s|^2
+c^{\nu, \varepsilon}_{z}|Z_s|^{2}
+c^{\nu, \varepsilon}_{ z^0}|Z^0_s|^{2}\\
&\hskip4.5cm
+\frac{1}{\varepsilon}|f(s,\textbf{0},\d_\textbf{0},\alpha_s)|^{2}\)\mathrm ds\],
\end{aligned}
$$
where   $\mathbf{l}_{f}^2:=  l_{f x}^2+ l_{f z}^2+ l_{f  z^0}^2+  l_{f \nu}^2$ and
$c^{\nu, \varepsilon}_{\zeta}=\frac{  l^2_{f \zeta}}{2   l^2_{f \zeta}+\varepsilon}+\frac{ l^2_{f \nu}}{2 l^2_{f \nu}+\varepsilon}$ for $\zeta=x, z,  z^0$, resp.
In the above,  we have used the  Young inequality and \eqref{E1-inequ}, for example,
$$
\begin{aligned}
&\dbE\big[2  l_{f \nu}|Y_s|\cd|\mathbf{w}_{2}(\mathcal{L}^1_{(X_s, 0,Z_s,Z^0_s)}, \delta_0)|\big]
\les (2  l_{f \nu}^2+\varepsilon)\dbE\big[|Y_s|^2\big]+\frac{  l_{f \nu}^2}{2 l_{f \nu}^2+\varepsilon}\dbE\big[|\mathbf{w}_{2}(\mathcal{L}^1_{(X_s, 0,Z_s,Z^0_s)}, \delta_0)|^2\big]  \\
&
\les 
 (2  l_{f \nu}^2+\varepsilon)\dbE\big[|Y_s|^2\big]+\frac{  l_{f \nu}^2}{2 l_{f \nu}^2+\varepsilon}
 \mathbb{E}\big[|X_s|^2+|Z_s|^2+|Z^0_s|^2\big].
\end{aligned}
$$

Combined with  \eqref{L-p-SDE-1}-(ii) ($p=2$), we  get some constant $C_\e>0$ such that
\bel{FH-est}
\begin{aligned}
&\dbE\[ \mathrm e^{\k t}|Y_t|^2+(\kappa-2\kappa_{y}-2 \kappa_{y\nu}\mathbf{1}_{\{\kappa_{y\nu}>0\}}-2\mathbf{l}_{f}^2-7 \varepsilon)\int_{t}^{T} \mathrm e^{\k s} |Y_s|^{2} \mathrm ds
+(1-c^{\nu, \varepsilon}_{z})\int_{t}^{T} \mathrm e^{\k s} |Z_s|^{2} \mathrm ds\\
&\q
+(1-c^{\nu, \varepsilon}_{ z^0})\int_{t}^{T} \mathrm e^{\k s} |Z^0_s|^{2} \mathrm ds
+ \int_{t}^{T} \mathrm e^{\k  s} {|K_s|^{2 }\bullet \mathrm d[M]_s}\]\\
&\les  \dbE\[\mathrm e^{\k T} |Y_T|^2+C_\e {\mathrm e^{\kappa t}}|x_t|^2+ C_\e \int_{t}^{T}{\mathrm e^{\kappa  s } }(| \G(s,0,\d_0,\alpha_s)|^2 + |f(s,\textbf{0},\d_{\textbf{0}},\alpha_s)|^{2})\mathrm ds  \].
\end{aligned}
\ee
Then,  {when $\underline{\k}<\kappa<\overline{\k}$ and $\kappa\les\kappa^\ast$}, letting $T \rightarrow \infty$  yields  \eqref{L-BSDE-1}.
Finally,  \eqref{L-BSDE-2}  follows similarly by considering the difference between two solutions.
\end{proof}

\bt{Th-BSDE}\sl Suppose that {\bf{(A1)}, \bf{(A2)}$_2$, \bf{(A4)}$_2$, \bf{(B1)}, \bf{(B2)}} hold.    Then  the   conditional McKean-Vlasov BSDE \eqref{equ-BSDE} admits a unique solution $(Y, Z,Z^0 ,K) \in  \cL^{2, \k}_\dbF(t,\i) $  whenever {$\underline{\k}<\kappa<\overline{\k}$ and $\kappa\les\kappa^\ast$}.
\et

\begin{proof}
{
The uniqueness follows directly from the a priori estimate \eqref{L-BSDE-2}.
We now prove the existence.

\emph{Step 1.} For each integer $n\ges1$, consider the following
finite horizon BSDE on $[t,t+n]$ with terminal condition $Y_{t+n}^n=0$:
\begin{equation}\label{BSDE-n}
\begin{aligned}
Y_s^n
&=
\int_s^{t+n} f\big(r,X_r,\Theta_r^n,
\cL^1_{(X_r,\Theta_r^n)},\alpha_r\big)\,{\mathrm d}r
-\int_s^{t+n} Z_r^n\,{\mathrm d}W_r        \\
&\quad
-\int_s^{t+n} Z_r^{0,n}\,{\mathrm d}W_r^0
-\int_s^{t+n} K_r^n\bullet{\mathrm d}M_r,
\end{aligned}
\end{equation}
where $\Theta^n=(Y^n,Z^n,Z^{0,n})$.
The well-posedness of this finite horizon system follows by freezing the
conditional law flow and applying the finite horizon well-posedness theory for
BSDEs with Markovian switching, together with a standard fixed-point argument
for the conditional McKean{-}Vlasov dependence; see, e.g., \cite{2024-GNY}.
We then extend the solution to $[t,\infty)$ by setting
$
Y_s^n=Z_s^n=Z_s^{0,n}=K_s^n=0,$ $ s>t+n.
$

We claim that $\{(Y^n,Z^n,Z^{0,n},K^n)\}_{n\ges1}$ is a Cauchy sequence in
$\mathcal L_{\mathbb F}^{2,\kappa}(t,\infty)$. Let $m>n$ and fix {$\underline\k<\kappa<\overline\k$ with $\kappa\les\kappa^\ast$}. By the finite horizon
counterpart of the stability estimate \eqref{L-BSDE-2} on $[t,t+n]$, we have
\bel{Ym-Yn}
\begin{aligned}
&\mathbb E\bigg[
{\mathrm e^{\kappa t}}|Y_t^m-Y_t^n|^2
+\int_t^{t+n}{\mathrm e^{\kappa s}}|\Theta_s^m-\Theta_s^n|^2\,{\mathrm d}s
+\int_t^{t+n}{\mathrm e^{\kappa s}}|K_s^m-K_s^n|^2\bullet{\mathrm d}[M]_s
\bigg]    \\
&\les
C\mathbb E\big[{\mathrm e^{\kappa(t+n)}}|Y_{t+n}^m|^2\big],
\end{aligned}
\ee
where we used $Y_{t+n}^n=0$.
Then, using \eqref{FH-est},
$$
\begin{aligned}
&\mathbb E\[
{\mathrm e^{\kappa t}}|Y_t^m-Y_t^n|^2
+\int_t^\infty {\mathrm e^{\kappa s}}|\Theta_s^m-\Theta_s^n|^2\,{\mathrm d}s
+\int_t^\infty {\mathrm e^{\kappa s}}|K_s^m-K_s^n|^2\bullet{\mathrm d}[M]_s
\]                                      \\
&\les
C\mathbb E\[ {\mathrm e^{\kappa(t+n)}}|Y_{t+n}^m|^2+
\int_{t+n}^{t+m} {\mathrm e^{\kappa s}}|\Theta_s^m|^2\,{\mathrm d}s
+\int_{t+n}^{t+m} {\mathrm e^{\kappa s}}|K_s^m|^2\bullet{\mathrm d}[M]_s
\]\\
& \les
C\mathbb E\[ \int_{t+n}^{t+m} {\mathrm e^{\kappa s}}
\big(|X_s|^2+
\mathbf{w}_{2}^2(\mathcal L_{X_s}^1,\delta_0)
+
|{f(s, \mathbf 0,\delta_{\mathbf 0},\alpha_s)}|^2
\big)\,{\mathrm d}s\] \\
& \les
C\mathbb E\[ \int_{t+n}^{t+m} {\mathrm e^{\kappa s}}
\big(|X_s|^2
+
|{f(s,\mathbf 0,\delta_{\mathbf 0},\alpha_s)}|^2
\big)\,{\mathrm d}s \]\to 0, \qq \mbox{as } m,n\to\infty,
\end{aligned}
$$
by $X\in L_\dbF^{2,\k}(t,\i;\dbR^n)$ and {$f(\cdot,\mathbf 0,\delta_{\mathbf 0},\imath)\in
L_{\mathbb F}^{2,\kappa^\ast} (0,\i;\dbR^n)\subseteq L_{\mathbb F}^{2,\kappa}(0,\i;\dbR^n)$.} Hence
$\{(Y^n,Z^n,Z^{0,n},$ $K^n)\}_{n\ges1}$ is a   Cauchy sequence in
$\mathcal L_{\mathbb F}^{2,\kappa}(t,\infty)$. Then
there exists
$(Y,Z,Z^0,K)\in\mathcal L_{\mathbb F}^{2,\kappa}(t,\infty)$ such that
\begin{equation}\label{BSDE-limit-conv}
(Y^n,Z^n,Z^{0,n},K^n)
\to
(Y,Z,Z^0,K)
\quad\text{in }\  \mathcal L_{\mathbb F}^{2,\kappa}(t,\infty).
\end{equation}

Next we prove the local $\mathcal D^{2,\kappa}$-convergence of
$\{Y^n\}_{n\ges1}$. Fix $T>t$. Let $m>n$ be sufficiently large such that
$T<t+n$. By \eqref{Ym-Yn}, we obtain
$$\ba{ll}
\ns\ds
\mathbb E\big[{\mathrm e^{\kappa T}}|Y_T^m-Y_T^n|^2\big]
\les
C\mathbb E\big[{\mathrm e^{\kappa(t+n)}}|Y_{t+n}^m|^2\big] \les
C\mathbb E\[\int_{t+n}^{t+m} {\mathrm e^{\kappa s}}
\left(
|X_s|^2+
|f(s,\mathbf 0,\delta_{\mathbf 0},\alpha_s)|^2
\right)\,{\mathrm d}s\]
\to 0,
\ea$$
 as $m,n\to\infty$.
  Hence, for each fixed $T>t$,  $\{Y_T^n\}_{n\ges1}$ is a Cauchy sequence under the
weighted $L^2$-norm.

Applying the standard arguments of  finite horizon stability estimate to $Y^m-Y^n$,
we obtain
$$
\begin{aligned}
&\mathbb E\[\sup_{s\in[t,T]}{\mathrm e^{\kappa s}}|Y_s^m-Y_s^n|^2\]   \les
C\mathbb E\[{\mathrm e^{\kappa T}}|Y_T^m-Y_T^n|^2
+\int_t^T {\mathrm e^{\kappa r}}
\left|\Theta_r^m-\Theta_r^n\right|^2\,\mathrm dr\]\to 0, \q  \mbox{as } m,n\to\infty.
\end{aligned}
$$
  Therefore $\{Y^n\}_{n\ges1}$ is a Cauchy sequence in
$\mathcal D_{\mathbb F}^{2,\kappa}(t,T;\dbR^n)$. Its limit is denoted by  $\widetilde Y\in\mathcal D_{\mathbb F}^{2,\kappa}(t,T;\dbR^n)$. On the other hand, by \eqref{BSDE-limit-conv},
$Y^n\to Y$ in $L_{\mathbb F}^{2,\kappa}(t,T;\dbR^n)$. Since convergence in
$\mathcal D_{\mathbb F}^{2,\kappa}(t,T;\dbR^n)$ implies convergence in
$L_{\mathbb F}^{2,\kappa}(t,T;\dbR^n)$, the uniqueness of the
$L_{\mathbb F}^{2,\kappa}$-limit yields $\widetilde Y=Y$ in
$L_{\mathbb F}^{2,\kappa}(t,T;\dbR^n)$.  Thus, choosing the corresponding c\`adl\`ag
version of $Y$, we have
$$
\lim_{n\to\infty}
\mathbb E\[
\sup_{s\in[t,T]}{\mathrm  e^{\kappa s}}|Y_s^n-Y_s|^2
\]=0.
$$
In particular, for any fixed $T>t$,
$\ds
\mathbb E\left[{\mathrm e^{\kappa T}}|Y_T^n-Y_T|^2\right]\to0,
$ as $n\to\i$.

 \ss

\emph{Step 2.}
It remains to identify the limit.
Fix $T>t$. For the sufficiently large $n$, we have $T<t+n$, and hence
\eqref{BSDE-n} can be rewritten on $[t,T]$ as
\begin{equation}\label{BSDE-n-local-discounted}
\begin{aligned}
{\mathrm e^{\frac{\kappa}{2}s}}Y_s^n
&=
{\mathrm e^{\frac{\kappa}{2}T}}Y_T^n
+\int_s^T
{\mathrm e^{\frac{\kappa}{2}r}}\[
f\big(r,X_r,\Theta_r^n,
\cL^1_{(X_r,\Theta_r^n)},\alpha_r\big)
-\frac{\kappa}{2}Y_r^n
\]\,{\mathrm d}r        \\
&\quad
-\int_s^T {\mathrm e^{\frac{\kappa}{2}r}}Z_r^n\,{\mathrm d}W_r
-\int_s^T {\mathrm e^{\frac{\kappa}{2}r}}Z_r^{0,n}\,{\mathrm d}W_r^0
-\int_s^T {\mathrm e^{\frac{\kappa}{2}r}}K_r^n\bullet{\mathrm d}M_r,
\qquad s\in[t,T].
\end{aligned}
\end{equation}
 By \eqref{BSDE-limit-conv}, and the local $\mathcal D_{\mathbb F}^{2,\kappa}$-convergence proved above, together with the It\^o isometry for the Brownian
stochastic integrals, the terms involving $Y^n$, $Z^n$ and $Z^{0,n}$ converge
to the corresponding terms involving $Y$, $Z$ and $Z^0$. Moreover, by the
Lipschitz continuity of $f$, as $n\to \i$,
$$
\mathbb E\[\!\int_t^T\!{\mathrm  e^{\kappa r}}
\big|
f(r,X_r,\Theta_r^n,\cL^1_{(X_r,\Theta_r^n)},\alpha_r)
-
f(r,X_r,\Theta_r,\cL^1_{(X_r,\Theta_r)},\alpha_r)
\big|^2\,{\mathrm d}r\]
\les
C\mathbb E\[\!\int_t^T\! {\mathrm e^{\kappa r}}|\Theta_r^n-\Theta_r|^2\,{\mathrm d}r\]
\to0.
$$
For the Markovian switching martingale term,  we have
$$
\mathbb E\[\Big|
\int_s^T {\mathrm e^{\frac\k2 r}}(K_r^n-K_r)\bullet{\mathrm d}M_r
\Big|^2\]
\les
C\mathbb E\[\int_s^T {\mathrm e^{\k r}}|K_r^n-K_r|^2\bullet{\mathrm d}[M]_r\]
\to0,\q \mbox{as } n\to\i.
$$

Finally, letting $n\to\infty$ in \eqref{BSDE-n-local-discounted}, we obtain, for all
$t\les s\les T$,
$$
\begin{aligned}
{\mathrm e^{\frac{\kappa}{2}s}}Y_s
&=
{\mathrm e^{\frac{\kappa}{2}T}}Y_T
+\int_s^T
{\mathrm e^{\frac{\kappa}{2}r}}\left[
f\big(r,X_r,\Theta_r,
\cL^1_{(X_r,\Theta_r)},\alpha_r\big)
-\frac{\kappa}{2} Y_r
\right]\,{\mathrm d}r        \\
&\quad
-\int_s^T {\mathrm e^{\frac{\kappa}{2}r}}Z_r\,{\mathrm d}W_r
-\int_s^T {\mathrm e^{\frac{\kappa}{2}r}}Z_r^0\,{\mathrm d}W_r^0
-\int_s^T {\mathrm e^{\frac{\kappa}{2}r}}K_r\bullet{\mathrm d}M_r,
\end{aligned}
$$
which is indeed equivalent to \eqref{equ-BSDE}.
Since $T>t$ is arbitrary, $(Y,Z,Z^0,K)$ satisfies the conditional
McKean{-}Vlasov BSDE \eqref{equ-BSDE} on $[t,\infty)$.}
  \end{proof}

\section{Infinite horizon  conditional McKean-Vlasov   FBSDEs with markovian switching---fully coupled case}
\label{MKV-FBSDE}

This section is devoted to the study of the following infinite horizon conditional McKean-Vlasov FBSDEs with Markovian switching,
\begin{equation}\label{equ-FBSDE-1}
 \left\{
\begin{aligned}
&\! \mathrm dX_s\!=\!b(s, X_s,\Theta_s, \mathcal{L}^1_{(X_s,\Th _s)}, \alpha_s) \, \mathrm ds
\!+\! \sigma  (s, X_s,\Theta_s, \mathcal{L}^1_{(X_s,\Th _s)}, \alpha_s) \, \mathrm dW_s \\
 &\qq
\!+\! \widetilde{\sigma}    (s, X_s,\Theta_s, \mathcal{L}^1_{(X_s,\Th _s)}, \alpha_s) \, \mathrm d W^0_s , \ s\ges t,\\
&\! \mathrm dY_s\!= \!-f(s,X_s,\Theta_s, \mathcal{L}^1_{(X_s,\Th _s)}, \alpha_s) \, \mathrm ds
\!+\!  Z_{s}  \, \mathrm dW_s
\!+\!   Z^0_{s}  \, \mathrm d W^0_s
\!+\! { K_s \bullet\mathrm d M_s}, \ s\ges t, \\
&\! X_t=\Phi(Y_t, \mathcal{L}^1_{Y_t}, \alpha_t), \quad
\alpha_t=\iota,
\end{aligned}
\right.
\end{equation}
where
the coefficients
$$
\begin{aligned}
&b:\Omega\times [0,\infty)\times\dbR^n\times \sR\times\sP(\dbR^n\times\sR)\times\textbf{M}\to\dbR^n,\\
&\si:\Omega\times[0,\infty)\times\dbR^n\times \sR\times\sP(\dbR^n\times\sR)\times\textbf{M}\to\dbR^{n d}, \\
&\widetilde{\sigma}:\Omega\times[0,\infty)\times\dbR^n\times \sR\times\sP(\dbR^n\times\sR)\times\textbf{M}\to\dbR^{n  d_0},\\
&f:\Omega\times[0,\infty)\times\dbR^n\times \sR\times\sP(\dbR^n\times\sR)\times\textbf{M}\to\dbR^{n},\\
&\Phi:\Omega\times \dbR^n {\times}\sP(\dbR^n )\times\textbf{M}\to\dbR^{n}.
\end{aligned}
$$
Note that  the forward and backward equations \eqref{equ-FBSDE-1} are   fully coupled.
We assume the coefficients to  satisfy the following conditions.

\begin{description}
    \item[(C1)]
  (i) For any $(y,\upsilon,\imath)\in
\dbR^n\times\mathscr{P}(\dbR^n)\times\mathbf{M}$,
$\Phi(\cdot,y,\upsilon,\imath)$ is $\mathcal{F}_t$-measurable.
Moreover, for any $(x,\theta,\nu,\imath)\in
\dbR^n\times\mathscr{R}\times\mathscr{P}(\dbR^n\times\sR)\times\mathbf{M}$,
$f(\cdot,x,\theta,\nu,\imath)$ and
$\G(\cdot,x,\theta,\nu,\imath)$ are $\dbF$-progressively measurable;

(ii) For all $\imath \in\textbf{M}$,  $\Phi(0, \delta_0, \imath) \in L_{\mathcal{F}_{t}}^{2}(\Omega; \mathbb{R}^{n})$ and
 $  f(\cdot,{ \bf{0}},\delta_{ \bf{0}}, \imath) \in L_{\mathbb{F}}^{2,  \k^\ast}(0, \infty ; \dbR^n) $, $\G (\cdot,{ \bf{0}},\delta_{ \bf{0}}, \imath) \in L_{\mathbb{F}}^{2,  \k^\ast}(0, \infty ;  \sR )$ for some $\k^\ast\in\mathbb{R}$;

(iii)  There exist   nonnegative constants $ l_{ {\mathbf{h}} x},  l_{ {\mathbf{h}} y},  l_{ {\mathbf{h}} z}, l_{ {\mathbf{h}}  z^0} ,   l_{ {\mathbf{h}} \nu},  l_{\Phi y},  l_{\Phi \upsilon}$  such that,
 for all $(\o,s, \imath) \in\Omega\times[0, \infty) \times \textbf{M}$,  $(x,\theta,\nu)$, $(\bar x,\bar{\theta},\bar{\nu}) \in \mathbb{R}^{n}\times\sR \times \mathscr{P}_{2}(\mathbb{R}^n\times\sR)$ and
 $ \upsilon , $ $   \bar{ \upsilon} \in   \mathscr{P}_{2}(\mathbb{R}^n)$,
 $$\ba{ll}
 \ns\ds |\Phi(y,  \upsilon , \imath) - \Phi(\bar{y}, \bar{ \upsilon} , \imath)| \les   l_{\Phi y} |y-\bar y | +  l_{\Phi \upsilon} \mathbf{w}_{2}( \upsilon , \bar{ \upsilon} ),\\
 \ns\ds | {\mathbf{h}}(s, x,\theta, \nu, \imath)\!-\! { \mathbf{h}}(s, \bar{x},\bar \theta ,\bar{\nu}, \imath)|
\les \!  l_{ {\mathbf{h}}x}|x-\bar x  | \!+\!  l_{ {\mathbf{h}}y}| y-\bar y  |  \!+ \! l_{ {\mathbf{h}} z}| z-\bar  z |\!+\! l_{ {\mathbf{h}}  z^0}| z^0-\bar{z}^0 |\!+ \!l_{ {\mathbf{h}}\nu}\mathbf{w}_2(\nu, \bar{\nu})
 \ea$$
 hold, where $ {\mathbf{h}}=b,\si,\widetilde{\sigma}, f$, resp.

  \item[(C2)]
 There exist  constants  $\k_x, \k_{x\mu} , \kappa_{y}, \k_{y\nu}  \in \mathbb{R}$ such that, for all $s\ges t$ and $(X, \Theta ), (\bar{X}, \bar{\Theta} )\in L^{2}_{\mathcal{F}_{s}}(\O;\mathbb{R}^n\times\sR) $,
$$
 \ba{ll}
 \ns\ds { \rm (i)}\ \
 \dbE \[2\lan b(s,X,\Theta, \mathcal{L}^1_{(X, \Th )},\alpha_s )-b(s, \bar{X}, \Th, \mathcal{L}^1_{(\bar{X}, \Th )},\alpha_s),  X-\bar X  \ran\\
 \ns\ds\qq \q   +  |\si(s,X,\Theta, \mathcal{L}^1_{(X, \Th )},\alpha_s )-\si(s, \bar{X}, \Th, \mathcal{L}^1_{(\bar{X}, \Th )},\alpha_s)|^2 \\
 \ns\ds\qq \q + |\wt\si(s,X,\Theta, \mathcal{L}^1_{(X, \Th )},\alpha_s )-\wt\si(s, \bar{X}, \Th, \mathcal{L}^1_{(\bar{X}, \Th )},\alpha_s)|^2 \]   \\
 \ns\ds\qq
 \les (\k_x+\k_{x\mu} \mathbf{1}_{\{\k_{x\mu} >0\}})  \dbE[ |X-\bar X |^2 ],\\
 \ns\ds{  \rm (ii)}\    \dbE\Big[ \langle f(s, X,Y,Z,Z^0, \mathcal{L}^1_{(X, Y,Z,Z^0 )},\alpha_s)-f(s,  {X}, \bar{Y},   Z,Z^0 \mathcal{L}^1_{( {X}, \bar{Y},Z,Z^0)},\alpha_s), Y-\bar{Y}\rangle\Big]\\
 \ns\ds\qq   \les(\k_y+\k_{y\nu} \mathbf{1}_{\{\k_{y\nu} >0\}})  \dbE[ |Y-\bar Y |^2   ].
   \ea $$

\item[(C3)]
There exist nonnegative constants $\alpha,\beta_1,\beta_2$ and
nonnegative measurable functions
$$
\begin{aligned}
&\phi:\Omega\times\mathbb{R}^n\times\mathbb{R}^n
\times\mathscr{P}_2(\mathbb{R}^n)\times\mathscr{P}_2(\mathbb{R}^n)
\times\textbf{M}\rightarrow[0,\infty),\\
&\psi:\Omega\times[0,\infty)\times\mathbb{R}^n\times\mathbb{R}^n
\times\mathscr{P}_2(\mathbb{R}^n)\times\mathscr{P}_2(\mathbb{R}^n)
\times\textbf{M}\rightarrow[0,\infty),\\
&\varphi:\Omega\times[0,\infty)\times\mathscr{R}\times\mathscr{R}
\times\mathscr{P}_2(\mathscr{R})\times\mathscr{P}_2(\mathscr{R})
\times\textbf{M}\rightarrow[0,\infty),
\end{aligned}
$$
whose square roots are globally Lipschitz continuous in their
respective Euclidean and measure arguments, uniformly in the
remaining arguments. For each $\imath\in\textbf{M}$, these functions
satisfy
$
 \phi(\mathbf{0},\delta_{\mathbf{0}},\imath)
\in L_{\mathcal{F}}^1(\Omega;\mathbb{R}),$ $ \psi(\cdot,\mathbf{0},\delta_{\mathbf{0}},\imath),\,
\varphi(\cdot,\mathbf{0},\delta_{\mathbf{0}},\imath)
\in L_{\dbF}^{1,\kappa}(0,\infty;\mathbb{R}).
$
In addition, the following conditions hold:

 \no (i)   Case 1: $\b_1>0$, $\b_2=0$ or  Case 2: $\b_1=0$, $\b_2>0$;\\
 \ms
 \no (ii)  (Domination condition)
For  all $s\ges t$ and
$(X, \Theta ), (\bar{X}, \bar{\Theta} )
\in L^{2}_{\mathcal{F}_{s}}(\O;\mathbb{R}^n\times\sR) $,
$$
 \left\{
 \begin{aligned}
&\!\mathbb{E}\left[\big|\Phi(Y, \mathcal{L}^1_{Y},\alpha_s)
 - \Phi(\bar{Y}, \mathcal{L}^1_{\bar{Y}},\alpha_s)\big|^{2}\right]
 \les  \frac{1}{\b_1} \mathbb{E}\left[\phi(Y,\bar{Y}, \mathcal{L}^1_{Y},\mathcal{L}^1_{\bar{Y}},\alpha_s)\right], \\
&\!\mathbb{E}\left[\big|f(s,X, \Theta, \mathcal{L}^1_{(X,\Theta)},\alpha_s)
 -f(s, \bar{X}, \Th, \mathcal{L}^1_{(\bar{X}, \Th )},\alpha_s)\big|^{2}\right]
\les  \frac{1}{\b_2} \mathbb{E}\left[\psi(s, X, \bar{X}, \mathcal{L}^1_{X}, \mathcal{L}^1_{\bar{X}},\alpha_s)\right], \\
&\!\mathbb{E}\left[\big|\mathbf{h}(s, X,\Theta, \mathcal{L}^1_{(X, \Th )},\alpha_s)
 -\mathbf{h}(s, X, \bar{\Th},   \mathcal{L}^1_{(X, \bar{\Th} )},\alpha_s)\big|^{2}\right]
\les  \frac{1}{\b_1}  \mathbb{E}\left[\varphi(s, \Theta, \bar{\Theta}, \mathcal{L}^1_{\Theta}, \mathcal{L}^1_{\bar{\Theta}},\alpha_s)\right],
 \end{aligned}\right.
$$
where $\mathbf{h}=b, \sigma, \widetilde{\sigma}$, resp;

%
%
%

 \no (iii)  (Monotonicity condition)
%
%
For  all $s\ges t$ and
$(X, \Theta ), (\bar{X}, \bar{\Theta} )
\in L^{2}_{\mathcal{F}_{s}}(\O;\mathbb{R}^n\times\sR) $, we have
$${\mbox{Case 1:}} \
 \left\{
 \begin{aligned}
 &\!\mathbb{E}\left[\langle\Phi(Y, \mathcal{L}^1_{Y},\alpha_s)
 - \Phi(\bar{Y}, \mathcal{L}^1_{\bar{Y}},\alpha_s), Y-\bar{Y}\rangle\right]
 \les  -\b_1 \mathbb{E}\left[\phi(Y,\bar{Y}, \mathcal{L}^1_{Y},\mathcal{L}^1_{\bar{Y}},\alpha_s)\right], \\
 &\!\mathbb{E}\left[ \left\langle\!\!  \left(\begin{array}{ccc}
  \!\!\!-f(s, X,\Theta, \mathcal{L}^1_{(X, \Th )},\alpha_s)+f(s, \bar{X}, \bar{\Th},   \mathcal{L}^1_{(\bar{X}, \bar{\Th} )},\alpha_s)\!\!\!\\
  \!\!\!\Gamma(s, X,\Theta, \mathcal{L}^1_{(X, \Th )},\alpha_s) -\G(s, \bar{X}, \bar{\Th},   \mathcal{L}^1_{(\bar{X}, \bar{\Th} )},\alpha_s)\!\!\!
  \end{array} \right)\!
 ,\!\left( \begin{array}{ccc}
 \!\!\!\!X-\bar{X} \!\!\! \!\\
 \!\!\!\!\Th-\bar{\Th}  \!\!\!\!   \\
   \end{array} \right)\!\!
  \right\rangle\right]\\
  &+\(-\frac{\k_x}{2}+\k_{y}\) \mathbb{E}\left[ \langle X-\bar{X}, Y-\bar{Y} \rangle \right]\\
 &\!\les - \b_2 \mathbb{E}\left[\psi(s, X, \bar{X}, \mathcal{L}^1_{X}, \mathcal{L}^1_{\bar{X}},\alpha_s)\right]
 -\b_1 \mathbb{E}\left[\varphi(s, \Theta, \bar{\Theta}, \mathcal{L}^1_{\Th}, \mathcal{L}^1_{\bar{\Th}},\alpha_s)\right],
 \end{aligned}\right. $$
or
$${\mbox{Case 2:}} \
 \left\{
 \begin{aligned}
 &\!\mathbb{E}\left[\langle\Phi(Y, \mathcal{L}^1_{Y},\alpha_s)
 - \Phi(\bar{Y}, \mathcal{L}^1_{\bar{Y}},\alpha_s), Y-\bar{Y}\rangle\right]
 \ges  \b_1 \mathbb{E}[\phi(Y,\bar{Y}, \mathcal{L}^1_{Y},\mathcal{L}^1_{\bar{Y}},\alpha_s)] , \\
 &\!\mathbb{E}\left[ \left\langle\!\!  \left(\begin{array}{ccc}
  \!\!\!-f(s, X,\Theta, \mathcal{L}^1_{(X, \Th)},\alpha_s)+f(s, \bar{X}, \bar{\Th},   \mathcal{L}^1_{(\bar{X}, \bar{\Th} )},\alpha_s)\!\!\!\\
  \!\!\!\Gamma(s, X,\Theta, \mathcal{L}^1_{(X, \Th)},\alpha_s) -\G(s, \bar{X}, \bar{\Th},   \mathcal{L}^1_{(\bar{X}, \bar{\Th} )},\alpha_s)\!\!\!
  \end{array} \right)\!
 ,\!\left( \begin{array}{ccc}
 \!\!\!\!X-\bar{X} \!\!\! \!\\
 \!\!\!\!\Th-\bar{\Th} \!\!\!\!   \\
   \end{array} \right)\!\!
  \right\rangle\right]\\
 & +\(-\frac{\k_x}{2}+\k_{y}\) \mathbb{E}\left[ \langle X-\bar{X}, Y-\bar{Y} \rangle \right]\\
 &\!\ges  \b_2 \mathbb{E}\left[\psi(s, X, \bar{X}, \mathcal{L}^1_{X}, \mathcal{L}^1_{\bar{X}},\alpha_s)\right]
 +\b_1 \mathbb{E}\left[\varphi(s, \Theta, \bar{\Theta}, \mathcal{L}^1_{\Th}, \mathcal{L}^1_{\bar{\Th}},\alpha_s)\right].
 \end{aligned}\right.
 $$
\end{description}
Here, with a bit of abuse of notations, when $\b_1=0$ (resp. $\b_2=0$ ), $1 / \b_1$ (resp. $1 / \b_2)$ means $+\infty$.
 This means the corresponding domination constraints vanish when   $\b_1=0$ or $\b_2=0$.

\begin{remark}\label{ReFBSDE}\sl
In line with the conventional approach in the mean-field literature  (e.g., \cite{Bensoussan-Yam-Zhang-2015, Ahuja-Ren-Yang-2019, Tian-Yu}, etc.), our conditions are formulated in probabilistic terms, i.e., using random variables and their distributions, as this framework is more natural and better suited for the derivations carried out in the proofs.
While this approach is natural for the derivations, we also present a version  of these conditions   in terms of state variables and measures to make the underlying structural relationships more explicit. We emphasize that this latter formulation ensures the validity of the probabilistic conditions.

We first set
$ \pi_x(x,\th ):=x,$ $\pi_y(x,\th ):=y, $ $
\pi_{x,z,z^0 }(x,\th ):=(x,z,z^0) ,
 $  $
\pi_{\th }(x,\th ):=\th
 $ for $(x,\th)\in \dbR^n\times\sR.$
For all
$(\o,s, \imath) \in\Omega\times[0, \infty) \times \textbf{M}$,
all $(x,\theta,\nu)$, $(\bar x,\bar{\theta},\bar{\nu}) \in \mathbb{R}^{n}\times\sR \times\mathscr{P}_{2}(\mathbb{R}^n\times\sR)$,
$  \upsilon, \bar{\upsilon}\in \mathscr{P}_{2}(\mathbb{R}^n)$,
$
   \mu_1,   \bar{\mu}_1\in \mathscr{P}_{2}(\mathbb{R}^n\times\sR)$  with $\mu_1\circ \pi_\th^{-1}=\bar\mu_1\circ \pi_\th^{-1}$,
$ \mu_2,   \bar{\mu}_2\in \mathscr{P}_{2}(\mathbb{R}^n\times\sR) $  with $\mu_2\circ \pi_x^{-1}=\bar\mu_2\circ \pi_x^{-1}$,
$  \nu_1,   \bar{\nu}_1\in \mathscr{P}_{2}(\mathbb{R}^n\times\sR) $  with $\nu_1\circ \pi_{x,z,z^0}^{-1}=\bar\nu_1\circ \pi_{x,z,z^0}^{-1} $,  the following holds, $\dbP$-a.s.,
%
 $$\ba{ll}
 \ns\ds {\textbf{\rm(C2)}}^ \prime \left\{\3n
 \ba{ll}
 \ns\ds { \rm (i)}\
2\lan b(s,x,\theta,\mu_1, \imath)-b(s,\bar x,\theta,\bar\mu_1 ,  \imath),  x-\bar x  \ran+  |\si(s,x,\th,\mu_1, \imath)-\si(s,\bar x,\th,\bar\mu_1, \imath)|^2 \\
 \ns\ds\qq  + |\widetilde{\sigma}(s,x, \theta, \mu_1,\imath)-\widetilde{\sigma}(s,\bar x,\theta,\bar\mu_1, \imath)|^2  \les \k_x|x-\bar x |^2+\k_{x\mu} |\mathbf{w}_{2}(\mu_1\circ\pi_x^{-1} ,\bar\mu_1\circ\pi_x^{-1} )|^{2} ,\\
 \ns\ds{  \rm (ii)}\    \langle f(s,x,y,z, z^0, \nu_1  , \imath)-f(s, x,  \bar{y}, z,  z^0 ,\bar\nu_1  ,   \imath ), y-\bar{y}\rangle\\
 \ns\ds\qq \les \k_{y} |y-\bar y |^{2} + \k_{y\nu} |\mathbf{w}_{2}(\nu_1\circ\pi_y^{-1} ,\bar \nu_1\circ\pi_y^{-1} )|^{2}  ;
   \ea \right.
   \ea$$
    $$\ba{ll}
\ns\ds {\textbf{\rm(C3)$^\prime $-(ii)}}\
 \left\{
 \ba{ll}
\ns\ds \!|\Phi(y,\upsilon,\imath)- \Phi(\bar y, \bar \upsilon, \imath) |^2
 \les  \frac{1}{\b_1} \phi(y,\bar y, \upsilon , \bar{\upsilon}, \imath ), \\
\ns\ds\!|f(s,x,\th,  \mu_1,\imath) -f(s, \bar x,  \th, \bar \mu_1, \imath)  |^{2} \les  \frac{1}{\b_2}
 \psi(s, x, \bar x, \mu_1\circ\pi_x^{-1}, \bar \mu_1\circ\pi_x^{-1}, \imath), \\
\ns\ds\!|\mathbf{h}(s,x,\th, \mu_2,\imath)
 -\mathbf{h}(s, x,\bar\th, \bar \mu_2,\imath)\big|^{2} \les  \frac{1}{\b_1}\varphi(s, \theta, \bar{\theta},\mu_2\circ\pi_{\th}^{-1},\bar{\mu}_2\circ\pi_{\th}^{-1}, \imath);
 \ea\right.
\ea$$
 and
$$\begin{aligned} &{\textbf{\rm(C3)$^\prime $-(iii)-Case 1}}\
 \left\{
 \begin{aligned}
 &\!\big\langle \Phi(y, \upsilon,\imath)
 - \Phi(\bar y, \bar \upsilon, \imath), y-\bar{y} \big\rangle
 \les -\b_1 \phi(y,\bar y,\upsilon , \bar{\upsilon }, \imath ), \\
 &\!\left\langle\2n  \left(\begin{array}{ccc}
               -f(s,x,\th,   \nu,\imath) +f(s, \bar x,  \bar{\th}, \bar \nu, \imath)\\
  \Gamma(s,x,\th,   \nu,\imath) -\G(s, \bar x,  \bar{\th}, \bar \nu, \imath)
  \end{array}\2n \right)
 ,\left(\2n \begin{array}{ccc}
 x - \bar{x}  \\
  \th - \bar{\th}    \\
   \end{array} \2n\right)\2n
  \right\rangle + \(\2n-\frac{\k_x}{2}+\k_{y}\)   \big\langle x - \bar{x}, y-\bar{y} \big\rangle\\
 &\!\les - \b_2\psi(s, x, \bar x, \nu\circ\pi_x^{-1}, \bar \nu\circ\pi_x^{-1}, \imath)
 - \b_1\varphi  (s, \theta, \bar{\theta}, \nu\circ\pi_{\th}^{-1},\bar{\nu}\circ\pi_{\th}^{-1}, \imath);
 \end{aligned}\right.\\
&{\textbf{\rm(C3)$^\prime $-(iii)-Case 2}}\
 \left\{
 \begin{aligned}
 &\! \big\langle \Phi(y, \upsilon,\imath)
 - \Phi(\bar y, \bar \upsilon, \imath), y-\bar{y} \big\rangle
 \ges \b_1 \phi(y,\bar y,\upsilon , \bar{\upsilon }, \imath ), \\
 &\!\left\langle\2n  \left(\begin{array}{ccc}
               -f(s,x,\th,   \nu,\imath) +f(s, \bar x,  \bar{\th}, \bar \nu, \imath)\\
  \Gamma(s,x,\th,   \nu,\imath) -\G(s, \bar x,  \bar{\th}, \bar \nu, \imath)
  \end{array}\2n \right)
 ,\left(\2n \begin{array}{ccc}
 x - \bar{x}  \\
  \th - \bar{\th}    \\
   \end{array} \2n\right)\2n
  \right\rangle  + \(\2n-\frac{\k_x}{2}+\k_{y}\)   \big\langle x - \bar{x}, y-\bar{y} \big\rangle\\
&\!\ges \b_2  \psi(s, x, \bar x, \nu\circ\pi_x^{-1}, \bar \nu\circ\pi_x^{-1}, \imath)
 + \b_1 \varphi(s, \theta, \bar{\theta}, \nu\circ\pi_{\th}^{-1},\bar{\nu}\circ\pi_{\th}^{-1}, \imath).
 \end{aligned}\right.
 \end{aligned}$$

\end{remark}

 \begin{remark}\sl   Assumption $\mathbf{(C3)}$ gives rise to  the following  four distinct combinations,
$$\ba{ll}
\ns\ds \mathbf{(C3.1)}: \ \mathbf{(C3)}\mbox{-(i)-Case 1},   \mathbf{(C3)}\mbox{-(ii), }  \mathbf{(C3)}\mbox{-(iii)-Case 1};\\
\ns\ds \mathbf{(C3.2)}:  \ \mathbf{(C3)}\mbox{-(i)-Case 2},   \mathbf{(C3)}\mbox{-(ii), }  \mathbf{(C3)}\mbox{-(iii)-Case 1};\\
\ns\ds \mathbf{(C3.3)}:  \ \mathbf{(C3)}\mbox{-(i)-Case 1},   \mathbf{(C3)}\mbox{-(ii), }  \mathbf{(C3)}\mbox{-(iii)-Case 2};\\
\ns\ds\mathbf{(C3.4)}:  \ \mathbf{(C3)}\mbox{-(i)-Case 2},   \mathbf{(C3)}\mbox{-(ii), }  \mathbf{(C3)}\mbox{-(iii)-Case 2}.\ea$$
The main results of this section will be established under $\mathbf{(C3.1)}$. The other cases can be handled similarly via analogous verification.
 \end{remark}

 { \begin{remark}\sl
We note that some domination{-}monotonicity conditions in the existing
literature already allow mean-field interactions in expectation or
distributional forms; see, for instance, \cite{Hua-Luo, Tian-Yu}. Here, however, the presence of common noise makes the mean-field terms random,
and hence they have to be described through conditional laws. Accordingly,
conditions (C3) and (C3)$'$ are formulated in terms of the individual
variables, their conditional laws, and the corresponding marginal
distributions. In addition, since the coefficients depend on the regime
process $\alpha_\cd$, the inequalities in (C3) and (C3)$'$ are required to hold
uniformly with respect to the regime. These features are essential for deriving
the a priori estimates for the infinite horizon conditional McKean{-}Vlasov
FBSDEs with Markovian switching. When the common noise and the regime-switching
component are removed, conditions (C3) and (C3)$'$ reduce to the corresponding
domination{-}monotonicity framework in the earlier setting.
\end{remark}}

We now present the main result of this section.
\begin{theorem}\label{Th-FBSDE}\sl
 Suppose the coefficients $(\Phi, f, \G)$  satisfy Assumptions $\mathbf{(C1)}$, $\mathbf{(C2)}$ and $\mathbf{(C3.1)}$. Then there exists a constant  $\varrho > 0$
such that conditional   McKean-Vlasov FBSDEs \eqref{equ-FBSDE-1} admits a unique solution
$(X,\Theta, K)\in  L^{2, \kappa}_\dbF(t,\i;\dbR^n) \times \cL_{\mathbb{F}}^{2, \kappa}(t, \infty)$  for any $\k$ satisfying
\begin{equation}\label{equ-FBSDE-3}
{\underline{\k}<\kappa<\overline{\k},
\qquad \kappa\les\kappa^\ast,}
\quad -\frac{\k_x}{2}+\k_{y}   \les \kappa<  -\frac{\k_x}{2}+\k_{y}  +\varrho.
\end{equation}
Moreover, the following estimate holds,
\begin{equation}\label{equ-FBSDE-4}
\begin{aligned}
&\mathbb{E}\[\mathrm e^{\kappa t}|Y_t|^{2}
+\ds\int_{t}^{\infty}\mathrm e^{\kappa s}(|X_s|^{2}+|\Theta_s|^{2}) \mathrm ds
+\ds\int_{t}^{\infty}\mathrm e^{\kappa s}{|K_s|^{2}\bullet\mathrm d[M]_s}\]\\
&\les  C \mathbb{E}\[\mathrm e^{\kappa t}|\Phi(0, \delta_0, \iota)|^{2}
+ \ds\int_{t}^{\infty}{\mathrm e^{\kappa  s}}(|f(s, {\bf{0}}, \delta_{\bf{0}}, \alpha_{s})|^{2}
+|\Gamma(s, {\bf{0}}, \delta_{\bf{0}}, \alpha_{s})|^{2}) \mathrm ds\],
\end{aligned}
\end{equation}
where $C>0$ is a constant depending on the Lipschitz constants  $\k_x, \k_{x\mu}, \k_{y}, \k_{y\nu}$ and $\b_1$.
\smallskip

Furthermore, let
$(\bar{X},\bar{\Theta}, \bar{K})\in L^{2, \kappa}_\dbF(t,\i;\dbR^n) \times \cL_{\mathbb{F}}^{2, \kappa}(t, \infty)$ be  the solution to   FBSDEs \eqref{equ-FBSDE-1} associated  with  another triple of coefficients  $(\bar{\Phi}, \bar{f},\bar{\Gamma})$ satisfying $\mathbf{(C1)}$, $\mathbf{(C2)}$ and $\mathbf{(C3.1)}$.
Then,
\begin{equation}\label{equ-FBSDE-5}
\begin{aligned}
&\mathbb{E}\[\mathrm e^{\kappa t}|Y_t-\bar{Y}_t|^{2}
+\ds\int_{t}^{\infty}\mathrm e^{\kappa s}(|X_s-\bar{X}_s|^{2}+|\Theta_s-\bar{\Theta}_s|^{2}) \mathrm ds
+\ds\int_{t}^{\infty}\mathrm e^{\kappa s}{|K_s-\bar{K}_s|^{2}\bullet\mathrm d[M]_s}\]\\
&\les   C \mathbb{E}\[\mathrm e^{\kappa t}|\Phi(\bar{Y}_t, \mathcal{L}^1_{\bar{Y}_t},\iota)
\! -\! \bar{\Phi}(\bar{Y}_t, \mathcal{L}^1_{\bar{Y}_t},\iota)|^{2}\\
&\qq \q   +  \ds\int_{t}^{\infty}\!\!\!\!{\mathrm e^{\kappa   s}}\(|f(s, \bar{X}_{s}, \bar{\Theta}_{s}, \mathcal{L}^1_{(\bar{X}_{s},\bar{\Theta}_{s})}, \alpha_{s})
\! -\! \bar{f}(s, \bar{X}_{s},\bar{\Theta}_{s}, \mathcal{L}^1_{(\bar{X}_{s},\bar{\Theta}_{s})}, \alpha_{s})|^{2}\\
&\qq\qq\qq\qq\q+|\Gamma(s, \bar{X}_{s}, \bar{\Theta}_{s}, \mathcal{L}^1_{(\bar{X}_{s},\bar{\Theta}_{s})}, \alpha_{s})
 - \bar{\Gamma}(s, \bar{X}_{s},\bar{\Theta}_{s}, \mathcal{L}^1_{(\bar{X}_{s},\bar{\Theta}_{s})}, \alpha_{s})|^{2}\) \mathrm ds\],
\end{aligned}
\end{equation}
with the same constant   $C$  as in \eqref{equ-FBSDE-4}.

\end{theorem}
 Theorem \ref{Th-FBSDE} will be proved by the method of continuation. To this end, we introduce a family of auxiliary equations parameterized by
  $\lambda\in[0, 1]$. Precisely, for  any $(\xi_0, f_0, {\Gamma}_0)\in L^2_{\cF_t}(\Omega;\dbR^n)\times {{L}^{2, \kappa }_\dbF}(0, \infty;\dbR^n\times\sR)$,  consider
\begin{equation}\label{equ-FBSDE-6}
\left\{\1n
\begin{aligned}
&\!\mathrm dX^{\lambda}_{s}= \!\left[b^{\lambda}(s, X^{\lambda}_{s}, \Theta^{\lambda}_{s}, \mathcal{L}^1_{(X^{\lambda}_{s}, \Theta^{\lambda}_{s})}, \alpha_s)+  b_0(s)\right] \1n \mathrm ds
+ \left[\sigma^{\lambda}(s, X^{\lambda}_{s}, \Theta^{\lambda}_{s}, \mathcal{L}^1_{(X^{\lambda}_{s},\Theta^{\lambda}_{s})}, \alpha_s) + \sigma_0(s)\right]\1n \mathrm dW_{s}\\
&\hskip1.1cm
+\left[\widetilde{\sigma}^{\lambda}(s, X^{\lambda}_{s}, \Theta^{\lambda}_{s}, \mathcal{L}^1_{(X^{\lambda}_{s},\Theta^{\lambda}_{s})}, \alpha_s) +  \widetilde{\sigma}_0(s)\right]   \mathrm d W^0_{s}, \ s\ges t, \\
%
&\!\mathrm dY^{\lambda}_{s}\! = \!- \left[f^{\lambda}(s, X^{\lambda}_{s}, \Theta^{\lambda}_{s}, \mathcal{L}^1_{(X^{\lambda}_{s},\Theta^{\lambda}_{s})}, \alpha_s) +f_0(s)\right]\1n  \mathrm ds
+Z_{s}^{\lambda}  \mathrm dW_{s}+ Z^{0 \lambda}_{s} \mathrm d W^0_{s}
+ {K_{s}^{\lambda} \bullet\mathrm d M_s}, \ s\ges t,  \\
&\!X^{\lambda}_{t} = \Phi^{\lambda}(Y^{\lambda}_{t}, \mathcal{L}^1_{Y^{\lambda}_{t}},\alpha_t)+\xi_0, \quad
\alpha_{t}=\iota,
\end{aligned}\right.
\end{equation}
where  for all $(s, \imath) \in [0, \infty) \times \textbf{M}$,  $(x,\theta,\nu)\in \mathbb{R}^{n}\times\sR \times \mathscr{P}_{2}(\mathbb{R}^n\times\sR)$ and
 $(y, \upsilon)\in \mathbb{R}^{n}\times \mathscr{P}_{2}(\mathbb{R}^n)$,
$$
\left\{
\begin{aligned}
&\!\Phi^{\lambda}(y,\upsilon,\imath):=\lambda\Phi(y,\upsilon,\imath),\\
&\!b^{\lambda}(s,x,\theta,\nu,\imath):=\lambda b(s,x,\theta,\nu,\imath) +\frac{1}{2}(1-\lambda)\k_x  x,\\
&\!\sigma^{\lambda}(s,x,\theta,\nu,\imath):=\lambda \sigma(s,x,\theta,\nu,\imath),\\
&\!\widetilde{\sigma}^{\lambda}(s,x,\theta,\nu,\imath):=\lambda \widetilde{\sigma}(s,x,\theta,\nu,\imath),\\
&\!f^{\lambda}(s,x,\theta,\nu,\imath):=\lambda f(s,x,\theta,\nu,\imath) +(1-\lambda)\k_y  y.
\end{aligned}\right.
$$

 For any $\lambda \in[0, 1]$, we first establish the a priori estimate for McKean-Vlasov FBSDEs \eqref{equ-FBSDE-6}, which  is key  when applying the continuation method.

\begin{proposition}\label{Pr-FBSDE-1}\sl

Let $(\Phi, f, \G)$  satisfy $\mathbf{(C1)}$, $\mathbf{(C2)}$,  $\mathbf{(C3.1)}$  and let {$\underline\k<\kappa<\overline\k$
and $\kappa\les\kappa^\ast$}.  For any  $(\xi_0, f_0, \Gamma_0)$, $(\bar{\xi}_0, \bar{f}_0, \bar{\Gamma}_0) \in L^2_{\cF_t}(\O;\dbR^n)\times {L_{\mathbb{F}}^{2, \k}}(0, \infty ; \dbR^n\times  \sR )   $, denote  $(X^{\lambda},\Theta^{\lambda},K^{\lambda})$, $(\bar{X}^{\lambda},\bar{\Theta}^{\lambda},\bar{K}^{\lambda})\in L^{2, \kappa}_\dbF(t,\i;\dbR^n) \times \cL_{\mathbb{F}}^{2, \kappa}(t, \infty)$ by the solution  to conditional   McKean-Vlasov FBSDEs \eqref{equ-FBSDE-6} associated with $(\Phi^{\lambda}+\xi_0, f^{\lambda}+f_0, \Gamma^{\lambda}+\Gamma_0 )$ and
$(\Phi^{\lambda}+\bar{\xi}_0, f^{\lambda}+\bar{f}_0, \Gamma^{\lambda}+\bar{\Gamma}_0)$, resp.
Then, there exists some constant $\varrho > 0$ such that, for all $\k\in [ -\frac{\k_x}{2}+\k_{y} , -\frac{\k_x}{2}+\k_{y} +\varrho) $  and $\l\in [0,1]$,
\begin{equation}\label{equ-FBSDE-10}
\begin{aligned}
&\mathbb{E}\[\mathrm e^{\kappa t}|Y^{\lambda}_t-\bar{Y}^{\lambda}_t|^{2}
+\ds\int_{t}^{\infty}\mathrm e^{\kappa s}(|X^{\lambda}_s-\bar{X}^{\lambda}_s|^{2}
+|\Theta^{\lambda}_s-\bar{\Theta}^{\lambda}_s|^{2}) \mathrm ds
+\ds\int_{t}^{\infty}\mathrm e^{\kappa s}{|K^{\lambda}_s-\bar{K}^{\lambda}_s|^{2} \bullet \mathrm d[M]_s}\]\\
&\les  C\mathbb{E}\[\mathrm e^{\kappa t}|\xi_0- \bar{\xi}_0|^{2}
+\ds\int_{t}^{\infty}\mathrm e^{{ \kappa  s}}(|f_0(s)-\bar{f}_0(s)|^{2}+|\Gamma_0(s)-\bar{\Gamma}_0(s)|^{2}) \mathrm ds\],
\end{aligned}
\end{equation}
where $C>0$ is a constant depending on the Lipschitz constants  $\k_x, \k_{x\mu} $,  $\k_{y}, \k_{y\nu}$ and $\b_1$.
\end{proposition}
\begin{proof}
For convenience, we set
$(\widehat{X}^{\lambda}, \widehat{Y}^{\lambda},\widehat{Z}^{\lambda},\widehat{ Z}^{0\lambda}, \widehat{K}^{\lambda}):=
(X^{\lambda},Y^{\lambda},Z^{\lambda}, Z^{0  \lambda}, K^{\lambda})-(\bar{X}^{\lambda},\bar{Y}^{\lambda},\bar{Z}^{\lambda},\bar{ Z}^{0  \lambda},$ $\bar{K}^{\lambda})$.
{
By $\mathbf{(C1)}$, $\kappa\les\kappa^\ast$, and the prescribed
source and solution spaces, the inhomogeneous terms in the frozen
equations below belong to $L_{\mathbb F}^{2,\kappa}$.
Thus, Corollary \ref{Cor-SDE-1} and Proposition \ref{Pr-BSDE-1}
apply with data-integrability exponent $\kappa$.}
 For $\widehat{X}^{\lambda}$, applying   the estimate \eqref{SDE-Cont} with $p=2$, when $\k<\overline\k$, we obtain
\begin{equation}\label{equ-FBSDE-11}
\begin{aligned}
&\mathbb{E} \[\ds\int_{t}^{\infty}\mathrm e^{\kappa s }|\widehat{X}^{\lambda}_{s}|^{2} \mathrm ds\]
 %
   \les K_1\mathbb{E}\[\mathrm e^{\kappa t}\big|\lambda\widehat{\Phi}(t)   +  \h\xi_0 \big|^{2} +\int_{t}^{\infty}\mathrm e^{\kappa s }  | \l\widehat{\G}(s)+\widehat \G_0(s)|^2 \mathrm ds
  \],
\end{aligned}
\end{equation}
where $K_1$ is a constant and
$$
\left\{
\begin{aligned}
&\!\widehat{\Phi}(t):= \Phi(Y^{\lambda}_t, \mathcal{L}^1_{Y^{\lambda}_t}, \alpha_{t})
-\Phi(\bar{Y}^{\lambda}_{t}, \mathcal{L}^1_{\bar{Y}^{\lambda}_{t}}, \alpha_{t})  ,\\
&\!\widehat{ \mathbf{h}}(s)  := \mathbf{h}(s, \bar{X}^{\lambda}_{s}, \Theta^{\lambda}_{s}, \mathcal{L}^1_{(\bar{X}^{\lambda}_{s}, \Theta^{\lambda}_{s})},\alpha_{s})
-\mathbf{h}(s, \bar{X}^{\lambda}_{s}, \bar{\Theta}^{\lambda}_{s}, \mathcal{L}^1_{(\bar{X}^{\lambda}_{s}, \bar{\Theta}^{\lambda}_{s})}, \alpha_{s}),\\
&\! \widehat{\mathbf{h}}_0(s):=\mathbf{h}_0(s)-\bar{\mathbf{h}}_0(s) ,\q \h\xi_0:=\xi_0-\bar\xi_0,\\
&\h\Gamma(s):=( \h b^{\,\top}(s),\h \sigma^{\top}(s), \h{\widetilde{\sigma}  }^{\top}(s))^{\top},\q \h\Gamma_0(s):=( \h b^{\,\top}_0(s),\h \sigma^{\top}_0(s), \h{\widetilde{\sigma}  }_0^{\top}(s))^{\top},
\end{aligned}\right.
$$
with  $\mathbf{h}=b, \sigma, \widetilde{\sigma}$ and $\mathbf{h}_0=b_0, \sigma_0, \widetilde{\sigma}_0$, resp.

Then, applying   {\textbf{\rm(C3)-(ii)}} (or \rm(C3)$^\prime $-(ii)) to \rf{equ-FBSDE-11}, we get
\begin{equation}\label{equ-FBSDE-15}
\begin{aligned}
&\mathbb{E} \[\ds\int_{t}^{\infty}\mathrm e^{\kappa s }|\widehat{X}^{\lambda}_{s}|^{2} \mathrm ds\]\les K_{2} \mathbb{E}\[\mathrm e^{\kappa t}\(\lambda \phi(Y^{\lambda}_t,\bar{Y}^{\lambda}_t, \mathcal{L}^1_{Y^{\lambda}_t},\mathcal{L}^1_{\bar{Y}^{\lambda}_t},\alpha_t)+|\h\xi_0 |^{2}\)\\
&\qq\qq\qq\qq\qq\qq\
+\ds\int_{t}^{\infty}\!\! \mathrm e^{\kappa s}
\(\lambda  \varphi(s, \Theta^{\lambda}_s, \bar{\Theta}^{\lambda}_s, \mathcal{L}^1_{\Theta^{\lambda}_s}, \mathcal{L}^1_{\bar{\Theta}^{\lambda}_s},\alpha_s)+ |\widehat{\Gamma}_0(s)|^2\)\, \mathrm ds\],
\end{aligned}
\end{equation}
where $K_2>0$  depends on $\b_1$.

Next, for $( \widehat{Y}^{\lambda},\widehat{Z}^{\lambda},\widehat{ Z}^{0\lambda}, \widehat{K}^{\lambda})$,  we apply the  estimate \eqref{L-BSDE-2}. When $\k\in (\underline{\k},\overline{\k})$,
 \begin{equation}\label{equ-FBSDE-13}
 \begin{aligned}
 &\dbE\[ \mathrm e^{\k t} |\widehat{Y}^{\lambda}_t|^2
+ \int_{t}^{\infty} \mathrm e^{\k  s}|\widehat{\Theta}^{\lambda}_s|^{2}\mathrm ds
 +\int_{t}^{\infty} \mathrm e^{\k  s}{|\widehat{K}^{\lambda}_s|^{2 }\bullet\mathrm d[M]_s}\]\\
 & \les  \mathbb{E}\[\mathrm e^{\kappa t}\big|\lambda\widehat{\Phi}(t)   +  \h\xi_0 \big|^{2} +\int_{t}^{\infty}\mathrm e^{\kappa s } \(| \l\widehat{\G}(s)+\widehat \G_0(s)|^2) +| \l\widehat{f}(s)+\widehat {f}_0(s)|^2\)\mathrm ds
  \] \\
&\les K_3  \mathbb{E}\[\mathrm e^{\kappa t}\(\!\lambda\phi(Y^{\lambda}_t,\bar{Y}^{\lambda}_t, \mathcal{L}^1_{Y^{\lambda}_t},\mathcal{L}^1_{\bar{Y}^{\lambda}_t},\alpha_t)\!+\!|\xi_0-\bar\xi_0|^{2}\!\)\\
&\hskip1.5cm
+\ds\int_{t}^{\infty} \!\!\!  \mathrm e^{\kappa s}
\(\lambda\varphi(s, \Theta^{\lambda}_s, \bar{\Theta}^{\lambda}_s, \mathcal{L}^1_{\Theta^{\lambda}_s}, \mathcal{L}^1_{\bar{\Theta}^{\lambda}_s},\alpha_s)+|\widehat{f}_0(s)|^{2}
+|\widehat{\Gamma}_0(s)|^{2}\)\, \mathrm ds \],
\end{aligned}
\end{equation}
where
$$
\left\{
\begin{aligned}
&\!\widehat{ f}(s)  := f(s, {X}^{\lambda}_{s}, \bar\Theta^{\lambda}_{s}, \mathcal{L}^1_{( {X}^{\lambda}_{s}, \bar\Theta^{\lambda}_{s})},\alpha_{s})
-f(s, \bar{X}^{\lambda}_{s}, \bar{\Theta}^{\lambda}_{s}, \mathcal{L}^1_{(\bar{X}^{\lambda}_{s}, \bar{\Theta}^{\lambda}_{s})}, \alpha_{s}),\\
&\! \widehat{f}_0(s):=f_0(s)-\bar{f}_0(s). \\
\end{aligned}\right.
$$
Combining \eqref{equ-FBSDE-15} and \eqref{equ-FBSDE-13}, we obtain
\begin{equation}\label{equ-FBSDE-17}
\begin{aligned}
&\dbE\[ \mathrm e^{\k t} |\widehat{Y}^{\lambda}_t|^2
\!+\!\int_{t}^{\infty} \mathrm e^{\k  s} (|\widehat{X}^{\lambda}_s|^{2}
+ |\widehat{\Theta}^{\lambda}_s|^{2})\mathrm ds
+\int_{t}^{\infty} \mathrm e^{\k  s}{|\widehat{K}^{\lambda}_s|^{2 }\bullet\mathrm d[M]_s}\]\\
&\les \!K_{4} \mathbb{E}\[\mathrm e^{\kappa t}\(\lambda\phi(Y^{\lambda}_t,\bar{Y}^{\lambda}_t, \mathcal{L}^1_{Y^{\lambda}_t},\mathcal{L}^1_{\bar{Y}^{\lambda}_t},\alpha_t)\!+\!|\h\xi_0 |^{2}\!\)\\
&\hskip1.5cm
\!+\!\ds\int_{t}^{\infty}\!\!\!\! \mathrm e^{\kappa s}
\(\!\lambda\varphi(s, \Theta^{\lambda}_s, \bar{\Theta}^{\lambda}_s, \mathcal{L}^1_{\Theta^{\lambda}_s}, \mathcal{L}^1_{\bar{\Theta}^{\lambda}_s},\alpha_s)+|\widehat{f}_0(s)|^{2}\!+\!|\widehat{\Gamma}_0(s)|^{2}\!\) \, \mathrm ds\].
\end{aligned}
\end{equation}

Next, applying It\^o's formula to $ \mathrm e^{\kappa s}\langle \widehat{X}^{\lambda}_s, \widehat{Y}^{\lambda}_s \rangle$ on $ [t, T] $, we get
$$
\begin{aligned}
&\dbE\[\mathrm e^{\kappa T}\langle \widehat{X}^{\lambda}_{T}, \widehat{Y}^{\lambda}_{T} \rangle\]
=\dbE\[\mathrm e^{\kappa t}\langle  \lambda \h{\Phi}(t)+\widehat \xi_0 , \widehat{Y}^{\lambda}_{t} \rangle\\
&\qq\ +\ds\int_{t}^{T}\mathrm e^{\kappa s}\(\lambda \lan\G(s, X^{\lambda}_s,\Theta^{\lambda}_s, \mathcal{L}^1_{(X^{\lambda}_s,\Theta^{\lambda}_s)}, \alpha_s)
-\G(s, \bar{X}^{\lambda}_s,\bar{\Theta}^{\lambda}_s, \mathcal{L}^1_{(\bar{X}^{\lambda}_s,\bar{\Theta}^{\lambda}_s)}, \alpha_s) , \widehat{\Theta}^{\lambda}_{s}\rangle\\
&\qq\qq\qq\q-\lambda \langle f(s, X^{\lambda}_s,\Theta^{\lambda}_s, \mathcal{L}^1_{(X^{\lambda}_s,\Theta^{\lambda}_s)}, \alpha_s)
-f(s, \bar{X}^{\lambda}_s,\bar{\Theta}^{\lambda}_s, \mathcal{L}^1_{(\bar{X}^{\lambda}_s,\bar{\Theta}^{\lambda}_s)}, \alpha_s), \widehat{X}^{\lambda}_{s}\rangle\\
&\qq\qq\qq\q
+\big [\kappa-(1-\lambda)(-\frac{\kappa_{x}}{2}+\kappa_{y})\big] \langle \widehat{X}^{\lambda}_{s}, \widehat{Y}^{\lambda}_{s}\rangle
-\langle \widehat{f}_0(s), \widehat{X}^{\lambda}_{s}\rangle
+\langle \widehat{\Gamma}_0(s), \widehat{\Theta}^{\lambda}_{s}\rangle\)\mathrm ds\].
\end{aligned}
$$
Using {\rm(C3)-(iii)-Case 1}  (or \rm(C3)$^\prime $-(iii)-Case 1) and
$\lim\limits_{T \to \infty}\big | \mathbb{E}[\mathrm e^{\kappa T} \langle \widehat{X}^{\lambda}_T, \widehat{Y}^{\lambda}_T \rangle ]\big|=0 $, we obtain
\begin{equation}\label{equ-FBSDE-14}
\begin{aligned}
&\lambda \b_1 \dbE\[\mathrm e^{\kappa t} \phi(Y^{\lambda}_t,\bar{Y}^{\lambda}_t, \mathcal{L}^1_{Y^{\lambda}_t},\mathcal{L}^1_{\bar{Y}^{\lambda}_t},\alpha_t)
+\ds\int_{t}^{\infty} \mathrm e^{\kappa s} \varphi(s, \Theta^{\lambda}_s, \bar{\Theta}^{\lambda}_s, \mathcal{L}^1_{\Theta^{\lambda}_s}, \mathcal{L}^1_{\bar{\Theta}^{\lambda}_s},\alpha_s) \, \mathrm ds \] \\
& \les  \mathbb{E} \[\mathrm e^{\kappa t} \langle\h\xi_0 , \widehat{Y}^{\lambda}_{t}\rangle
+\ds\int_{t}^{\infty}\mathrm e^{\kappa s} \( (\kappa+\frac{\k_x}{2}-\k_{y}) \langle\widehat{X}^{\lambda}_{s}, \widehat{Y}^{\lambda}_{s}\rangle-\langle \widehat{f}_0(s), \widehat{X}^{\lambda}_{s}\rangle
+\langle \widehat{\Gamma}_0(s), \widehat{\Theta}^{\lambda}_{s}\rangle\)\mathrm ds\].
\end{aligned}
\end{equation}

Since $\kappa \ges -\frac{\k_x}{2}+\k_{y}$, substituting \eqref{equ-FBSDE-14} into \eqref{equ-FBSDE-17} leads to
\begin{equation}\label{equ-K-5}
\begin{aligned}
&\dbE\[ \mathrm e^{\k t} |\widehat{Y}^{\lambda}_t|^2
+\int_{t}^{\infty} \mathrm e^{\k  s} (|\widehat{X}^{\lambda}_s|^{2}+|\widehat{\Theta}^{\lambda}_s|^{2})\mathrm ds
+\int_{t}^{\infty} \mathrm e^{\k  s}{|\widehat{K}^{\lambda}_s|^{2 }\bullet\mathrm d[M]_s}\]\\
&\les  \frac{K_{4}}{\b_1} \mathbb{E} \[\mathrm e^{\kappa t} \langle\h\xi_0, \widehat{Y}^{\lambda}_{t}\rangle
+\ds\int_{t}^{\infty}\mathrm e^{\kappa s} \( (\kappa+\frac{\k_x}{2}-\k_{y}) \langle\widehat{X}^{\lambda}_{s}, \widehat{Y}^{\lambda}_{s}\rangle-\langle \widehat{f}_0(s), \widehat{X}^{\lambda}_{s}\rangle
+\langle \widehat{\Gamma}_0(s), \widehat{\Theta}^{\lambda}_{s}\rangle\)\mathrm ds\]\\
&\q+K_{4} \mathbb{E}\[\mathrm e^{\kappa t} |\h\xi_0 |^{2}
\!+\!\ds\int_{t}^{\infty}  \mathrm e^{\kappa s}
\( |\widehat{f}_0(s)|^{2}\!+\!|\widehat{\Gamma}_0(s)|^{2}\!\) \, \mathrm ds\] \\
&\les   \mathbb{E}\[\frac{1}{2}\mathrm e^{\kappa t}|\widehat{Y}^{\lambda}_{t}|^{2}
+ (\frac{1}{4} + K_{5})\ds\int_{t}^{\infty}
\mathrm e^{\kappa s}\big(|\widehat{X}^{\lambda}_{s}|^{2}+|\widehat{\Theta}^{\lambda}_{s}|^{2}\big) \mathrm ds \]\\
&\qq \qq + {K_{6}} \mathbb{E}\[\mathrm e^{\kappa t}|\h\xi_0 |^{2}
+  \ds\int_{t}^{\infty}\mathrm e^{\kappa s}\big(|\widehat{f}_0(s)|^{2}+|\widehat{\Gamma}_0(s)|^{2}\big) \mathrm ds \],
\end{aligned}
\end{equation}
\no where
$K_{5}:= \frac{K_{4}( \kappa+\frac{\k_x}{2}-\k_{y}) } {2 \b_1}$ and
$K_{6}:=  K_{4}\left(1+\frac{K_{4}}{\b_1^{2}}\right)$.

Finally,  choosing $\ds\varrho=\frac{3\b_1}{2 K_{4}}$,  when $\kappa< -\frac{\k_x}{2}+\k_{y} +\varrho$, we get $K_5<\frac 34$.   Then the desired estimate \eqref{equ-FBSDE-10}  follows from \rf{equ-K-5}.
\end{proof}

With the help of the a priori estimate \eqref{equ-FBSDE-10}, we establish the following continuation  result.
\begin{lemma}\label{Le-FBSDE-2}\sl
{Under all the assumptions and weight restrictions of Proposition
\ref{Pr-FBSDE-1}},
%
 there exists some constant $\delta_{0}>0$ independent of $\lambda$, such that
 if for $ \lambda_{0} \in [0, 1)$,  McKean-Vlasov FBSDEs \eqref{equ-FBSDE-6}
is uniquely solvable in $L^{2, \kappa}_\dbF(t,\i;\dbR^n) \times \cL_{\mathbb{F}}^{2, \kappa}(t, \infty) $ with any $(\xi_0, f_0, {\Gamma}_0)\in L^2_{\cF_t}(\O;\dbR^n)\times {L_{\mathbb{F}}^{2,   \k}}(0, \infty ; \dbR^n\times  \sR )$,
then the same holds for all   $\lambda \in (\lambda_{0}, (\lambda_{0}+\delta_{0})\wedge 1]$.
\end{lemma}
\begin{proof}
Let $\delta_{0} >0$ be some constant to be  determined.
Suppose $ (  X,\Theta, K)\in  L^{2, \kappa}_\dbF(t,\i;\dbR^n) \times \cL_{\mathbb{F}}^{2, \kappa}(t, \infty)$ solves \eqref{equ-FBSDE-6} with $\l=\l_0$ and   $(\xi_0, f_0, {\Gamma}_0)\in L^2_{\cF_t}(\O;\dbR^n)\times {L_{\mathbb{F}}^{2,  \k}}(0, \infty ; \dbR^n\times  \sR )$.
For any $\delta \in [0, \delta_{0}\wedge(1-\lambda_{0})]$,
we consider  the following McKean-Vlasov FBSDEs,
\begin{equation}\label{equ-FBSDE-18}
\left\{\2n
\begin{aligned}
& \mathrm d\mathbf{X}  _{s}\!=\!\[b^{\lambda_{0}}(s, \mathbf{X}  _{s},\mathbf{\Theta}_{s},
\mathcal{L}^1_{(\mathbf{X} _{s},\mathbf{\Theta}_{s})}, \alpha_{s})+b_1(s)\]  \mathrm ds
%
\!+\!\[\sigma^{\lambda_{0}}(s, \mathbf{X} _{s},\mathbf{\Theta}_{s},
\mathcal{L}^1_{(\mathbf{X}  _{s},\mathbf{\Theta}_{s})}, \alpha_{s})
+ \sigma_1(s)\]  \mathrm dW_{s}\!\!\!\!\!\!\!\!\\
&\hskip1cm
+\!\[\widetilde{\sigma}^{\lambda_{0}}(s, \mathbf{X}  _{s},\mathbf{\Theta}_{s},
\mathcal{L}^1_{(\mathbf{X}  _{s},\mathbf{\Theta}_{s})}, \alpha_{s})
+\widetilde{\sigma} _1(s)\]  \mathrm d W^0 _{s},\ s \ges t,   \\
& \mathrm d\mathbf{Y}_s\!=\!-\[f^{\lambda_{0}}(s, \mathbf{X}  _{s},\mathbf{\Theta}_{s},
\mathcal{L}^1_{(\mathbf{X}_{s},\mathbf{\Theta}_{s})}, \alpha_{s})+f_1(s)\] \mathrm ds
+\mathbf{Z}_{s}\, \mathrm dW_{s}+\mathbf{Z}^0_{s}  \mathrm d W^0 _{s}
+{\mathbf{K}_{s} \bullet\mathrm d M_s}, \ s \ges t, \\
& \mathbf{X}  _t=\Phi^{\lambda_{0}}(\mathbf{Y}_t, \mathcal{L}^1_{\mathbf{Y}_t},\alpha_t) + \xi_1,\quad
\alpha_t=\iota,
\end{aligned}\right.
\end{equation}
\no where
%
%
%
%
%
%
\begin{equation}\label{equ-FBSDE-19}
\left\{\2n
\begin{aligned}
& \xi_1:=\delta\Phi(Y_t, \mathcal{L}^1_{Y_t},\alpha_t)+ \xi_0,\\
& b_1(s):=\delta b(s,X_{s},\Theta_{s}, \mathcal{L}^1_{(X_{s},\Theta_{s})}, \alpha_s)
 -\frac{1}{2}\delta \k_x  X_{s}+b_0(s), \\
& \sigma_1(s):=\delta \sigma (s,X_{s},\Theta_{s}, \mathcal{L}^1_{(X_{s},\Theta_{s})}, \alpha_s)
+\sigma_0(s), \\
& \widetilde{\sigma}_1(s):=\delta \widetilde{\sigma} (s,X_{s},\Theta_{s}, \mathcal{L}^1_{(X_{s},\Theta_{s})}, \alpha_s)+\widetilde{\sigma}_0(s), \\
& f_1(s):=\delta f(s,X_{s},\Theta_{s}, \mathcal{L}^1_{(X_{s},\Theta_{s})}, \alpha_s)
 -\delta \k_{y}  Y_{s}+f_0(s). \\
\end{aligned}\right.
\end{equation}
{
By $\mathbf{(C1)}$, $\kappa\les\kappa^\ast$,
each input $(Y_t,X,\Theta,K)$ in the fixed-point space yields
$
(\xi_1,f_1,\Gamma_1)\in
L^2_{\cF_t}(\O;\dbR^n)\times
L_{\mathbb F}^{2,\kappa}(0,\infty;\dbR^n\times\sR),
$
where $f_1$ and $\Gamma_1$ are extended by zero on $[0,t)$.}
Hence \rf{equ-FBSDE-18} admits a unique solution $( \mathbf X ,\mathbf{\Theta}, \mathbf{K})\in L^{2, \kappa}_\dbF(t,\i;\dbR^n) \times \cL_{\mathbb{F}}^{2, \kappa}(t, \infty)$.
 This defines the following  mapping  from  $
L_{\mathcal{F}_{t}}^{2}(\Omega; \mathbb{R}^{n}) \times L^{2, \kappa}_\dbF(t,\i;\dbR^n) \times \cL_{\mathbb{F}}^{2, \kappa}(t, \infty)$ into itself,
$$
\mathcal{J}_{\lambda_{0} + \delta}(Y_t, X,\Theta, K)=(\mathbf{Y}_t, \mathbf X  ,\mathbf{\Theta}, \mathbf{K}) .
$$


Next, we show that   $\mathcal{J}_{\lambda_{0} + \delta}$ is contractive
for sufficiently small  $\delta$.
Take  any $(Y_t, X, \Theta, K)$, $(\bar{Y}_t, \bar{X},\bar{\Theta}, \bar{K})\in
L_{\mathcal{F}_{t}}^{2}(\Omega; \mathbb{R}^{n})\times L^{2, \kappa}_\dbF(t,\i;\dbR^n) \times \cL_{\mathbb{F}}^{2, \kappa}(t, \infty)$
and let  $$(\mathbf{Y}_t, \mathbf X ,\mathbf{\Theta}, \mathbf{K}) = \mathcal{J}_{\lambda_{0}+\delta}(Y_t, X,\Theta,K),\q (\bar{\mathbf{Y}}_t, \bar{ \mathbf X  },\bar{\mathbf{\Theta}},\bar{\mathbf{K}}) = \mathcal{J}_{\lambda_{0}+\delta}(\bar{Y}_t, \bar{X},\bar{\Theta}, \bar{K}).$$
By the a priori estimate \eqref{equ-FBSDE-10}  and
 Lipschitz continuity of $(\Phi, f, \Gamma)$,  we have
$$
\begin{aligned}
&\mathbb{E}\[\mathrm e^{\kappa t}\big|\mathbf{Y}_{t}-\bar{\mathbf{Y}}_{t}\big|^{2}
+\ds\int_{t}^{\infty}\mathrm e^{\kappa s}
(\big| \mathbf X  _s-\bar{ \mathbf X }_s\big|^{2}+\big|\mathbf{\Theta}_s-\bar{\mathbf{\Theta}}_s\big|^{2})\mathrm ds
+\ds\int_{t}^{\infty}\mathrm e^{\kappa s}{\big|\mathbf{K}_s-\bar{\mathbf{K}}_s\big|^{2}\bullet\mathrm d[M]_s}\]\\
%
&\les C \delta^{2} \mathbb{E}\[\mathrm e^{\kappa t}\big|Y_{t}-\bar{Y}_{t}\big|^{2}
+\ds\int_{t}^{\infty}\mathrm e^{\kappa s}
(\big|X_s-\bar{X}_s\big|^{2}+\big|\Theta_s-\bar{\Theta}_s\big|^{2})\mathrm ds
+\ds\int_{t}^{\infty}\mathrm e^{\kappa s}{\big|K_s-\bar{K}_s\big|^{2}\bullet\mathrm d[M]_s}\],
\end{aligned}
$$
%
%
%
%
{where $C>0$ may depend on the fixed weight $\kappa$ and on
the relevant structural parameters, but is independent of $\lambda_0$
and $\delta$.}
Choosing $\delta_{0}=\frac{1}{2 \sqrt{C}}$,  we see that for any  $\delta \in [0, \delta_{0}\wedge(1-\lambda_{0})]$,
the mapping $\mathcal{J}_{\lambda_{0}+\delta}$ is contractive. Hence, there exists a unique fixed point
$(\bar {\mathbf{Y}}_t, \bar { \mathbf X  },\bar{ \mathbf{\Theta}},\bar { \mathbf{K}}) \in
L_{\mathcal{F}_{t}}^{2}(\Omega; \mathbb{R}^{n}) \times L^{2, \kappa}_\dbF(t,\i;\dbR^n) \times \cL_{\mathbb{F}}^{2, \kappa}(t, \infty)$ such that
$$(\bar {\mathbf{Y}}_t, \bar { \mathbf X },\bar{ \mathbf{\Theta}},\bar { \mathbf{K}}) = \mathcal{J}_{\lambda_{0}+\delta}(\bar {\mathbf{Y}}_t, \bar { \mathbf X },\bar{ \mathbf{\Theta}},\bar { \mathbf{K}}),$$
and $(\bar { \mathbf X },\bar{ \mathbf{\Theta}},\bar { \mathbf{K}})$ is thereby the unique solution to McKean-Vlasov FBSDEs  \eqref{equ-FBSDE-6}  with $\lambda=\lambda_0+\delta$.
%
\end{proof}

Now, we give the   proof of  Theorem \ref{Th-FBSDE} as follows.

\begin{proof}[\rm\text{\textbf{Proof of Theorem \ref{Th-FBSDE}}}]

We first note that \eqref{equ-FBSDE-3} implies, separately,
{$
2\kappa_y<\kappa<-\kappa_x,
$ $ \kappa\les\kappa^\ast.
 $}
Setting  $\lambda = 0$ in  \eqref{equ-FBSDE-6} yields
\begin{equation}\label{equ-FBSDE-6-0}
\left\{\2n
\begin{aligned}
& \mathrm dX^{0}_{s}= \big( \frac{1}{2}\k_x  X^{0}_{s} +  b_0(s)\big) \mathrm ds
+ \sigma_0(s) \mathrm dW_{s} +\widetilde{\sigma}_0(s) \mathrm d W^0_{s},
  \q s\ges t,\\
& \mathrm dY^{0}_{s}  =   -\big(\k_{y} Y^{0}_{s}+f_0(s)\big)\mathrm ds +Z_{s}^{0} \, \mathrm dW_{s}+ Z^{00}_{s}\, \mathrm d W^0_{s}+ { K_{s}^{0} \bullet \mathrm d M_s}, \q s\ges t,  \\
& X^{0}_{t} = \xi_0, \quad
\alpha_{t}=\iota.
\end{aligned}\right.
\end{equation}
{By Corollary \ref{Cor-SDE-1} and Theorem
\ref{Th-BSDE}, applied with their data-integrability exponent chosen as
$\kappa$, the forward and backward equations in
\rf{equ-FBSDE-6-0} can be solved separately because
$2\kappa_y<\kappa<-\kappa_x$. This application uses the data condition
at equality and therefore applies to every
$(\xi_0,f_0,\Gamma_0)\in L^2_{\cF_t}(\O;\dbR^n)\times
L_{\mathbb F}^{2,\kappa}(0,\infty;\dbR^n\times\sR)$.}
Subsequently,
 Lemma  \ref{Le-FBSDE-2} then implies that  the McKean-Vlasov FBSDEs \eqref{equ-FBSDE-6}  with $\l=1$ is uniquely solvable.
In particular, choosing $(\xi_0,f_0, \Gamma_0)=(0,0,0)$ reduces  system  \eqref{equ-FBSDE-6}  to the original FBSDEs \eqref{equ-FBSDE-1}, which therefore admits a unique solution.

Further, the estimate \eqref{equ-FBSDE-5} follows from \rf{equ-FBSDE-10} by taking $\lambda=1$, $(\xi_0, f_0,\Gamma_0)=(0, 0, 0)$
and
$$
\left\{
\begin{aligned}
&\!\bar{\xi}_0=\bar{\Phi}(\bar{Y}_{t}, \mathcal{L}^1_{\bar{Y}_{t}},\alpha_t)
-\Phi(\bar{Y}_{t}, \mathcal{L}^1_{\bar{Y}_{t}},\alpha_t),\\
&\!\bar{f}_0=\bar{f}(\cdot, \bar{X},\bar{\Theta}, \mathcal{L}^1_{(\bar{X},\bar{\Theta})}, \alpha)
-f(\cdot, \bar{X},\bar{\Theta}, \mathcal{L}^1_{(\bar{X},\bar{\Theta})}, \alpha), \\
&\!\bar{\Gamma}_0=\bar{\Gamma}(\cdot, \bar{X},\bar{\Theta},  \mathcal{L}^1_{(\bar{X},\bar{\Theta})}, \alpha)-\Gamma(\cdot, \bar{X},\bar{\Theta},\mathcal{L}^1_{(\bar{X},\bar{\Theta})}, \alpha),
\end{aligned}\right.
$$
as in Proposition \ref{Pr-FBSDE-1}.
Finally, taking $(\bar{\Phi},\bar{f}, \bar{\Gamma})=(0, 0, 0)$ in \eqref{equ-FBSDE-5}  gives the estimate \eqref{equ-FBSDE-4}.

\end{proof}

\section{Infinite Horizon Linear Quadratic Mean Field {Control and Game} Problems with Common Noise and Regime-Switching} 
\label{sec:LQ}

{
The purpose of this section is to apply the abstract well-posedness result
established in Theorem \ref{Th-FBSDE} to infinite horizon LQ mean field control
and game problems. In the LQ framework, the state equation, the adjoint equation
and the optimality condition together lead to a Hamiltonian system, which is a
fully coupled linear FBSDEs of conditional McKean{-}Vlasov type with Markovian
switching. This Hamiltonian system falls within the nonlinear framework studied
in Section \ref{MKV-FBSDE}. Therefore, its solvability can be obtained from
Theorem \ref{Th-FBSDE}, once the corresponding abstract assumptions are verified.
In the present LQ setting, these assumptions can be expressed in terms of
structural conditions on the coefficients of the controlled state equation and
on the weighting matrices in the cost functionals.}

We begin with  an
infinite horizon linear-quadratic mean field control problem.

\subsection{Infinite horizon  linear quadratic mean-field  control problems} 
\label{sub:MF-LQ-1}
 For any initial  $(t, x_{t}, \iota)\in \sD^2$, consider
\begin{equation}\label{equ-MFLQ-2}
\left\{
\begin{aligned}
& \mathrm dX_s=\( A (s, \alpha_{s}) X_s + \bar{A} (s, \alpha_{s})\mathbb{E}^{0}_{s}{[X_s]} + B (s, \alpha_{s}) u_{s}+\bar{B} (s, \alpha_{s})\mathbb{E}^{0}_{s}{[u_{s}]}+b_s\) \, \mathrm ds\\
& \ + \(C (s, \alpha_{s}) X_s+\bar{C}  (s, \alpha_{s})\mathbb{E}^0_{s}{[X_s]}
+D  (s, \alpha_{s}) u_{s}
+\bar{D}  (s, \alpha_{s})\mathbb{E}^0_{s}{[u_{s}]}+\si_s\) \, \mathrm dW_{s} \\
&\ + \(M (s, \alpha_{s}) X_s+\bar{M}  (s, \alpha_{s})\mathbb{E}^0_{s}{[X_s]}+N (s, \alpha_{s}) u_{s}
+\bar{N}  (s, \alpha_{s})\mathbb{E}^0_{s}{[u_{s}]}+\g_s\) \, \mathrm d W^0 _{s}, \ s\ges t,\\
& X_t=x_{t}, \quad
\alpha_t=\iota,
\end{aligned}\right.
\end{equation}
where
 $\dbE^0_s[\cd]$ is short for the conditional expectation $\mathbb{E}[\cd\mid{\mathcal{F}^{0}_s}]$.
The cost  is introduced as follows,
\begin{equation}\label{equ-MFLQ-1}
\begin{array}{lll}
&J^{\kappa}(t, x_{t},\iota ; u_\cd)=
\dfrac{1}{2} \mathbb{E} \ds\int_{t}^{\infty} \mathrm e^{\kappa s}g(s,\omega^0,X_s,\dbE_s^0[X_s],u_s,\dbE_s^0[u_s],\a_s)\mathrm ds,
\end{array}
\end{equation}
where  $\kappa
\in \mathbb{R}$ is a constant, and for $(s,\omega^0,x,x',u,u',\imath)\in [0,\infty)\times\Omega^0\times\dbR^n\times\dbR^n\times\dbR^m\times\dbR^m\times\mathbf{M}, $
\begin{equation}\label{equ-MFLQ-3}
\begin{array}{lll}
 \ds   g(s, \omega^0,x,x',u,u',\imath ):= \left\langle\2n \left(
                                     \begin{array}{ccc}
                                    \3n    Q(s, \omega^0 , \imath)  & \3n S(s, \omega^0 , \imath) ^\top\3n  \\
                                      \3n   S(s, \omega^0 , \imath)  &\3n R(s, \omega^0 , \imath)\  \3n \\
                                     \end{array}
                                   \right)\2n
                                   \left(
                                     \begin{array}{ccc}
                                     \3n  x \3n  \\
                                       \3n  u\3n    \\
                                     \end{array}
                                   \right),
                                   \left(
                                     \begin{array}{ccc}
                                     \3n  x  \3n \\
                                      \3n   u\3n    \\
                                     \end{array}
                                   \right)\2n \right\rangle + 2 \left\langle\2n  \left(
                                     \begin{array}{ccc}
                                  \3n   q(s, \omega^0 , \imath)   \3n \\
                                  \3n       r(s, \omega^0 , \imath)   \3n \\
                                     \end{array}
                                   \right),
                                   \left(
                                     \begin{array}{ccc}
                                    \3n   x  \3n \\
                                     \3n    u   \3n \\
                                     \end{array}
                                   \right)\2n \right\rangle
 \\
\ns\ds \hskip 3.2cm  +\left\langle\2n  \left(
                                     \begin{array}{ccc}
                                      \3n  \bar Q(s, \omega^0 , \imath)  & \3n\bar S(s, \omega^0 , \imath)^\top\3n   \\
                                    \3n  \bar   S(s, \omega^0 , \imath)  &\3n\bar R(s, \omega^0 , \imath) \ \3n \\
                                     \end{array}
                                   \right)\2n
                                   \left(
                                     \begin{array}{ccc}
                                     \3n  x ' \3n \\
                                      \3n   u'  \3n  \\
                                     \end{array}
                                   \right),
                                   \left(
                                     \begin{array}{ccc}
                                    \3n   x ' \3n \\
                                      \3n   u'  \3n  \\
                                     \end{array}
                                   \right)\2n \right\rangle   + 2 \left\langle\2n  \left(
                                     \begin{array}{ccc}
                                   \3n  \bar q(s, \omega^0 , \imath)  \3n  \\
                                    \3n  \bar    r(s, \omega^0 , \imath)  \3n  \\
                                     \end{array}
                                   \right),
                                   \left(
                                     \begin{array}{ccc}
                                   \3n    x '\3n  \\
                                    \3n     u '\3n   \\
                                     \end{array}
                                   \right)\2n \right\rangle.
\end{array}
\end{equation}
%
%
%

Let the set of admissible controls be $\mathcal{U}_{\text{ad}}^{\kappa}[t, \infty) := L_{\mathbb{F}}^{2, \kappa}(t, \infty; \mathbb{R}^m)$. We now formulate the control problem informally as follows.

\ss

$\textbf{Problem (C)}^\k_S$
For any $(t, x_{t}, \iota)\in \sD^2$, find an admissible control $u^{\ast} \in \mathcal{U}_{a d}^{\kappa}[t, \infty)$ such that
\begin{equation}\label{equ-MFLQ-5}
J^{\kappa}\left(t, x_{t} , \iota; u^{\ast}_\cd\right)=\inf _{u \in \mathcal{U}_{a d}^{\kappa}[t, \infty)} J^{\kappa}\left(t, x_{t} , \iota; u_\cd\right).
\end{equation}
When $S(\cd,\cd)=\bar S(\cd,\cd)\equiv0$, the above problem is denoted by    Problem (C)$^\k_0$.
 In what follows, we often suppress  the explicit dependence of the coefficients on their variables for notational simplicity.

 To ensure the well-posedness of Problem (C)$^\k_S$,
 the constant
$\k$   must be carefully chosen  to define the admissible set  $  \mathcal{U}_{\text{ad}}^{\kappa}[t, \infty)$.
However, as indicated in our previous studies \cite{WY-2021, Wei-Xu-Yu-2023}, an optimal control may fail to exist  in this predefined set   when the cross-term coefficients  $ S(\cd)$ and $\bar S(\cd)$ are not both identically zero.
This issue persists in the present setting. To address it in a unified manner,   we introduce  a transformation for   Problem (C)$^\k_S$, which can  eliminate    $ S(\cd)$ and $\bar S(\cd)$
  from the cost functional.
   {The transformation below involves the inverses of $R(\cd)$ and
$R(\cd)+\bar R(\cd)$; their well-definedness, together with the
positive semidefiniteness of the resulting state-weighting coefficients,
will be guaranteed by condition \textbf{(PD)} imposed below.}
Precisely,  we introduce the following notations,
 $$\left\{\3n\ba{ll}
 \ns  \sA :=A-BR^{-1}S, &  \bar\sA:= \bar{A}+BR^{-1}S -(B+\bar{B})(R+\bar{R})^{-1}(S+\bar S),\\
  \ns  \sC :=C-DR^{-1}S, & \bar\sC:= \bar{C}+ DR^{-1}S-(D+\bar{D})(R+\bar{R})^{-1}(S+\bar S),\\
   \ns \sM :=M-NR^{-1}S, & \bar\sM:= \bar{M}+ NR^{-1}S-(N+\bar{N})(R+\bar{R})^{-1}(S+\bar S),\\
   \ns \sQ:=Q-S^\top R^{-1}S, & \bar\sQ:= \bar{Q} +S^\top R^{-1}S-(S+\bar{S})^\top (R+\bar{R})^{-1}(S+\bar S) ,\\
   \ns \mathbbm{q}:= q-S^\top R^{-1}r,  & \bar{\mathbbm{q}}:= \bar q+S^\top R^{-1}r-(S+\bar S)^\top (R+\bar R)^{-1}(r+\bar r),
\ea\right.
$$
and apply the  control transformation
\bel{u-trans}{\mathbbm{u}}:= u +R^{-1} S  ( X -\dbE_s^0[X ])+(R +\bar R)^{-1}  (S +\bar S) \dbE_s^0[X ] \ee
to the state equation \rf {equ-MFLQ-2} and the cost \rf{equ-MFLQ-1}. This yields the following transformed system,
\begin{equation}\label{equ-MFLQ-tran}
\left\{
\begin{aligned}
& \mathrm dX_s=\( \sA  (s, \alpha_{s}) X_s + \bar{\sA} (s, \alpha_{s})\mathbb{E}^{0}_{s}{[X_s]} + B (s, \alpha_{s}) \mathbbm{u}_{s}+\bar{B} (s, \alpha_{s})\mathbb{E}^{0}_{s}{[\mathbbm{u}_{s}]}+b_s\) \, \mathrm ds\\
& \ + \(\sC (s, \alpha_{s}) X_s+\bar{\sC}  (s, \alpha_{s})\mathbb{E}^0_{s}{[X_s]}
+D  (s, \alpha_{s}) \mathbbm{u}_{s}
+\bar{D}  (s, \alpha_{s})\mathbb{E}^0_{s}{[\mathbbm{u}_{s}]}+\si_s\) \, \mathrm dW_{s} \\
&\ + \(\sM (s, \alpha_{s}) X_s+\bar{\sM}  (s, \alpha_{s})\mathbb{E}^0_{s}{[X_s]}+N (s, \alpha_{s}) \mathbbm{u}_{s}
+\bar{N}  (s, \alpha_{s})\mathbb{E}^0_{s}{[\mathbbm{u}_{s}]}+\g_s\) \, \mathrm d W^0 _{s}, \ s\ges t,\\
& X_t=x_{t}, \quad
\alpha_t=\iota,
\end{aligned}\right.
\end{equation}
 and   the following transformed cost,
\begin{equation}\label{equ-MFLQ-tran-cost}
\begin{array}{lll}
&\sJ^{\kappa}(t, x_{t},\iota ; \mathbbm{u}_\cd)=J^{\kappa}\left(t, x_{t} , \iota; u_\cd\right)=
\dfrac{1}{2} \mathbb{E} \ds\int_{t}^{\infty} \mathrm e^{\kappa s}\mathbbm g(s,\omega^0,X_s,\dbE_s^0[X_s],\mathbbm{u}_s,\dbE_s^0[\mathbbm{u}_s],\a_s)\mathrm ds,
\end{array}
\end{equation}
where for $(s,\omega^0,x,x',\mathbbm{u},\mathbbm{u}',\imath)\in [0,\infty)\times\Omega^0\times\dbR^n\times\dbR^n\times\dbR^m\times\dbR^m\times\mathbf{M}, $
\begin{equation*}\label{equ-MFLQ-3}
\begin{array}{lll}
 \ds \mathbbm g(s, \omega^0,x,x',\mathbbm{u},\mathbbm{u}',\imath ):= \left\langle\2n \left(
                                     \begin{array}{ccc}
                                    \3n    \sQ(s, \omega^0 , \imath)  & \3n 0\3n  \\
                                      \3n   0  &\3n R(s, \omega^0 , \imath) \3n \\
                                     \end{array}
                                   \right)\2n
                                   \left(
                                     \begin{array}{ccc}
                                     \3n  x \3n  \\
                                       \3n  \mathbbm{u}\3n    \\
                                     \end{array}
                                   \right),
                                   \left(
                                     \begin{array}{ccc}
                                     \3n  x  \3n \\
                                      \3n   \mathbbm{u}\3n    \\
                                     \end{array}
                                   \right)\2n \right\rangle+2 \left\langle\2n  \left(
                                     \begin{array}{ccc}
                                  \3n   \mathbbm{q}(s, \omega^0 , \imath)   \3n \\
                                  \3n       r(s, \omega^0 , \imath)   \3n \\
                                     \end{array}
                                   \right),
                                   \left(
                                     \begin{array}{ccc}
                                    \3n   x  \3n \\
                                     \3n    \mathbbm{u}  \3n \\
                                     \end{array}
                                   \right)\2n \right\rangle
 \\
\ns\ds \hskip 3.8cm  +\left\langle\2n  \left(
                                     \begin{array}{ccc}
                                      \3n  \bar \sQ(s, \omega^0 , \imath)  & \3n 0 \3n   \\
                                    \3n   0  &\3n\bar   R(s, \omega^0 , \imath) \3n \\
                                     \end{array}
                                   \right)\2n
                                   \left(
                                     \begin{array}{ccc}
                                     \3n  x ' \3n \\
                                      \3n  \mathbbm{u}'  \3n  \\
                                     \end{array}
                                   \right),
                                   \left(
                                     \begin{array}{ccc}
                                    \3n   x ' \3n \\
                                      \3n   \mathbbm{u}'  \3n  \\
                                     \end{array}
                                   \right)\2n \right\rangle   + 2 \left\langle\2n  \left(
                                     \begin{array}{ccc}
                                   \3n  \bar  {\mathbbm q}(s, \omega^0 , \imath)  \3n  \\
                                    \3n  \bar    r(s, \omega^0 , \imath)  \3n  \\
                                     \end{array}
                                   \right),
                                   \left(
                                     \begin{array}{ccc}
                                   \3n    x '\3n  \\
                                    \3n    \mathbbm{u} '\3n   \\
                                     \end{array}
                                   \right)\2n \right\rangle.
\end{array}
\end{equation*}

We now formulate the following optimal control problem based on the transformed system \eqref{equ-MFLQ-tran} and   \eqref{equ-MFLQ-tran-cost}.
\ss

 \textbf{Problem (C)$^\k$}
For any $(t, x_{t}, \iota)\in \sD^2$, find an admissible control $\mathbbm{u}^{\ast} \in \mathcal{U}_{a d}^{\kappa}[t, \infty)$ such that
\begin{equation}\label{equ-MFLQ-5}
\sJ^{\kappa}\left(t, x_{t} , \iota; \mathbbm{u}^{\ast}_\cd\right)=\inf _{\mathbbm{u} \in \mathcal{U}_{a d}^{\kappa}[t, \infty)} \sJ^{\kappa}\left(t, x_{t} , \iota; \mathbbm{u}_\cd\right).
\end{equation}
In this case, $\mathbbm{u}^{\ast}_\cd$ is called {\emph {an open-loop optimal control}} for Problem (C)$^\k$ at $(t, x_{t}, \iota)$, the corresponding state process $\dbX \equiv X(\cdot; t, x_{t}, \iota, \mathbbm{u}^{\ast})$ is called   the  optimal state  trajectory, and $(\mathbbm{u}^{\ast}, \dbX)$ is called   an optimal pair.

 \ms

Let $\dbS^{n}$ be the set of all $(n \times n)$ symmetric matrices.
For any $\dbS^{n}$-valued stochastic process $\Lambda(\cdot)$, if   $\Lambda(s,\omega)\ges 0$ (resp.$> 0,\ \les 0,\ < 0$) for almost all $(\omega, s)\in \Omega \times [0, \infty)$,  we denote $\Lambda(\cd,\cd)\ges 0$ (resp. $> 0,\ \les 0,\ < 0$).
Moreover, we write $\Lambda(\cd,\cd) \gg 0$ (respectively, $\ll 0$) if there exists a constant $\delta > 0$ such that $\Lambda(\cd,\cd) - \delta \mathrm{I}_n \ges 0$ (respectively, $\les 0$), where $\mathrm{I}_n$ denotes the $n \times n$ identity matrix.  In the following, we denote by
$ \lambda_{\max }( \L) $
  the largest eigenvalue of the symmetric matrix  $ \L$.

Now we  specify the conditions for the well-posedness of  Problem (C)$^\k$.
\begin{description}
    \item[(H1)]
(i) For each  $\imath\in \mathbf{M}$,
$$
\ba{ll}
\ns\ds A(\cd,\imath), \bar{A}(\cd,\imath)\in L^{\infty}_{\mathbb{F}^0}(t, \infty ; \mathbb{R}^{n \times n}),\qquad\
B(\cd,\imath), \bar{B}(\cd,\imath)\in L^{\infty}_{\mathbb{F}^0}(t, \infty ; \mathbb{R}^{n \times m}),\\
\ns\ds C(\cd,\imath), \bar{C}(\cd,\imath)\in L^{\infty}_{\mathbb{F}^0}(t, \infty ; \mathbb{R}^{nd \times n}),\qq
D(\cd,\imath), \bar{D}(\cd,\imath)\in L^{\infty}_{\mathbb{F}^0}(t, \infty ; \mathbb{R}^{nd \times m}),\\
\ns\ds  M(\cd,\imath), \bar{M}(\cd,\imath)\in L^{\infty}_{\mathbb{F}^0}(t, \infty ; \mathbb{R}^{nd_0 \times n}),\q
N(\cd,\imath), \bar{N}(\cd,\imath)\in L^{\infty}_{\mathbb{F}^0}(t, \infty ; \mathbb{R}^{nd_0 \times m});
\ea
$$

  (ii) for some $\k^\ast\in\dbR $, $b_{\cdot}\in L_{\mathbb{F}}^{2, \kappa^\ast}(t, \infty; \mathbb{R}^n)$,
$\sigma_{\cdot}\in L_{\mathbb{F}}^{2, \kappa^\ast}(t, \infty; \mathbb{R}^{nd})$,
$\gamma_{\cdot}\in L_{\mathbb{F}}^{2, \kappa^\ast}(t, \infty; \mathbb{R}^{nd_0})$ .

    \item[(H2)]
(i) For each  $\imath\in \mathbf{M}$,
$$ \ba{ll}
\ns\ds Q(\cd,\imath), \bar{Q}(\cd,\imath)\in L^{\infty}_{\mathbb{F}^0}(t, \infty ; \dbS^{n}), \q
R(\cd,\imath),\bar{R}(\cd,\imath)\in L^{\infty}_{\mathbb{F}^0}(t, \infty ; \dbS^{m}),  \\
\ns\ds S(\cd,\imath), \bar{S}(\cd,\imath)\in L^{\infty}_{\mathbb{F}^0}(t, \infty ;\mathbb{R}^{m \times n});
 \ea$$

 (ii)    $q(\cd,\imath), \bar{q}(\cd,\imath)\in L^{2, \kappa^\ast}_{\mathbb{F}}(t, \infty ;\mathbb{R}^{n}),$ $
r(\cd,\imath), \bar{r}(\cd,\imath)\in L^{2, \kappa^\ast}_{\mathbb{F}}(t, \infty ;\mathbb{R}^{m})$.

   \item[(H3)]
There exist  constants  $  \k_x, \k_{x\mu}  \in \mathbb{R}$ and  $\kappa_{y}, \kappa_{y\nu} \in \mathbb{R}$ such that, for all $(s, \omega^0, \imath)\in [0, \infty)\times \Omega^0\times\mathbf{M}$,

 (i)
$
\left\{\2n
\ba{ll}
\ns\ds  \k_x \ges  \mathop{\esssup}\limits_{(s, \omega^0, \imath) } \lambda_{\max }\big(\sA+\sA^{\top} + \sC^{\top}\sC+\sM^{\top}\sM\big)(s, \omega^0, \imath),\\
\ns\ds \k_{x\mu}  \ges  \mathop{\esssup}\limits_{(s, \omega^0, \imath) }
 \lambda_{\max }\big(\bar\sA  +  \bar\sA    ^\top +  (\sC+ \bar\sC)  ^\top(\sC+ \bar\sC)    -\sC^\top \sC\\
\ns\ds  \hskip 3.5cm +  (\sM+ \bar\sM)  ^\top(\sM+ \bar\sM)   -\sM^\top \sM \big)(s,\omega^0,\imath),
\ea\right.
$

%
%
%
(ii)  $
\left\{\2n
\ba{ll}
\ns\ds \kappa_{y} \ges \kappa+\frac{1}{2}\mathop{\esssup}\limits_{(s, \omega^0, \imath)  }
 \lambda_{\max }\big(\sA+\sA^{\top}\big)(s, \omega^0, \imath),\\
\ns\ds  \kappa_{y\nu} \ges \frac{1}{2} \mathop{\esssup}\limits_{(s, \omega^0, \imath)  }
 \lambda_{\max }\big(\bar{\sA} +\bar\sA^{\top}\big)(s, \omega^0, \imath).
\ea\right.
$

   \item[(PD)]   For all $\imath \in\textbf{M}$, the following hold,
$$
\ba{ll}
\ns\ds {\rm(i)}\q R (\cdot, \imath) \gg 0,\q
\left(\begin{array}{cc}
\3n Q (\cdot, \imath) & \3n S(\cdot, \imath)^{\top} \3n \\
\3n S(\cdot, \imath) & \3n R(\cdot, \imath)\ \3n
\end{array}\right) \ges 0, \\
\ns\ds   {\rm(ii)}\q   (R+\bar R) (\cdot, \imath) \gg 0,\q  \left(\begin{array}{cc}
\3n (Q +\bar Q)(\cdot, \imath) & \3n (S+\bar S) (\cdot, \imath)^{\top}\2n \\
\3n (S+\bar S)(\cdot, \imath) & \3n (R+\bar R)(\cdot, \imath)\  \3n
\end{array}\right) \ges 0.
 \ea $$
 \end{description}

For convenience, we define
\bel{k}\left\{\3n\ba{ll}
 \ns \ds{\overline{\kappa}:=-{\k}_x-{\k}_{x\mu}
 \mathbf{1}_{\{{\k}_{x\mu}>0\}}},\\
\ns\ds
\underline{\kappa}:=2 (\k_{y}+ \k_{y \nu}\mathbf{1}_{\{\k_{y \nu}>0\}}+K ),\\
\ns\ds K:= \max\limits_{  \imath\in\textbf{M}}\Big\{\Vert \mathscr{Q} (\cdot, \imath)\Vert^{2}
+\Vert \mathscr{C}(\cdot, \imath)\Vert^{2}+\Vert \mathscr{M} (\cdot, \imath)\Vert^{2}
+\max\{\Vert   \bar{\mathscr{Q}} (\cdot, \imath)\Vert^{2}, \Vert \bar{\mathscr{C}} (\cdot, \imath)\Vert^{2},
\Vert\bar{\mathscr{M}} (\cdot, \imath)\Vert^{2}\}\!\Big\}.\!\!\!\!

\ea\right.\ee
%
%
%
Under  $\mathbf{(H1)}$ and  $\mathbf{(H3)}$-(i), for  any $(t, x_{t}, \iota)\in \sD^2$ and
$\mathbbm{u} _\cd\in \mathcal{U}_{\text{ad}}^{\kappa}[t, \infty)$ with {$\kappa<\overline{\k}$ and $\kappa\les\kappa^\ast$},
 SDE \eqref{equ-MFLQ-tran} admits the unique solution
$X \equiv X(\cdot; t, x_t, \iota, \mathbbm{u}) \in L_{\mathbb{F}}^{2, \kappa}(t, \infty; \mathbb{R}^n)$ (see  Corollary \ref{Cor-SDE-1}).
In this case, the pair  $(\mathbbm{u}_\cd, X_\cd)$ is  referred to as   an admissible  pair of Problem (C)$^\k$.
 Moreover,  under   $\mathbf{(H2)}$,  the cost functional \eqref{equ-MFLQ-tran} is  well-defined for every admissible pair $(\mathbbm{u}_\cd, X_\cd)$.
 This ensures that   Problem (C)$^\k$
 is well-posed whenever {$\kappa<\overline{\kappa}$ and
 $\kappa\les\kappa^\ast$}.

{
\begin{remark}\label{rem-common-regime}\sl
(i) By the construction of $\mathbb F^0$ in
Section~\ref{SEC_2}, $\alpha_s$ is $\mathcal F_s^0$-measurable.
Hence, if $\Lambda(\cdot,\imath)$ is $\mathbb F^0$-progressively
measurable for each $\imath\in\textbf{M}$, then
$\Lambda(s,\alpha_s)$ is $\mathcal F_s^0$-measurable. Consequently,
for any integrable random variable $\xi_s$ of compatible dimension,
$$
\mathbb E_s^0\big[\Lambda(s,\alpha_s)\xi_s\big]
=
\Lambda(s,\alpha_s)\mathbb E_s^0[\xi_s].
$$
Under $\mathbf{(H1)}$--$\mathbf{(H2)}$, this identity applies to all
the regime-dependent matrix coefficients appearing in the subsequent
calculations; under $\mathbf{(PD)}$, it also applies to their relevant
inverses. We shall use this pull-out property repeatedly below when
taking conditional expectations in the variational and Hamiltonian
calculations.

We emphasize that this property concerns only
$\mathcal F_s^0$-measurable coefficient--variable products. It does
not imply that the common Markov-chain martingale term vanishes under
conditional projection. As explained in Section~\ref{SEC_2}, its
conditional projection remains a Markov-chain martingale term with
the projected integrand. Accordingly, all such terms are retained in
the adjoint and Hamiltonian systems below.

(ii) Although stated here for later reference, the identity in part
(i) has already been used in deriving
\eqref{equ-MFLQ-tran} and \eqref{equ-MFLQ-tran-cost} from the control
transformation \eqref{u-trans}. Indeed,
$$
\mathbb E_s^0\!\left[
R(s,\alpha_s)^{-1}S(s,\alpha_s)
\big(X_s-\mathbb E_s^0[X_s]\big)
\right]=0,
$$
and hence
$$
\mathbb E_s^0[\mathbbm u_s]
=
\mathbb E_s^0[u_s]
+(R+\bar R)^{-1}(S+\bar S)\mathbb E_s^0[X_s],
$$
where the coefficients in the last identity are evaluated at
$(s,\alpha_s)$. Therefore, the transformed state equation
\eqref{equ-MFLQ-tran} and cost functional
\eqref{equ-MFLQ-tran-cost} follow directly from
\eqref{u-trans} and the measurability property in part (i).
\end{remark}
}

We now turn to  characterizing the open-loop optimal control of Problem (C)$^\k$.


\begin{lemma}\label{Le-MFLQ-1}\sl
Suppose    $\mathbf{(H1)}$-$\mathbf{(H3)}$ and $\mathbf{(PD)}$ hold. For {$\underline\kappa<\kappa<\overline\kappa$ and
$\kappa\les\kappa^\ast$}  and the initial triple  $(t, x_{t}, \iota) \in \sD^2$,
let $(\mathbbm{u}^{\ast}_\cd, \dbX_\cd ) \in \mathcal{U}_{ad}^{\kappa}[t, \infty) \times L_{\mathbb{F}}^{2, \kappa}(t, \infty ; \mathbb{R}^{n})$ be an admissible pair and  $(\dbY,\dbZ,\dbZ^0 , \dbK )\in \cL_{\mathbb{F}}^{2, \kappa}(t, \infty  )  $ be the solution of  the following BSDE (the adjoint equation of Problem (C)$^\k$)
\begin{equation}\label{equ-MFLQ-8}
\begin{aligned}
&\mathrm d\dbY _{s}
= -\(\big(\kappa \mathrm I_n + \sA \big)^{\top} \dbY _{s}
+\bar{\sA} \,^{\top}\mathbb{E}^0_{s}{[\dbY _{s}]}+\sC ^{\top} \dbZ_{s}
+\bar{\sC} \,^{\top}\mathbb{E}^0_{s}{[\dbZ_{s}]}+ \sM^{\top} \dbZ^0 _{s}
+\bar{\sM}\, ^{\top}\mathbb{E}^0_{s}{[   \dbZ^0 _{s} ]}
 \\
&\hskip1.2cm+\sQ  \dbX _{s}
+ \bar{\sQ} \mathbb{E}^0_{s}{[\dbX _{s}]}
+ \mathbbm{q} + \bar{\mathbbm{q}} \) \, \mathrm ds
%
%
+\dbZ_{s}\, \mathrm dW_{s}
+ \dbZ^0 _{s} \, \mathrm d W^0 _{s}
+ {\dbK_{s} \bullet\mathrm d M_s}, \ s \ges t.
\end{aligned}
\end{equation}
Then $\mathbbm{u}^{\ast}_\cd$ is  optimal  for Problem (C)$^\k$ at $(t, x_{t}, \iota)$ if and only if the following optimality condition holds,
\begin{equation}\label{equ-MFLQ-9}
\begin{aligned}
& B^{\top} \dbY_s +\bar{B} ^{\top} \mathbb{E}^{0}_s{[\dbY _s]}
+D ^{\top} \dbZ_s +\bar{D} ^{\top} \mathbb{E}^{0}_s{[\dbZ_s]}
+N ^{\top}  \dbZ^0 _{s}
+\bar{N} ^{\top} \mathbb{E}^{0}_s{[ \dbZ^0 _{s}]}\\
&
+R \mathbbm{u}^{\ast}_s
+ \bar{R} \mathbb{E}^{0}_s{[\mathbbm{u}^{\ast}_s]}
+r + \bar{r} =0, \q \text{a.e.} \ s \ges t, \ \dbP\text{-a.s.}
\end{aligned}
\end{equation}
\end{lemma}
The proof follows a similar line of argument as in our previous works \cite{WY-2021, Wei-Xu-Yu-2023} and is therefore omitted here.
 \ms
Observe that condition $(\mathbf{PD})$  is equivalent to
$$%
R (\cdot, \cd) \gg 0, \ (R +\bar R)(\cdot, \cd) \gg 0    \quad
\sQ (\cdot, \cd) \ges 0,\ (\sQ+\bar \sQ)  (\cdot, \cd) \ges 0 .
$$
In this case,  the optimal control   $ \mathbbm{u}^\ast_\cd $
characterized by \rf{equ-MFLQ-9}  has  the following explicit representation:
\begin{equation}\label{equ-MFLQ-19}
 \begin{aligned}
&\mathbbm{u}^{\ast}_{\cdot} = -R^{-1}\Big(  B^{\top}(\dbY - \mathbb{E}^0_s[\dbY])  +  D^{\top}(\dbZ-\mathbb{E}^0_s[\dbZ])
 +  N^{\top}( \dbZ^0  - \mathbb{E}^0_s[ \dbZ^0])  + r + \bar{r} - \mathbb{E}^0_s[r + \bar{r}])\Big)\\
&\qq -(R + \bar{R})^{-1}\Big( (B + \bar{B})^{\top} \mathbb{E}^0_s[\dbY] + (D + \bar{D})^{\top} \mathbb{E}^0_s[\dbZ]
 + (N + \bar{N})^{\top} \mathbb{E}^0_s[ \dbZ^0]  + \mathbb{E}^0_s[r + \bar{r}]\Big) .
 %
 \end{aligned}
  \end{equation}
{The representation \eqref{equ-MFLQ-19} shows that the optimal control is
regime-dependent. Indeed, all the coefficients appearing in this formula are
evaluated at $(s,\alpha_s)$, where $\alpha_s$ is the current state of the
Markov chain. Thus the optimal control is adjusted dynamically according to the current
regime of the Markov chain.
}

 Substituting $\mathbbm{u}^\ast_{\cdot}$ into  the state equation  \eqref{equ-MFLQ-2} and the adjoint equation \eqref{equ-MFLQ-8},  we obtain  the following infinite horizon
fully coupled McKean-Vlasov FBSDEs,
 \begin{equation}\label{equ-MFLQ-Hamiltonian}
\left\{\3n
\ba{ll}
 \ds \mathrm d\dbX _s=b(s, \dbX_s,\dbY_s,\dbZ_s,\dbZ^0_s,\cL^1_{( \dbX_s,\dbY_s,\dbZ_s,\dbZ^0_s)},\a_s)\, \mathrm ds+ \si(s, \dbX_s,\dbY_s,\dbZ_s,\dbZ^0_s,\cL^1_{( \dbX_s,\dbY_s,\dbZ_s,\dbZ^0_s)},\a_s)\, \mathrm dW_s \!\!\!\!\! \\
\ns\ds \hskip1cm
+ \wt \si(s, \dbX_s,\dbY_s,\dbZ_s,\dbZ^0_s,\cL^1_{( \dbX_s,\dbY_s,\dbZ_s,\dbZ^0_s)},\a_s)\,  \mathrm dW^0_{s}, \ s\ges t,\\
\ns\ds  \mathrm d\dbY _s =
-f(s, \dbX_s,\dbY_s,\dbZ_s,\dbZ^0_s,\cL^1_{( \dbX_s,\dbY_s,\dbZ_s,\dbZ^0_s)},\a_s)\, \mathrm ds
+ \dbZ_s \, \mathrm dW_{s}
+ \dbZ^0 _{s} \, \mathrm d W^0 _{s}
+ {\dbK _s \bullet \mathrm d M_s}, \  s\ges t, \\
\ns\ds  \dbX _{t}=x_{t}, \quad \alpha_{t}=\iota,
\ea\right.
\end{equation}
where, for   $(s,x,\th, \nu,\imath)\in [0,\i)\times\dbR^n\times \sR\times \sP_2(\dbR^n\times\sR)\times {\textbf M} $,
 \bel{b-si-f}\left\{\3n
\ba{ll}
 \ds b(s,x,\th, \nu,\imath)= b_s +\sA  x   + \bar\sA   \int_{\dbR^n\times\sR} \tilde x \, \mathrm d\nu(\tilde x,\tilde \th)  +\Pi_  B (s, \th, \nu,\imath)    ,\\
 \ns \ds \si(s,x,\th, \nu,\imath)= \si_s +\sC  x  + \bar\sC   \int_{\dbR^n\times\sR} \tilde x \, \mathrm d\nu(\tilde x,\tilde \th)  +\Pi_  D(s,  \th,\nu,\imath)    ,\\
 \ns  \ds \wt\si(s,x,\th, \nu,\imath)= \wt\si_s +\sM  x + \bar\sM  \int_{\dbR^n\times\sR} \tilde x \, \mathrm d\nu(\tilde x,\tilde \th)  +\Pi_  N(s,  \th,\nu,\imath)    ,\\

 \ns \ds f(s,x,\th, \nu,\imath)=\sQ  x
+  \bar{\sQ} \int_{\dbR^n\times\sR } \tilde x \, \mathrm d\nu(\tilde x,\tilde \th)  +\big(\kappa \mathrm I_n + \sA \big)^{\top} y+\sC ^{\top} z + \sM^{\top} z^0
\\
\ns\ds \qq\qq \qq\q+ \int_{\dbR^n\times\sR }\big(\bar{\sA}\, ^{\top}\tilde y+\bar{\sC}\,^\top \tilde z+\bar{\sM}\,^\top \tilde {z}^0\big ) \mathrm d\nu(\tilde x,\tilde \th)
+ \mathbbm{q} + \bar{\mathbbm{q} },
\ea
\right.\ee
 with
 $$ \ba{ll}
 \ns \ds \ \Pi_  \Phi(s,  \th,\nu,\imath) \\
  \ns\ds :=-\Phi R^{-1}\[B^{\top}y+D^{\top}z+N^{\top} z^0   -\int_{\dbR^n\times\sR }\big(B^\top\tilde y+D^\top \tilde z+N^\top \tilde {z}^0\big ) \mathrm d\nu(\tilde x,\tilde \th)+r+\bar{r}-\dbE^0_s[r+\bar r]\] \\
\ns \ds\q      -(\Phi+\bar{\Phi})(R+\bar{R})^{-1}\[\int_{\dbR^n\times\sR }\big((B+\bar{B})^{\top}\tilde y+(D+\bar{D})^{\top} \tilde z + (N+\bar{N})^{\top}  \tilde{z}^0 \big)\mathrm d\nu(\tilde x,\tilde \th) +\dbE^0_s[r+\bar r]  \]   ,
\ea
$$
for $\Phi=B,D, N$, resp.
 System  \rf{equ-MFLQ-Hamiltonian} is in fact   the Hamiltonian system of   Problem (C)$^\k$ at $(t, x_t, \iota)$.
The above analysis shows that the existence and uniqueness of an optimal control of  Problem (C)$^\k$ is closely related to the unique solvability of the Hamiltonian system \rf{equ-MFLQ-Hamiltonian}.
 Therefore, we   establish the  {well-posedness} result  of \rf{equ-MFLQ-Hamiltonian} in the following.

\begin{theorem}\label{Pro-MFLQ}\sl
Assume $\mathbf{(H1)}$-$\mathbf{(H3)}$, $(\mathbf{PD})$ and  {$\ds \k= -\frac{{\k}_x}{2}+{\k}_{y}\in
(\underline{{\k}},\overline {\k})$ with $\kappa\les\kappa^\ast$}. Then
 for every $(t, x_t, \iota) \in \sD^2$, system \eqref{equ-MFLQ-Hamiltonian} admits a unique solution $(\dbX ,\dbY,\dbZ,\dbZ^0 , \dbK ) \in L_\dbF^{2,\k}(t,\i;\dbR^n) \times \cL_{\mathbb{F}}^{2, { \k}}(t, \infty)$.
In this case, the following
 \begin{equation}\label{equ-MFLQ-u-opt}
 \begin{aligned}
&u^\ast \!=\! -R^{-1}\Big(  S  ( \dbX \!-\!\dbE_s^0[\dbX ])+ B^{\top}(\dbY \!-\! \mathbb{E}^0_s[\dbY])
 + D^{\top}(\dbZ\!-\!\mathbb{E}^0_s[\dbZ])
 + N^{\top}( \dbZ^0 \! -\!\mathbb{E}^0_s[ \dbZ^0])  + r + \bar{r} - \mathbb{E}^0_s[r + \bar{r}]\Big)\\
&\qq -(R\! + \!\bar{R})^{-1}\Big(  (S \!+\!\bar S) \dbE_s^0[\dbX ] +(B \!+ \!\bar{B})^{\top} \mathbb{E}^0_s[\dbY] + (D + \bar{D})^{\top} \mathbb{E}^0_s[\dbZ]
 + (N + \bar{N})^{\top} \mathbb{E}^0_s[ \dbZ^0]  + \mathbb{E}^0_s[r + \bar{r}]\Big) ,
 \end{aligned}
  \end{equation}
lies in $\mathcal{U}_{ad}^{{\k}}[t, \infty)$ and is the unique open-loop optimal control of Problem (C)$^\k_S$ at $(t, x_t, \iota)$.

\end{theorem}

\begin{proof}
Observe that \rf{equ-MFLQ-Hamiltonian} is   a linear form of the general  system \rf{equ-FBSDE-1}. Therefore, it suffices to verify that the assumptions of Theorem \ref{Th-FBSDE}  are satisfied. To this end, for $(s,\th,\bar\th,\nu,\bar \nu )\in [0,\infty)\times\mathscr{R}\times\mathscr{R}\times\mathscr{P}_2(\mathscr{R} )\times \mathscr{P}_2(\mathscr{R})  $,
 $(  \upsilon, \bar{\upsilon})\in  \mathscr{P}_2(\mathbb{R}^n ) \times\mathscr{P}_2(\mathbb{R}^n)$  and $ \imath\in \textbf{M}$,  we set
\begin{equation*}\label{equ-MFLQ-abbreviation}
\begin{aligned}
&\b_1= \min\Big\{ \frac1{L_1}, L_2\Big\},\q \b_2=0,\q \phi(y, \bar{y}, \upsilon, \bar{\upsilon}, \imath)\equiv0, \quad \psi(s,\th,\bar\th,\nu,\bar \nu, \imath ) \equiv0,\\
& \varphi(s,\th,\bar\th,\nu,\bar \nu, \imath ) :=\Big|\rho_1(s,\th-\bar \th,\imath)-\( \int_{\dbR^n\times\sR }\rho_1(s,\tilde\th,\imath)\mathrm d \nu (\tilde x,\tilde \th)-\int_{\dbR^n\times\sR}\rho_1(s,\tilde\th,\imath)\mathrm d \bar \nu (\tilde x,\tilde \th) \)  \Big|^2\!\\
& \hskip 3.2cm
+\Big|\int_{\dbR^n\times\sR }\rho_2(s, \tilde \th,\imath)\mathrm d \nu (\tilde x,\tilde \th)-\int_{\dbR^n\times\sR}\rho_2(s, \tilde \th,\imath)\mathrm d \bar\nu (\tilde x,\tilde \th) \Big|^2,
\end{aligned}
\end{equation*}
  where
$$\left\{
\begin{aligned}
&\rho_1(s,\th,\imath):=B^\top y+D^\top z+N ^\top z^0, \\
&\rho_2(s,\th,\imath):=(B+\bar B)^\top y+(D+\bar D)^\top z+ (N+\bar N)^\top z^0, \\
&L_{1}:=2\max\Big\{ \max \limits_{ \imath\in\textbf M }\|B R^{-1}(\cdot, \imath) \|^2,\ \max \limits_{ \imath\in\textbf M }\|D R^{-1}(\cdot, \imath) \|^2, \max \limits_{ \imath\in\textbf M }\|N R^{-1}(\cdot, \imath) \|^2, \\
&\hskip 2.4cm   \max \limits_{ \imath\in\textbf M } \|(B+\bar B) (R+\bar R)^{-1} (\cdot, \imath)\|^2, \ \max \limits_{ \imath\in\textbf M }\|(D+\bar D) (R+\bar R)^{-1}(\cdot, \imath) \|^2,\\
 &\hskip 2.4cm  \max \limits_{ \imath\in\textbf M }\|(N+\bar N) (R+\bar R)^{-1}(\cdot, \imath) \|^2\Big\},\\
&L_{2}:= \min\Big\{\mathop{\essinf}\limits_{(s, \omega^{0}, \imath)}\lambda_{\min } R\,^{-1} (s, \imath ) ,
\mathop{\essinf}\limits_{(s, \omega^{0}, \imath)}\lambda_{\min } (R+\bar R)^{-1}(s,\imath)\Big\}.
\end{aligned}\right.
$$
With {$\ds \k= -\frac{{\k}_x}{2}+{\k}_{y}\in
(\underline{{\k}},\overline {\k})$ and $\kappa\les\kappa^\ast$},  and  under $(\mathbf{PD})$, it is straightforward to verify that the coefficients $b$, $\si$, $\wt\si$ and $f$  in \eqref{b-si-f}     satisfy   $\mathbf{(C1)}$, $\mathbf{(C2)}$ and $\mathbf{(C3.1)}$. The unique solvability of   \rf{equ-MFLQ-Hamiltonian} then follows directly from Theorem  \ref{Th-FBSDE}.
Therefore, we conclude that  $\mathbbm u^{\ast}_\cd$ in   \eqref{equ-MFLQ-19} lies in $ \mathcal{U}_{ad}^{{\k}}[t, \infty)$ and is the unique open-loop optimal control   of Problem (C)$^\k$   at $(t, x_t, \iota)$.
Furthermore, by applying the transformation   \rf{u-trans}, $u^\ast_\cd$ given in \rf{equ-MFLQ-u-opt} is consequently the unique open-loop optimal control of Problem (C)$_S^\k$   at $(t, x_t, \iota)$.
\end{proof}

Based on the preceding results, we observe that the presence of  $S(\cd,\cd)$  and $\bar S(\cd,\cd) $ indeed influences condition  {\bf(H3)}.
 When $S(\cd,\cd)=\bar S(\cd,\cd)\equiv0$, {\bf(H3)} simplifies to
\ss

\no {\bf(H3)$^0$}
(i) $\left\{\2n
\ba{ll}
\ns\ds  \k^0_x \ges  \mathop{\esssup}\limits_{(s, \omega^0, \imath) } \lambda_{\max }\big(A+A^{\top} + C^{\top}C+M^{\top}M\big)(s, \omega^0, \imath),\\
\ns\ds \k^0_{x\mu}  \ges  \mathop{\esssup}\limits_{(s, \omega^0, \imath) }
 \lambda_{\max }\big(\bar A  +  \bar A    ^\top +  ( C+ \bar C)  ^\top( C+ \bar C)    - C^\top C\\
\ns\ds  \hskip 3.5cm +  (M+ \bar M)  ^\top( M+ \bar M)   - M^\top M \big)(s,\omega^0,\imath),
\ea\right.
$

  \ss

\q\  (ii)  $
\left\{\2n
\ba{ll}
\ns\ds \kappa_{y} ^0\ges \kappa+\frac{1}{2}\mathop{\esssup}\limits_{(s, \omega^0, \imath)  }
 \lambda_{\max }\big(A+A^{\top}\big)(s, \omega^0, \imath),\\
\ns\ds  \kappa_{y\nu}^0 \ges \frac{1}{2} \mathop{\esssup}\limits_{(s, \omega^0, \imath)  }
 \lambda_{\max }\big(\bar{A} +\bar A^{\top}\big)(s, \omega^0, \imath),
\ea\right.
$ \\
and  \rf{k} becomes
$$\left\{\2n\ba{ll}
 \ns \ds{\overline{\kappa}^0:=-{\k}_x^0-{\k}_{x\mu}^0
 \mathbf{1}_{\{{\k}_{x\mu}^0>0\}}},\\
\ns\ds
\underline{\kappa}^0:=2 (\k_{y}^0+ \k_{y \nu}^0\mathbf{1}_{\{\k_{y \nu}^0>0\}}+K^0 ),\\
\ns\ds K^0:= \max\limits_{  \imath\in\textbf{M}}\Big\{\Vert {Q} (\cdot, \imath)\Vert^{2}
+\Vert {C}(\cdot, \imath)\Vert^{2}+\Vert{M} (\cdot, \imath)\Vert^{2}
+\max\{\Vert   \bar{{Q}} (\cdot, \imath)\Vert^{2}, \Vert \bar{{C}} (\cdot, \imath)\Vert^{2},
\Vert\bar{{M}} (\cdot, \imath)\Vert^{2}\}\Big\}.\!\!\!\!
\ea\right.
$$
Therefore, under  $(\mathbf{PD})$, when  $S(\cd,\cd)=\bar S(\cd,\cd)\equiv0$ and {$\ds \k= -\frac{{\k}_x^0}{2}+{\k}_{y}^0\in
(\underline{{\k}}^0,\overline {\k}^0)$ with $\kappa\les\kappa^\ast$},   Problem (C)$^\k_0$ admits a unique open-loop optimal control $ u^{\ast}_\cd\in \mathcal{U}_{ad}^{{\k}}[t, \infty)$ at $(t, x_t, \iota)$, give by
  \begin{equation*}\label{equ-MFLQ-u-opt-0}
 \begin{aligned}
&u^\ast  = -R^{-1}\Big(    B^{\top}(\dbY - \mathbb{E}^0_s[\dbY])  +  D^{\top}(\dbZ-\mathbb{E}^0_s[\dbZ])
 +  N^{\top}( \dbZ^0  - \mathbb{E}^0_s[ \dbZ^0])  + r + \bar{r} - \mathbb{E}^0_s[r + \bar{r}]\Big)\\
&\qq -(R + \bar{R})^{-1}\Big(  (B + \bar{B})^{\top} \mathbb{E}^0_s[\dbY] + (D + \bar{D})^{\top} \mathbb{E}^0_s[\dbZ]
 + (N + \bar{N})^{\top} \mathbb{E}^0_s[ \dbZ^0]  + \mathbb{E}^0_s[r + \bar{r}]\Big)  ,
 \end{aligned}
  \end{equation*}
 where $(\dbX ,\dbY,\dbZ,\dbZ^0 , \dbK ) $ is the solution of the  renew version of system \rf{equ-MFLQ-Hamiltonian} with  $S(\cd,\cd)=\bar S(\cd,\cd)\equiv0$.

\br{}\sl
 The above analysis reveals that under   $(\mathbf{PD})$,   the condition $S(\cd,\cd)=\bar S(\cd,\cd)\equiv0$ directly affects  the space in which the optimal control exists--a phenomenon first noted in \cite{WY-2021}.
 This represents a major difference from the finite horizon problem. We further  observe that if the positive semidefiniteness condition  {\bf(PD)} is strengthened to strict positive definiteness, this distinction may disappear, as illustrated  in \cite{Hua-Luo}.

\er

%
%
%
%
%
%
\subsection{Infinite horizon  linear quadratic  mean field  game problems } 
\label{sub:mf_lqg}

In this part, we study an infinite horizon linear quadratic mean field game problem.
Let a pair of stochastic processes  $(\BX , \Bu  )  \in L_{\mathbb{F}^0}^{2, \kappa}(t, \infty; \mathbb{R}^n)\times L_{\mathbb{F}^0}^{2, \kappa}(t, \infty; \mathbb{R}^m)$ be given.
For any initial triple $(t, x_{t}, \iota)\in \sD^2$,   the state equation is given by
\begin{equation}\label{equ-MFLQG-2}
\left\{\2n
\begin{aligned}
& \mathrm dX_s=\(A (s, \alpha_{s}) X_s + \bar{A} (s, \alpha_{s})\BX   _s
+ B (s, \alpha_{s}) u_{s}+\bar{B}(s, \alpha_{s})\Bu_{s}+b_s\) \, \mathrm ds\\
& + \(C(s, \alpha_{s}) X_s+\bar{C}(s, \alpha_{s})\BX _s
+D(s, \alpha_{s}) u_{s}
+\bar{D}(s, \alpha_{s})\Bu_{s}+\sigma_s\) \, \mathrm dW_{s} \\
&+ \(M(s, \alpha_{s}) X_s+\bar{M}(s, \alpha_{s})\BX _s+N(s, \alpha_{s}) u_{s}
+\bar{N}(s, \alpha_{s})\Bu_{s}+\gamma_s\) \, \mathrm d W^0 _{s}, \ s \ges t,\\
& X_t=x_{t}, \quad
\alpha_t=\iota,
\end{aligned}\right.
\end{equation}
and the cost functional is introduced as
\begin{equation}\label{equ-MFLQG-1}
\begin{array}{lll}
\ns\ds \BJ ^{\kappa}(t, x_{t},\iota, \BX, \Bu ; u_\cd)=
\dfrac{1}{2} \mathbb{E} \ds\int_{t}^{\infty} \mathrm e^{\kappa s}\[\langle Q(s, \alpha_{s}) X_s, X_s\rangle
+2\langle S(s, \alpha_{s}) X_s, u_{s}\rangle
+\langle R(s, \alpha_{s}) u_{s}, u_{s}\rangle\\
\ns\ds \hskip1.3cm
+2\langle \bar{Q}(s, \alpha_{s}) \BX_s, X_s\rangle
+2\langle \bar{S}_{1}(s, \alpha_{s})\BX_s, u_{s}\rangle
+2\langle \bar{S}_{2}(s, \alpha_{s}) X_s, \Bu_{s}\rangle
+2 \langle \bar{R}(s, \alpha_{s}) \Bu_{s}, u_{s}\rangle \\
\ns\ds \hskip1.3cm
+2\langle q(s, \alpha_{s}), X_s\rangle+2\langle \bar{q}(s, \alpha_{s}), \BX_s \rangle
+2\langle r(s, \alpha_{s}), u_s\rangle+2\langle \bar{r}(s, \alpha_{s}), \Bu_{s}\rangle\] \mathrm ds,
\end{array}
\end{equation}
subject to $u_\cd\in \cU_{ad}^\k[t,\i)$ for  some suitable $\kappa\in \mathbb{R}$.
The game problem is then formulated as follows.

\ms

$\textbf{Problem (G)$_S^\k$}$
For any $(t, x_{t}, \iota)\in \sD^2$ and  $(\BX , \Bu  )  \in L_{\mathbb{F}^0}^{2, \kappa}(t, \infty; \mathbb{R}^n)\times L_{\mathbb{F}^0}^{2, \kappa}(t, \infty; \mathbb{R}^m)$,  find an admissible control $\bar{ {u}} \in \mathcal{U}_{a d}^{\kappa}[t, \infty)$ such that
\begin{equation}\label{equ-MFLQG-u1}\ba{ll}
\ns\ds
\BJ ^{\kappa}(t, x_{t}, \iota, \BX, \Bu; \bar{u}_{\cdot})
=\inf_{u \in \mathcal{U}_{a d}^{\kappa}[t, \infty)} \BJ ^{\kappa}(t, x_{t} , \iota, \BX, \Bu ; u_{\cdot}),
\ea\end{equation}
and
\begin{equation}\label{equ-MFLQG-u2}\ba{ll}
\ns\ds   \mathbb{E}^0_{s}[\bar{ X }_s]= \BX_s, \
\mathbb{E}^0_{s}[\bar{ {u}}_s]=\Bu_s,
\ea\end{equation}
where   $\bar{ X } \equiv X(\cd;t, x_{t}, \iota,\bar{ {u}} )$  denotes the solution of \eqref{equ-MFLQG-2} corresponding to $\bar{ {u}} $.
Such a control  $\bar{{u}}_{\cd}$ is called a   {\it mean field Nash equilibrium} for Problem (G)$^\k_S$ at $(t, x_{t}, \iota)$.
In the special case where  $S(\cd,\cd)=\bar {S}_1(\cd,\cd)=\bar {S}_2(\cd,\cd)\equiv0$,  the problem is referred to as Problem (G)$^{\k}_0$.

Indeed,  \rf{equ-MFLQG-u1} defines a  mapping  $ \cT_1 :  L_{\mathbb{F}^0}^{2, \kappa}(t, \infty; \mathbb{R}^n \times \mathbb{R}^m) \longmapsto L_{\mathbb{F} }^{2, \kappa}(t, \infty; \mathbb{R}^n \times \mathbb{R}^m) $ such that
 $$ \cT_1(\BX,\Bu )=(\bar X ,\bar{{u}} ). $$
Now,   introduce another mapping
$$\cT_2   : L_{\mathbb{F} }^{2, \kappa}(t, \infty; \mathbb{R}^n \times \mathbb{R}^m) \longmapsto L_{\mathbb{F} ^0}^{2, \kappa}(t, \infty; \mathbb{R}^n \times \mathbb{R}^m), $$
  defined by  $ \cT_2 (X ,u )(s)=(\dbE_s^0[X_s],\dbE_s^0[u_s])$, $s\in [t,\i)$.
    Let  $\cT:=\cT_1\circ\cT_2$.
 From this  viewpoint, the objective  of Problem (G)$_S^\k$ is to find a pair  $(\bar X,\bar u)$ satisfying
  $$\cT(\bar X,\bar u)=(\bar X,\bar u),$$
  i.e.,  a fixed point  of $\cT$.

We now decompose Problem   (G)$_S^\k$ into two subproblems: the first is to   characterize the pair $(\bar X,\bar  u)$ satisfying \rf{equ-MFLQG-u1}, and the second is to    establish the existence and uniqueness of a fixed point of $\cT$.
Fix $(\BX , \Bu  )  \in L_{\mathbb{F}^0}^{2, \kappa}(t, \infty; \mathbb{R}^n)\times L_{\mathbb{F}^0}^{2, \kappa}(t, \infty; \mathbb{R}^m)$ be fixed.
We begin by addressing the first subproblem.

\ms

$\textbf{Problem (G1)$_S^\k$}$
For any $(t, x_{t}, \iota)\in \sD^2$, find an admissible control $u^\ast \in \mathcal{U}_{a d}^{\kappa}[t, \infty)$ with corresponding state process $\cX  $ such that
\begin{equation}\label{equ-MFLQG-3}
\BJ ^{\kappa}(t, x_{t}, \iota, \BX, \Bu; {u^{\ast}} )
=\inf_{u \in \mathcal{U}_{a d}^{\kappa}[t, \infty)} \BJ ^{\kappa}(t, x_{t} , \iota, \BX, \Bu; u ) .
\end{equation}
We note that Problem  (G1)$_S^\k$ is in fact a special case of  Problem (C)$^\k_S$.  Specifically, for a given pair $(\BX, \Bu)$, define

\bel{Bb}\ba{ll}
\ns\ds \Bb_s:=\bar{A} (s, \alpha_{s})\BX   _s+\bar{B}(s, \alpha_{s})\Bu_{s}+b_s,\\
\ns\ds \bm{\sigma}_s :=\bar{C}(s, \alpha_{s})\BX _s
+\bar{D}(s, \alpha_{s})\Bu_{s}+\sigma_s,\\
\ns\ds  \bm{\g}_s :=\bar{M}(s, \alpha_{s})\BX _s
+\bar{N}(s, \alpha_{s})\Bu_{s}+\gamma_s,\\
\ns\ds \Bq(s, \alpha_{s}):=\bar{Q}(s, \alpha_{s}) \BX_s + \bar{S}_{2}(s, \alpha_{s})^\top \Bu_{s} + q(s, \alpha_{s}),\\
\ns\ds \Br(s, \alpha_{s}):= \bar{S}_{1}(s, \alpha_{s})\BX_s+\bar{R}(s, \alpha_{s}) \Bu_{s}+r(s, \alpha_{s}),

\ea\ee
Then the state equation \rf{equ-MFLQG-2} and the cost  functional  \rf{equ-MFLQG-1} can be rewritten as
\begin{equation}\label{equ-MFLQG-state}
\left\{\2n
\begin{aligned}
& \mathrm dX_s=\(A (s, \alpha_{s}) X_s
+ B (s, \alpha_{s}) u_{s} +\Bb_s\) \, \mathrm ds  + \(C(s, \alpha_{s}) X_s
+D(s, \alpha_{s}) u_{s}+ \bm{\sigma}_s\) \, \mathrm dW_{s} \\
&\qq\q+ \(M(s, \alpha_{s}) X_s+N(s, \alpha_{s}) u_{s}+ \bm{\g}_s \) \, \mathrm d W^0 _{s}, \ s \ges t,\\
& X_t=x_{t}, \quad
\alpha_t=\iota,
\end{aligned}\right.
\end{equation}
and
\begin{equation}\label{equ-MFLQG-cost}
\begin{array}{lll}
\ns\ds \BJ ^{\kappa}(t, x_{t},\iota, \BX, \Bu ; u_\cd)=
\dfrac{1}{2} \mathbb{E} \ds\int_{t}^{\infty} \mathrm e^{\kappa s}\[\langle Q(s, \alpha_{s}) X_s, X_s\rangle
+2\langle S(s, \alpha_{s}) X_s, u_{s}\rangle
+\langle R(s, \alpha_{s}) u_{s}, u_{s}\rangle\!\!\!\\
\ns\ds  \qq\qq\qq\qq
+2\langle\Bq(s, \alpha_{s}), X_s\rangle
+2\langle  \Br(s, \alpha_{s}), u_s\rangle
 +2\langle \bar{q}(s, \alpha_{s}), \BX_s \rangle+2\langle \bar{r}(s, \alpha_{s}), \Bu_{s}\rangle\] \mathrm ds.\!\!\!
\end{array}
\end{equation}
%
%
%
%
%
For convenience, we define
\bel{kC}\left\{\2n\ba{ll}
 \ns \ds{\overline{\Bk}:=-{\k}_x},\\
\ns\ds
\underline{\Bk}:=2 \(\k_{y}+\max\limits_{  \imath\in\textbf{M}}\{\Vert \mathscr{Q} (\cdot, \imath)\Vert^{2}
+\Vert \mathscr{C}(\cdot, \imath)\Vert^{2}+\Vert \mathscr{M} (\cdot, \imath)\Vert^{2}\}\).

\ea\right.\ee
Based on the above observation, we apply Theorem \ref{Pro-MFLQ} to  derive the following result.

\begin{lemma}\label{Le-MFLQG-1}\sl
Suppose assumptions $\mathbf{(H1)}$-$\mathbf{(H3)}$ and $\mathbf{(PD)}$-(i) hold.  If   {$\ds \k=-\frac{{\k}_x}{2}+{\k}_{y} \in
(\underline{\Bk},\overline{\Bk})$ and $\kappa\les\kappa^\ast$},  and   for each  $\imath\in \mathbf{M}$,  $\bar{S}_1(\cd,\imath), \bar{S}_2(\cd,\imath)\in L^{\infty}_{\mathbb{F}^0}(t, \infty ;\mathbb{R}^{m \times n})$, then
for any given pair $(\BX, \mathbf{u})\in L_{\mathbb{F}^0}^{2, \kappa}(t, \infty; \mathbb{R}^n \times \mathbb{R}^m)$
and any initial triple $(t, x_{t}, \iota) \in \sD^2$,
the  following Hamiltonian  system   associated with  Problem (G1)$_S^\k$, given by
\begin{equation}\label{equ-MFLQG-Ham-1}
\left\{\2n
\begin{aligned}
& \mathrm d\cX_s=\(\sA\cX _s
   -B  R^{-1} \big(B^{\top} \cY _s +D^{\top} \cZ _s
+N^{\top}  \cZ ^0_{s}+ \Br\big)+\Bb_s   \) \, \mathrm ds  \\
&\qq\q + \(\sC \cX _s
-D R^{-1} \big(B^{\top} \cY _s +D^{\top} \cZ _s
+N^{\top}  \cZ ^0_{s} + \Br\big)+ \bm{\sigma} _s\) \, \mathrm dW_{s} \\
&\qq\q+ \(\sM \cX _s-NR^{-1} \big(B^{\top} \cY _s +D^{\top} \cZ _s
+N^{\top}  \cZ ^0_{s}+ \Br\big)+ \bm{\gamma} _s \) \, \mathrm d W^0 _{s}, \ s \ges t,\\
&\mathrm d\cY _{s}
= -\(\big(\kappa \mathrm I_n + \sA\big)^{\top} \cY  _{s}
+ \sC^{\top} \cZ _{s} + \sM^{\top} \cZ  ^0_{s}
+\sQ \bar{X } _{s}+\Bq-S^\top R^{-1}\Br  \) \, \mathrm ds\\
&\hskip1cm
+\cZ _{s}\, \mathrm dW_{s}
+ \cZ ^0_{s} \, \mathrm d W^0 _{s}
+ {\cK_{s} \bullet\mathrm d M_s}, \  s\ges t,\\
& \cX _t=x_{t},\
\alpha_t=\iota,
\end{aligned}\right.
\end{equation}
 admits a unique solution
$(\cX  , \cY ,\cZ ,\cZ ^0, \cK) \in  L_\dbF^2(t,\i;\dbR^n)\times\cL_{\mathbb{F}}^{2, {\k}}(t, \infty)$.
 In this case, the optimal control of Problem  (G1)$_S^\k$ is given by
\bel{G-opt-u}  u^\ast  
 =-R^{-1} \(S \cX  +B^{\top} \cY  +D^{\top} \cZ
+N^{\top}  \cZ ^0+  \bar{S}_{1}\BX
+ \bar{R}\Bu+r\) . \ee

\end{lemma}

Based on the notation introduced in  \rf{Bb}, the system \rf{equ-MFLQG-Ham-1} can be expressed as follows
\begin{equation}\label{equ-MFLQG-Ham}
\left\{\2n
\begin{aligned}
& \mathrm d\cX_s=\(\sA\cX _s
   -B  R^{-1} \big(B^{\top} \cY _s +D^{\top} \cZ _s
+N^{\top}  \cZ ^0_{s} \big) \\
&\qq\qq\q +   (\bar{A}-B  R^{-1}\bar{S}_{1})\BX_s+(\bar{B}
-B  R^{-1} \bar{R})\Bu_s-B  R^{-1}r   +b_s\) \, \mathrm ds  \\
&\qq\q + \(\sC \cX _s
-D R^{-1} \big(B^{\top} \cY _s +D^{\top} \cZ _s
+N^{\top}  \cZ ^0_{s} \big)\\
&\qq\qq\q +( \bar{C} -D R^{-1}   \bar{S}_{1})\BX_s
+(\bar{D} -D R^{-1}\bar{R})\Bu_s-D R^{-1}r
\ +\sigma_s\) \, \mathrm dW_{s} \\
&\qq\q+ \(\sM \cX _s-NR^{-1} \big(B^{\top} \cY _s +D^{\top} \cZ _s
+N^{\top}  \cZ ^0_{s}\big)\\
&\qq\qq\q + (\bar{M} -NR^{-1}  \bar{S}_{1})\BX_s
+(\bar{N} -NR^{-1} \bar{R})\Bu_s -NR^{-1} r    +\gamma_s \) \, \mathrm d W^0 _{s}, \ s \ges t,\\
&\mathrm d\cY _{s}
= -\(\big(\kappa \mathrm I_n + \sA\big)^{\top} \cY  _{s}
+ \sC^{\top} \cZ _{s} + \sM^{\top} \cZ  ^0_{s}
+\sQ \bar{X } _{s}\\
&\qq\qq\q
+ (\bar{Q} -S^\top R^{-1}  \bar{S}_{1} )\BX_s
+( \bar{S}_{2}^{\top} -S^\top R^{-1} \bar{R} ) \Bu_{s}+q-S^\top R^{-1} r \) \, \mathrm ds\\
&\hskip1cm
+\cZ _{s}\, \mathrm dW_{s}
+ \cZ ^0_{s} \, \mathrm d W^0 _{s}
+ {\cK_{s} \bullet\mathrm d M_s}, \  s\ges t,\\
& \cX _t=x_{t},\
\alpha_t=\iota.
\end{aligned}\right.
\end{equation}
From the  above, when {$\ds \k=-\frac{{\k}_x}{2}+{\k}_{y}
\in(\underline{\Bk},\overline{\Bk})$ and $\kappa\les\kappa^\ast$},
 the mapping $\cT_1$ is explicitly given by
 $$ \cT_1(\BX,\Bu )=(\cX ,u^\ast ), $$
where $ (\cX ,u^\ast)$ satisfies \rf{equ-MFLQG-Ham} and  \rf{G-opt-u}, resp.

 We now turn to the fixed point of $\cT$.
Obviously, a pair  $(\bar X,\bar u)$ is a fixed point of $\cT$  only and only if
  $$\cT_2(\bar X,\bar u)(s)=(\dbE^0_s[\bar X_s],\dbE^0_s[\bar u_s])\mbox{ and } \cT_1(\dbE^0_s [\bar X_s],\dbE^0_s[\bar u_s])=(\bar X_s,\bar u_s),\ s\in [t,\i).$$
From this observation, the existence and uniqueness of a fixed point of $\cT$ is equivalent to that of the following system:
 \begin{equation}\label{equ-cT-1}
\left\{\2n
\begin{aligned}
& \mathrm d\bar X_s=\(\sA\bar X_s
   -B  R^{-1} \big(B^{\top} \bar{Y}_s +D^{\top} \bar{Z}_s
+N^{\top}  \bar{Z}^0_{s} \big) \\
&\qq\qq +   (\bar{A}-B  R^{-1}\bar{S}_{1})\dbE_s^0[\bar X_s]+(\bar{B}
-B  R^{-1} \bar{R})\dbE_s^0[\bar u_s]-B  R^{-1}r   +b_s\) \, \mathrm ds  \\
&\qq + \(\sC \bar X_s
-D R^{-1} \big(B^{\top} \bar{Y}_s +D^{\top} \bar{Z}_s
+N^{\top}  \bar{Z}^0_{s} \big)\\
&\qq\qq+( \bar{C} -D R^{-1}   \bar{S}_{1})\dbE_s^0[\bar X_s]
+(\bar{D} -D R^{-1}\bar{R})\dbE_s^0[\bar u_s]-D R^{-1}r
\ +\sigma_s\) \, \mathrm dW_{s} \\
&\qq+ \(\sM \bar X_s-NR^{-1} \big(B^{\top} \bar{Y}_s +D^{\top} \bar{Z}_s
+N^{\top}  \bar{Z}^0_{s}\big)\\
&\qq\qq+ (\bar{M} -NR^{-1}  \bar{S}_{1})\dbE_s^0[\bar X_s]
+(\bar{N} -NR^{-1} \bar{R})\dbE_s^0[\bar u_s] -NR^{-1} r    +\gamma_s \) \, \mathrm d W^0 _{s}, \ s \ges t,\\
&\mathrm d\bar{Y}_{s}
= -\(\big(\kappa \mathrm I_n + \sA\big)^{\top} \bar{Y} _{s}
+ \sC^{\top} \bar{Z}_{s} + \sM^{\top} \bar{Z} ^0_{s}
+\sQ \bar{X } _{s}\\
&\qq\qq
+ (\bar{Q} -S^\top R^{-1}  \bar{S}_{1} )\dbE_s^0[\bar X_s]
+( \bar{S}_{2}^{\top} -S^\top R^{-1} \bar{R} )\dbE_s^0[\bar u_s]-S^\top R^{-1} r+q_s \) \, \mathrm ds\\
&\hskip0,7cm
+\bar{Z}_{s}\, \mathrm dW_{s}
+ \bar{Z}^0_{s} \, \mathrm d W^0 _{s}
+ {\bar{K}_{s} \bullet\mathrm d M_s}, \  s\ges t,\\
& \bar X_t=x_{t},\
\alpha_t=\iota,\\
\end{aligned}\right.
\end{equation}
where
 \bel{G-opt-u-3}\bar{ {u}} =-R^{-1} \(S\bar X +B^{\top} \bar{Y} +D^{\top} \bar{Z}
+N^{\top}  \bar{Z}^0 +\bar{S}_{1}\dbE_s^0[\bar X]
+ \bar{R}\dbE_s^0[\bar u]+r\).\ee

Indeed, under condition  $\mathbf{(PD)}$, the process $\bar u $ can be explicitly    expressed   as
\bel{equ-MFLG-u-opt}\ba{ll}
\ns\ds  \bar  u   = -R^{-1}\Big({S}(\bar X - \mathbb{E}^0_s[\bar X])    +B^{\top}(\bar Y - \mathbb{E}^0_s[\bar Y])  +  D^{\top}(\bar Z-\mathbb{E}^0_s[\bar Z])
 +  N^{\top}( \bar Z^0  - \mathbb{E}^0_s[ \bar Z^0])  + r - \mathbb{E}^0_s[r]\!\Big)\\
\ns\ds \qq -(R + \bar{R})^{-1}\Big( {(S+\bar{S}_1)\mathbb{E}^0_s[\bar X]}+B^{\top} \mathbb{E}^0_s[\bar Y] + D^{\top} \mathbb{E}^0_s[\bar Z]
 + N^{\top} \mathbb{E}^0_s[ \bar Z^0]  + \mathbb{E}^0_s[r]\Big).
 \ea\ee
 Substituting \eqref{equ-MFLG-u-opt} into
  \rf{equ-cT-1} and \rf{G-opt-u-3}, we obtain the following coupled
system:
 \begin{equation}\label{equ-MFLG-Hamiltonian}
\left\{\3n
\ba{ll}
 \ds \mathrm d\bar X _s=b(s, \bar X_s,\bar Y_s,\bar Z_s,\bar Z^0_s,\cL^1_{( \bar X_s,\bar Y_s,\bar Z_s,\bar Z^0_s)},\a_s) \mathrm ds+ \si(s, \bar X_s,\bar Y_s,\bar Z_s,\bar Z^0_s,\cL^1_{( \bar X_s,\bar Y_s,\bar Z_s,\bar Z^0_s)},\a_s)  \mathrm dW_s \!\!\!\!\!\!\!\!\! \!\!\\
\ns\ds \hskip1cm
+ \wt \si(s, \bar X_s,\bar Y_s,\bar Z_s,\bar Z^0_s,\cL^1_{( \bar X_s,\bar Y_s,\bar Z_s,\bar Z^0_s)},\a_s)\,  \mathrm dW^0_{s}, \ s\ges t,\\
\ns\ds  \mathrm d\bar Y _s =
-f(s, \bar X_s,\bar Y_s,\bar Z_s,\bar Z^0_s,\cL^1_{( \bar X_s,\bar Y_s,\bar Z_s,\bar Z^0_s)},\a_s)\, \mathrm ds
+ \bar Z_s \, \mathrm dW_{s}
+ \bar Z^0 _{s} \, \mathrm d W^0 _{s}
+ {\bar K _s \bullet\mathrm d M_s}, \  s\ges t, \\
\ns\ds  \bar X _{t}=x_{t}, \quad \alpha_{t}=\iota,
\ea\right.
\end{equation}
where, for   $(s,x,\th, \nu,\imath)\in [0,\i)\times {\dbR^n\times\sR }\times \sP_2(\dbR^n \times\sR )\times {\textbf M} $,
 \bel{b-si-f-G}\left\{\3n
\ba{ll}
 \ds b(s,x,\th, \nu,\imath):= b_s +\sA  x   + \tilde\sA   \int_{\dbR^n\times\sR } \tilde x \, \mathrm d\nu(\tilde x,\tilde \th)  +\L_  B (s, \th, \nu,\imath)    ,\\
 \ns \ds \si(s,x,\th, \nu,\imath):= \si_s +\sC  x  +\tilde\sC   \int_{\dbR^n\times\sR } \tilde x \, \mathrm d\nu(\tilde x,\tilde \th)  +\L_  D(s,  \th,\nu,\imath)    ,\\
 \ns  \ds \wt\si(s,x,\th, \nu,\imath):= \wt\si_s +\sM  x +\tilde\sM  \int_{\dbR^n\times\sR } \tilde x \, \mathrm d\nu(\tilde x,\tilde \th)  +\L_  N(s,  \th,\nu,\imath)    ,\\
 %


   \ns \ds   f(s,x,\th, \nu,\imath):= \sQ  x
+  \tilde{\sQ} \int_{\dbR^n\times\sR } \tilde x \, \mathrm d\nu(\tilde x,\tilde \th) +\big(\kappa \mathrm I_n + \sA \big)^{\top} y+\sC ^{\top} z + \sM^{\top} z^0
\\
\ns\ds \qq\qq\qq\q  + \int_{\dbR^n\times\sR }\big( \dbA^{\top}\tilde y+\dbC^{\top} \tilde z+\dbM ^{\top} \tilde {z}^0\big ) \mathrm d\nu(\tilde x,\tilde \th)\\
\ns\ds \qq\qq\qq\q{-( \bar{S}_{2}^{\top} -S^\top R^{-1} \bar{R} ) (R+ \bar{R})^{-1}\dbE_s^0[r]} +   q-S^\top R^{-1} r,
  \\
\ea
\right.\ee
 with
  $$\ba{ll}
 \ns\ds \tilde\sA:=\bar{A}+B  R^{-1} {S}-(B+\bar  B)  (R+\bar R)^{-1}(S+\bar{S}_{1}) , \q \
 \dbA:=B R^{-1} {S}-B(R+\bar R)^{-1}(S+\bar{S}_{2}) ,\\
 \ns\ds \tilde\sC:=\bar{C}+D  R^{-1} {S}-(D+\bar  D)  (R+\bar R)^{-1}(S+\bar{S}_{1}) ,\q \
  \dbC:=D R^{-1} {S}-D(R+\bar R)^{-1}(S+\bar{S}_{2}) ,\\
\ns\ds\tilde\sM:=\bar{M}+{N}  R^{-1} {S}-{(N+\bar N)}  (R+\bar R)^{-1}(S+\bar{S}_{1}),  \
  \dbM :=N R^{-1} {S}-N(R+\bar R)^{-1}(S+\bar{S}_{2}) ,\\
 \ns\ds \tilde\sQ:= \bar{Q}+S^\top   R^{-1} {S}-(S+\bar{S}_2)^{\top}(R+\bar{R})^{-1}(S+\bar{S}_1),\\
 \ea$$
 and
 $$ \ba{ll}
 \ns \ds \ \L_  \Phi(s,  \th,\nu,\imath) \\
  \ns\ds :=-\Phi R^{-1}\[B^{\top}y+D^{\top}z+N^{\top} z^0   -\int_{\dbR^n\times\sR}\big(B^\top\tilde y+D^\top \tilde z+N^\top \tilde {z}^0\big ) \mathrm d\nu(\tilde x,\tilde \th)+r -\dbE_s^0[r]\] \\
\ns \ds\q      -(\Phi+\bar{\Phi})(R+\bar{R})^{-1}\[\int_{\dbR^n\times\sR}\big(B^{\top}\tilde y+D^{\top} \tilde z + N^{\top}  \tilde{z}^0 \big)\mathrm d\nu(\tilde x,\tilde \th) +\dbE_s^0[r]  \]   ,\q \Phi=B,D, N, \mbox{ resp.}
\ea
$$

We note that  \rf{equ-MFLG-Hamiltonian} remains  the  fully coupled FBSDEs. This naturally leads to the question of whether Theorem   \ref{Th-FBSDE} can be applied to  establish the {well-posedness} of \rf{equ-MFLG-Hamiltonian}. The careful analysis
shows that Theorem \ref{Th-FBSDE} is  indeed applicable to certain special cases of \rf{equ-MFLG-Hamiltonian}. We now examine the following case.
 \begin{description}
   \item[(S1)]
 (i) For every $\imath\in \mathbf{M}$, $\bar{A}(\cd,\imath), \bar{C}(\cd,\imath), \bar{M}(\cd,\imath)\equiv0$.

(ii) There exists a  constant  $k \ges-1 $ such that for all $\imath\in \mathbf{M}$,   $\bar{B}(\cd,\imath)=kB(\cd,\imath),$ $\bar{D}(\cd,\imath)=kD(\cd,\imath), $ $\bar{N}(\cd,\imath)=kN(\cd,\imath).$

   \item[(S2)]
 (i)  For each  $\imath\in \mathbf{M}$,  the processes  $\bar{S}_1(\cd,\imath)$ and $ \bar{S}_2(\cd,\imath)$ belongs to $L^{\infty}_{\mathbb{F}^0}(t, \infty ;\mathbb{R}^{m \times n})$.

 (ii)  For each  $\imath\in \mathbf{M}$, $kS(\cd,\imath)+(k+1)\bar{S}_1(\cd,\imath)-\bar{S}_2(\cd,\imath)=0$, where   $k$ is the same constant  as  $k$ in $\mathbf{(S1)}$-(ii).

   \item[(S3)]
There exist  constants  $  \Bbbk _x, \Bbbk _{x\mu}  ,\Bbbk _y, \Bbbk _{y\nu}  \in \mathbb{R}$ such that, for all $(s, \omega^0, \imath)\in [0, \infty)\times \Omega^0\times\mathbf{M}$,
$$
\left\{\2n
\ba{ll}
\ns\ds  \Bbbk  _x \ges  \mathop{\esssup}\limits_{(s, \omega^0, \imath) } \lambda_{\max }\big(\sA+\sA^{\top} + \sC^{\top}\sC+\sM^{\top}\sM\big)(s, \omega^0, \imath),\\
\ns\ds \Bbbk _{x\mu}  \ges  \mathop{\esssup}\limits_{(s, \omega^0, \imath) }
 \lambda_{\max }\big(\tilde\sA  +  \tilde\sA   \, ^\top +  (\sC+ \tilde\sC)  ^\top(\sC+ \tilde\sC)    -\sC^\top \sC\\
\ns\ds  \hskip 3.5cm +  (\sM+ \tilde\sM)  ^\top(\sM+ \tilde\sM)   -\sM^\top \sM \big)(s,\omega^0,\imath),\\
\ns\ds \Bbbk_{y} \ges \kappa+\frac{1}{2}\mathop{\esssup}\limits_{(s, \omega^0, \imath)  }
 \lambda_{\max }\big(\sA+\sA^{\top}\big)(s, \omega^0, \imath),\\
\ns\ds  \Bbbk_{y\nu} \ges \frac{1}{2} \mathop{\esssup}\limits_{(s, \omega^0, \imath)  }
 \lambda_{\max }\big(\dbA +\dbA^{\top}\big)(s, \omega^0, \imath).
\ea\right.
$$

   \item[(PD)$_G$]   For all $\imath \in\textbf{M}$,
$$
\ba{ll}
\ns\ds {\rm(i)}\q R (\cdot, \imath) \gg 0,\q
\left(\begin{array}{cc}
\3n Q (\cdot, \imath) & \3n S(\cdot, \imath)^{\top} \3n \\
\3n S(\cdot, \imath) & \3n R(\cdot, \imath)\3n
\end{array}\right) \ges 0, \\
\ns\ds   {\rm(ii)}\q   (R+\bar R) (\cdot, \imath) \gg 0,\q  \left(\begin{array}{cc}
\3n (Q +\bar Q)(\cdot, \imath) & \3n (S+\bar{S}_2) (\cdot, \imath)^{\top}\2n \\
\3n (S+\bar{S}_1)(\cdot, \imath) & \3n (R+\bar R)(\cdot, \imath)\3n
\end{array}\right) \ges 0.
 \ea $$
 \end{description}

Moreover, we put
\bel{kG}\left\{\3n\ba{ll}
 \ns \ds{\overline{\Bbbk}:=-\Bbbk_x-\Bbbk_{x\mu}
 \mathbf{1}_{\{\Bbbk_{x\mu}>0\}}},\\
\ns\ds
\underline{\Bbbk}:=2 (\Bbbk _{y}+ \Bbbk _{y \nu}\mathbf{1}_{\{\Bbbk _{y \nu}>0\}}+\dbK),\\
\ns\ds \dbK:= \max\limits_{  \imath\in\textbf{M}}\Big\{\Vert \mathscr{Q} (\cdot, \imath)\Vert^{2}
+\Vert \mathscr{C}(\cdot, \imath)\Vert^{2}+\Vert \mathscr{M} (\cdot, \imath)\Vert^{2}
+\max\{\Vert   \tilde{\mathscr{Q}} (\cdot, \imath)\Vert^{2}, \Vert \dbC (\cdot, \imath)\Vert^{2},
\Vert \dbM  (\cdot, \imath)\Vert^{2}\}\Big\}.\!\!\!\!\!\!\!\!

\ea\right.\ee

\begin{theorem}\label{Pro-MFLG}\sl
Assume that $\mathbf{(H1)}$-$\mathbf{(H2)}$, $\mathbf{(S1)}$-$\mathbf{(S3)}$, and $(\mathbf{PD})_G$ hold.
If {$\ds \Bbbk= -\frac{\Bbbk _x}{2}+\Bbbk _{y}\in
(\underline{\Bbbk },\overline{\Bbbk })$ and
$\Bbbk\les\kappa^\ast$},  then  for any $(t, x_t, \iota) \in \sD^2$, system \eqref{equ-MFLG-Hamiltonian} admits   a unique solution $(\bar{X}, \bar{Y}, \bar{Z}, \bar{Z}^0, \bar{K} ) \in L_\dbF^{2, {\Bbbk}}(t,\i;\dbR^n)\times\cL_{\mathbb{F}}^{2, {\Bbbk}}(t, \infty)$.
Moreover,
%
 the process $\bar u(\cd)$ in \rf{equ-MFLG-u-opt} belongs to $\mathcal{U}_{ad}^{{\Bbbk}}[t, \infty)$ and is the unique mean field Nash equilibrium of Problem (G)$^\Bbbk_S$ at $(t, x_t, \iota)$.

\end{theorem}


\begin{proof}
Observe that \rf{equ-MFLG-Hamiltonian} is   a linear case of the general  system \rf{equ-FBSDE-1}. Hence, it suffices to verify that the assumptions of Theorem \ref{Th-FBSDE}  are satisfied. To this end, for $(s,\th,\bar\th,\nu,\bar \nu, \imath )\in[0,\infty)\times\mathscr{R}\times\mathscr{R}\times{\mathscr{P}_2(\mathscr{R})\times\mathscr{P}_2(\mathscr{R})}\times \textbf{M}$
and $  \upsilon,\bar \upsilon \in  \mathscr{P}_2(\mathbb{R}^n) $, define
\begin{equation}\label{equ-MFLQ-b1b2}
\begin{aligned}
&\b_1= \min\Big\{ \frac1{L_1}, L_2\Big\},\q \b_2=0,\q
\phi(y, \bar{y}, \upsilon, \bar{\upsilon}, \imath)\equiv0, \q
\psi(s,\th,\bar\th,\nu,\bar \nu, \imath ) \equiv0,\\
& \varphi(s,\th,\bar\th,\nu,\bar \nu, \imath ) :=\big|\rho_1(s,\th-\bar \th,\imath)   \big|^2{{\bf 1}_{\{k>-1\}}+\big|\rho_2(s,\th,\nu,\imath)-\rho_2(s,\bar\th,\bar\nu,\imath)\big|^2{\bf 1}_{\{k=-1\}}},
\end{aligned}
\end{equation}
  where
$$
\begin{aligned}
&\rho_1(s,\th,\imath):=B^\top y+D^\top z+N ^\top z^0, \\
&{\rho_2(s,\th,\nu, \imath):=\rho_1(s,\th,\imath)-\int_{\sR} \rho_1(s,\tilde\th,\imath) \mathrm d\nu(\tilde \th),} \\
&L_{1}:=2\max\Big\{ \max \limits_{ \imath\in\textbf M }\|B R^{-1}(\cdot, \imath) \|^2,\ \max \limits_{ \imath\in\textbf M } \|(B+\bar B) (R+\bar R)^{-1} (\cdot, \imath)\|^2,\ \max \limits_{ \imath\in\textbf M }\|D R^{-1}(\cdot, \imath) \|^2, \\
&\qq\q \max \limits_{ \imath\in\textbf M }\|(D\!+\!\bar D) (R\!+\!\bar R)^{-1}(\cdot, \imath) \|^2, \ \max \limits_{ \imath\in\textbf M }\|N R^{-1}(\cdot, \imath) \|^2, \ \max \limits_{ \imath\in\textbf M }\|(N\!+\!\bar N) (R\!+\!\bar R)^{-1}(\cdot, \imath) \|^2\Big\}{{\bf 1}_{\{k>-1\}}}\\
&\qq + {\max\Big\{ \max \limits_{ \imath\in\textbf M }\|B R^{-1}(\cdot, \imath) \|^2,\ \max \limits_{ \imath\in\textbf M }\|D R^{-1}(\cdot, \imath) \|^2,
 \ \max \limits_{ \imath\in\textbf M }\|N R^{-1}(\cdot, \imath) \|^2 \Big\}{\bf 1}_{\{k=-1\}}}\\
&L_{2}:= \min\Big\{\mathop{\essinf}\limits_{(s, \omega^{0}, \imath)}\lambda_{\min } R^{-1} (s, \imath ), \ (k+1)
\mathop{\essinf}\limits_{(s, \omega^{0}, \imath)}\lambda_{\min } (R+\bar R)^{-1}(s,\imath)\Big\}{\bf 1}_{\{k>-1\}}\\
 &\qq +{\mathop{\essinf}\limits_{(s, \omega^{0}, \imath)}\lambda_{\min }} R^{-1} (s, \imath ) {\bf 1}_{\{k=-1\}}.
\end{aligned}
$$

We now verify that the coefficients  $b$, $\si$, $\wt\si$ and $f$   given in \eqref{b-si-f-G}   satisfy  $\mathbf{(C3.1)}$.
First, consider {\textbf{\rm(C3)-(ii)}} for $b$:
%
%
$$
\ba{ll}
\mathbb{E}\left[\big|b(s, X,\Theta, \mathcal{L}^1_{(X, \Th )},\alpha_s)
 -b(s, X, \bar{\Th},   \mathcal{L}^1_{(X, \bar{\Th} )},\alpha_s)\big|^{2}\right]\\
 %
\ns\ds \les L_1 \mathbb{E}^0\bigg[\mathbb{E}^0_s\[\big|B^{\top}(\widehat{Y}-\mathbb{E}^0_s[\widehat{Y}])
+D^{\top}(\widehat{Z}-\mathbb{E}^0_s[\widehat{Z}])+N^{\top}(\widehat{Z}^0-\mathbb{E}^0_s[ \widehat{Z}^0])\big|^{2}\\
\ns\ds\qq\qq\q+\big|B^{\top}\mathbb{E}^0_s[\widehat{Y}]+D^{\top}\mathbb{E}^0_s[\widehat{Z}]
+N^{\top}\mathbb{E}^0_s[\widehat{Z}^0]\big|^{2}\]{{\bf 1}_{\{k>-1\}}} \\
\ns\ds\qq\q+{\mathbb{E}^0_s\[\big|B^{\top}(\widehat{Y}-\mathbb{E}^0_s[\widehat{Y}])
+D^{\top}(\widehat{Z}-\mathbb{E}^0_s[\widehat{Z}])+N^{\top}(\widehat{Z}^0-\mathbb{E}^0_s[ \widehat{Z}^0])\big|^{2}\]{\bf 1}_{\{k=-1\}}}\bigg]\\
\ns\ds = L_1 \mathbb{E}\[\varphi(s, \Theta, \bar{\Theta}, \mathcal{L}^1_{\Theta}, \mathcal{L}^1_{\bar{\Theta}},\alpha_s)\].
\ea$$
A similar argument applies to   $\si$ and $\tilde\si$.

 \ms
Next,  we verify {\textbf{\rm(C3)-(iii)-Case 1}}.
Under   (S1), (S2) and (PD)$_G$, we have    ${\tilde\sA=\dbA,  \tilde\sC=\dbC, \tilde\sM=\dbM}$, and
$$Q-S^{\top}R^{-1}S\ges0, \q Q+\bar{Q}-(S+\bar{S}_2)^{\top}(R+\bar{R})^{-1}(S+\bar{S}_1)\ges0.$$
A direct computation then shows that
$$
\ba{ll}
\!\mathbb{E}\left[ \left\langle\!\!  \Bigg(\begin{array}{ccc}
  \!\! -f(s, X,\Theta, \mathcal{L}^1_{(X, \Th )},\alpha_s)+f(s, \bar{X}, \bar{\Th},   \mathcal{L}^1_{(\bar{X}, \bar{\Th} )},\alpha_s)\!\! \\
  \!\! \Gamma(s, X,\Theta, \mathcal{L}^1_{(X, \Th )},\alpha_s) -\G(s, \bar{X}, \bar{\Th},   \mathcal{L}^1_{(\bar{X}, \bar{\Th} )},\alpha_s)\!\!\!
  \end{array} \Bigg)
 , \Bigg( \begin{array}{ccc}
 \!\!\!X-\bar{X} \!\! \!\\
 \!\!\!\Th-\bar{\Th}  \!\!\!   \\
   \end{array} \Bigg)\!\!
  \right\rangle\right]\\
\ns\ds+\({-\frac{\Bbbk_x}{2}+\Bbbk_y}\) \mathbb{E}\left[ \langle X-\bar{X}, Y-\bar{Y} \rangle \right]\\
\ea$$
$$
\ba{ll}
\ns\ds= \mathbb{E}^0\bigg[\mathbb{E}^0_s\[  -\big\langle (Q-S^{\top}R^{-1}S)(\widehat{X}-\mathbb{E}^0_s[\widehat{X}])
+[Q+\bar{Q}-(S+\bar{S}_2)^{\top}(R+\bar{R})^{-1}(S+\bar{S}_1)]\mathbb{E}^0_s[\widehat{X}],\widehat{X} \big\rangle\\
\ns\ds\qq +\big\langle {(\tilde\sA-\dbA)}\mathbb{E}^0_s[\widehat{X}], \widehat{Y}\big\rangle  +\big\langle{(\tilde\sC-\dbC)}\mathbb{E}^0_s[\widehat{X}], \widehat{Z}\big\rangle +\big\langle{(\tilde\sM-\dbM)}\mathbb{E}^0_s[\widehat{X}], \widehat{Z}^0\big\rangle\\
\ns\ds\qq -\big\langle R^{-1}\big(B^{\top}\widehat{Y}+D^{\top}\widehat{Z}+N^{\top}\widehat{Z}^0\big),
B^{\top}\widehat{Y}+D^{\top}\widehat{Z}+N^{\top}\widehat{Z}^0\big\rangle\\
\ns\ds\qq +\big\langle R^{-1}\big(B^{\top}\mathbb{E}^0_s[\widehat{Y}]+D^{\top}\mathbb{E}^0_s[\widehat{Z}]
+N^{\top}\mathbb{E}^0_s[ \widehat{Z}^0]\big), B^{\top}\mathbb{E}^0_s[\widehat{Y}]+D^{\top}\mathbb{E}^0_s[\widehat{Z}]
+N^{\top}\mathbb{E}^0_s[ \widehat{Z}^0]\big\rangle\\
\ns\ds\qq -\big\langle (R+\bar{R})^{-1}\big(B^{\top}\mathbb{E}^0_s[\widehat{Y}]
+D^{\top}\mathbb{E}^0_s[\widehat{Z}]+N^{\top}\mathbb{E}^0_s[ \widehat{Z}^0]\big),\\
\ns\ds\qq\qq\qq\qq \qq\qq
(B+ \bar{B} )^{\top}\mathbb{E}^0_s[\widehat{Y}]
+(D+ \bar{D} )^{\top}\mathbb{E}^0_s[\widehat{Z}]+(N+ \bar{N})^{\top}\mathbb{E}^0_s[ \widehat{Z}^0]\big\rangle \] \bigg]\\
 \ns\ds\les  \mathbb{E}^0\bigg[\mathbb{E}^0_s\[  -\big\langle R^{-1}\big(B^{\top}\widehat{Y}+D^{\top}\widehat{Z}+N^{\top}\widehat{Z}^0\big),
B^{\top}\widehat{Y}+D^{\top}\widehat{Z}+N^{\top}\widehat{Z}^0\big\rangle\\
\ns\ds\qq +\big\langle R^{-1}\big(B^{\top}\mathbb{E}^0_s[\widehat{Y}]+D^{\top}\mathbb{E}^0_s[\widehat{Z}]
+N^{\top}\mathbb{E}^0_s[ \widehat{Z}^0]\big), B^{\top}\mathbb{E}^0_s[\widehat{Y}]+D^{\top}\mathbb{E}^0_s[\widehat{Z}]
+N^{\top}\mathbb{E}^0_s[ \widehat{Z}^0]\big\rangle\\
\ns\ds\qq -(k+1)\big\langle (R+\bar{R})^{-1}\big(B^{\top}\mathbb{E}^0_s[\widehat{Y}]
+D^{\top}\mathbb{E}^0_s[\widehat{Z}]+N^{\top}\mathbb{E}^0_s[ \widehat{Z}^0]\big),
 B ^{\top}\mathbb{E}^0_s[\widehat{Y}]
+D^{\top}\mathbb{E}^0_s[\widehat{Z}]+N^{\top}\mathbb{E}^0_s[ \widehat{Z}^0]\big\rangle \] \bigg]\\
 %
 \ns\ds= \mathbb{E}^0\bigg[\mathbb{E}^0_s\[ - \big\langle R^{-1}\big(B^{\top}(\widehat{Y}-\mathbb{E}^0_s[\widehat{Y}])+D^{\top}(\widehat{Z}-\mathbb{E}^0_s[\widehat{Z}])
+N^{\top}(\widehat{Z}^0-\mathbb{E}^0_s[ \widehat{Z}^0])\big),\\
\ns\ds\qq\qq\qq\qq\qq\qq
B^{\top}(\widehat{Y}-\mathbb{E}^0_s[\widehat{Y}])+D^{\top}(\widehat{Z}-\mathbb{E}^0_s[\widehat{Z}]) +N^{\top}(\widehat{Z}^0-\mathbb{E}^0_s[ \widehat{Z}^0]) \big\rangle\\
\ns\ds\qq -(k+1)\big\langle (R+\bar{R})^{-1}\big(B^{\top}\mathbb{E}^0_s[\widehat{Y}]
+D^{\top}\mathbb{E}^0_s[\widehat{Z}]+N^{\top}\mathbb{E}^0_s[ \widehat{Z}^0]\big),
 B ^{\top}\mathbb{E}^0_s[\widehat{Y}]
+D^{\top}\mathbb{E}^0_s[\widehat{Z}]+N^{\top}\mathbb{E}^0_s[ \widehat{Z}^0]\big\rangle \] \bigg]\\
\ns\ds\les -L_2\mathbb{E}^0\bigg[\mathbb{E}^0_s\[\big| B^{\top}(\widehat{Y}-\mathbb{E}^0_s[\widehat{Y}])+D^{\top}(\widehat{Z}-\mathbb{E}^0_s[\widehat{Z}])
+N^{\top}(\widehat{Z}^0-\mathbb{E}^0_s[ \widehat{Z}^0])\big|^2\\
\ns\ds\qq\qq\qq  +\big|B^{\top}\mathbb{E}^0_s[\widehat{Y}]+D^{\top}\mathbb{E}^0_s[\widehat{Z}]+N^{\top}\mathbb{E}^0_s[ \widehat{Z}^0]\big|^2 \]{{\bf 1}_{\{k>-1\}}}\\
\ns\ds\qq\qq{+ \mathbb{E}^0_s\[\big| B^{\top}(\widehat{Y}-\mathbb{E}^0_s[\widehat{Y}])+D^{\top}(\widehat{Z}-\mathbb{E}^0_s[\widehat{Z}])
+N^{\top}(\widehat{Z}^0-\mathbb{E}^0_s[ \widehat{Z}^0])\big|^2 \]{\bf 1}_{\{k=-1\}}   \bigg]}\\
\ns\ds= -L_2\mathbb{E}\big[\varphi(s, \Theta, \bar{\Theta}, \mathcal{L}^1_{\Theta}, \mathcal{L}^1_{\bar{\Theta}},\alpha_s)\big].
\ea$$
Then, when   when {$\ds \Bbbk= -\frac{\Bbbk _x}{2}+\Bbbk _{y}
\in(\underline{\Bbbk },\overline{\Bbbk })$ and
$\Bbbk\les\kappa^\ast$},  the unique solvability of the McKean-Vlasov FBSDEs \rf{equ-MFLG-Hamiltonian}  follows directly from Theorem  \ref{Th-FBSDE}.
As a consequence, we obtain the unique mean field Nash equilibrium  $\bar{u}_\cd\in \mathcal{U}_{ad}^{{\Bbbk}}[t, \infty)$ of Problem (G)$^{{\Bbbk}}_S$   at $(t, x_t, \iota)$.
%
\end{proof}

\br{}\sl   It is worth noting that our conditions (S1) and (S2) encompass the case where
$$\bar{B}(\cdot, \imath), \bar{D}(\cdot, \imath), \bar{N}(\cdot, \imath)=0,\q  \bar{S}_1(\cdot, \imath)=\bar{S}_2(\cdot, \imath)$$
for each $\imath \in \mathbf{M},$
a setting previously studied in  \cite{Ahuja-Ren-Yang-2019, Graber-2016}, among others.
\er

We close this section with a brief discussion on the role of Markovian
switching in the above LQ control and game problems. These problems provide an
application-oriented illustration of the present framework. {In the present
formulation, $\alpha$ is a common regime process and is included in the common
filtration $\mathbb F^0$. Consequently, the conditional laws and conditional
mean-field terms are defined with respect to the common information generated
jointly by $W^0$ and $\alpha$.}

 Although finite-state Markov chains and their compensated martingales are
classical objects, their presence changes the structure of the conditional
McKean\textcolor{red}{-}Vlasov FBSDEs and the associated LQ control and game
problems. On the one hand, the switching process selects the active
coefficients in the state equation, the cost functional, and the Hamiltonian
system according to the current regime. {On the other hand, the compensated
martingales associated with $\alpha$ appear as common martingale terms in the
backward equations.} Hence, the open-loop optimal control in
\eqref{equ-MFLQ-19} and the Nash equilibrium in
\eqref{equ-MFLG-u-opt} are regime-dependent.  In particular, even after fixing
the state, the adjoint variables, and the conditional mean-field terms, the
corresponding decision rules still depend explicitly on the current regime
$\alpha_s$ through the regime-dependent coefficients.  Thus, the optimal
response may change when the Markov chain switches from one regime to another.
This shows that Markovian switching enters both the common information
structure and the characterization of the optimal decision rules in an
essential way.

 This observation also suggests several directions for future research. The
present work focuses on open-loop characterizations via conditional
McKean\textcolor{red}{-}Vlasov FBSDEs with Markovian switching. It would be
interesting to develop a closed-loop or feedback theory for the corresponding
infinite horizon mean field control and game problems. Such a theory may
involve regime-dependent feedback laws, coupled Riccati-type equations, and
closed-loop Nash equilibria under a common filtration generated jointly by the
common Brownian noise and the common regime process. These questions are left
for future research.

{\section*{Acknowledgements}
The authors are grateful to the anonymous reviewers for their careful reading
of the manuscript and for their valuable comments and constructive suggestions,
which have significantly improved the manuscript.}

\end{document}